\documentclass[11pt,oneside,reqno]{amsart}
\usepackage[latin9]{inputenc}
\usepackage{geometry}
\usepackage{mathrsfs}
\usepackage{amstext}
\usepackage{amsthm}
\usepackage{amssymb}
\usepackage{graphicx}
\usepackage{setspace}
\usepackage[unicode=true,
 bookmarks=false,
 breaklinks=false,pdfborder={0 0 1},backref=false,colorlinks=false]
 {hyperref}

\makeatletter
\numberwithin{equation}{section}
\numberwithin{figure}{section}
\theoremstyle{plain}
\newtheorem{thm}{\protect\theoremname}[section]
\theoremstyle{plain}
\newtheorem{prop}[thm]{\protect\propositionname}
\theoremstyle{remark}
\newtheorem{rem}[thm]{\protect\remarkname}
\theoremstyle{definition}
\newtheorem{condition}[thm]{\protect\conditionname}
\theoremstyle{definition}
\newtheorem{defn}[thm]{\protect\definitionname}
\theoremstyle{plain}
\newtheorem{lem}[thm]{\protect\lemmaname}
\theoremstyle{remark}
\newtheorem*{acknowledgement*}{\protect\acknowledgementname}

\usepackage{amsthm}
\usepackage{mathrsfs}
\usepackage[noadjust]{cite}
\usepackage{stmaryrd}
\usepackage{enumitem}
\setlist[itemize]{leftmargin=*}
\setlist[enumerate]{leftmargin=*}
\usepackage{bbm}

\DeclareFontFamily{U}{matha}{\hyphenchar\font45}
\DeclareFontShape{U}{matha}{m}{n}{
      <5> <6> <7> <8> <9> <10> gen * matha
      <10.95> matha10 <12> <14.4> <17.28> <20.74> <24.88> matha12
      }{}
\DeclareSymbolFont{matha}{U}{matha}{m}{n}
\DeclareFontSubstitution{U}{matha}{m}{n}

\DeclareFontFamily{U}{mathx}{\hyphenchar\font45}
\DeclareFontShape{U}{mathx}{m}{n}{
      <5> <6> <7> <8> <9> <10>
      <10.95> <12> <14.4> <17.28> <20.74> <24.88>
      mathx10
      }{}
\DeclareSymbolFont{mathx}{U}{mathx}{m}{n}
\DeclareFontSubstitution{U}{mathx}{m}{n}

\DeclareMathDelimiter{\vvvert}{0}{matha}{"7E}{mathx}{"17}

\DeclareMathAlphabet{\scal}{U}{dutchcal}{m}{n}

\def\th@plain{\thm@notefont{}\itshape}
\def\th@definition{\thm@notefont{}\normalfont}

\makeatother

\providecommand{\acknowledgementname}{Acknowledgement}
\providecommand{\conditionname}{Condition}
\providecommand{\definitionname}{Definition}
\providecommand{\lemmaname}{Lemma}
\providecommand{\propositionname}{Proposition}
\providecommand{\remarkname}{Remark}
\providecommand{\theoremname}{Theorem}

\begin{document}
\title{Second Time Scale of the Metastability of Reversible Inclusion Processes}
\author{Seonwoo Kim}
\address{S. Kim. Department of Mathematical Sciences, Seoul National University,
Republic of Korea.}
\email{ksw6leta@snu.ac.kr}
\begin{abstract}
We investigate the \emph{second time scale} of the metastable behavior
of the reversible inclusion process in an extension of the study by
{[}Bianchi, Dommers, and Giardin\`{a}, \textit{Electronic Journal
of Probability}, 22:1--34, 2017{]}, which presented the first time
scale of the same model and conjectured the scheme of multiple time
scales. We show that $N/d_{N}^{2}$ is indeed the correct second time
scale for the most general class of reversible inclusion processes,
and thus prove the first conjecture of the foresaid study. Here, $N$
denotes the number of particles, and $d_{N}$ denotes the small scale
of randomness of the system. The main obstacles of this research arise
in \emph{calculating the sharp asymptotics for the capacities}, and
in the fact that the methods employed in the former study are not
directly applicable due to the complex geometry of particle configurations.
To overcome these problems, we first \emph{thoroughly examine the
landscape of the transition rates} to obtain a proper test function
of the equilibrium potential, which provides the upper bound for the
capacities. Then, we \emph{modify the induced test flow and precisely
estimate the equilibrium potential near the metastable valleys} to
obtain the correct lower bound for the capacities. 
\end{abstract}

\keywords{Metastability, multiple time scales, interacting particle systems,
inclusion process.}
\maketitle

\section{Introduction}

An interacting particle system was introduced in \cite{GKR,GKRV}
as a dual process of a certain class of energy diffusion models, known
as Brownian momentum (energy) processes. In \cite{GRV}, this process
was first named as a \emph{(symmetric) inclusion process}, which was
treated as a bosonic\footnote{Bosonic particle systems represent dynamics in which particles tend
to attract each other. They are mostly used to represent dynamical
systems in low temperatures.} counterpart of the well-known exclusion process. Since then, this
particular random system has gathered the interest of numerous researchers.
A general overview on the study of inclusion processes is provided
in \cite[Chapters 2 and 6]{Chleboun}.

\emph{Condensation} takes place in various particle systems that exhibit
attractive interactions. It is defined by the situation in which a
significant portion of particles in the system becomes concentrated
at a single site (cf. \eqref{condensation}), due to the bosonic interactions
among them. This phenomenon has been a consistently popular research
subject during the past few decades. Condensation of inclusion processes
was first studied in \cite{GroReVa 11}, where the authors presented
the unique invariant measure of the system under some restrictions,
together with the condensation result of particles in the dynamics.
Since then, a variety of results were reported on condensation of
inclusion processes under various conditions and geometries of the
system \cite{ACR,Chleboun,KR,OpokuRedig}.

\emph{Metastability} represents the macroscopic phenomenon that occurs
when certain observables in a system linger in one state for an extended
period of time and at a random moment later evolve to another state
within a relatively short time. In the context of particle systems,
this is described as follows: After condensation occurs, the \emph{condensate}
of particles remains at its site for a relatively long time. However,
on appropriately long time scales, it tends to move to another site
within the system. This behavior, also referred to as \emph{tunneling},
can be characterized by a suitable random walk of the condensate on
the collection of sites. In the context of inclusion processes, this
phenomenon was first characterized in \cite{GroReVa 13}, where the
authors showed the asymptotic behavior of formation and evolution
of the condensate. However, this striking result was obtained only
for \emph{symmetric} inclusion processes. Accordingly, the next objective
was to find a similar result for a more general class of inclusion
processes. This project has been steadily maturing over the past few
years. \cite{BDG} reported a result on the metastable behavior for
reversible inclusion processes. Moreover, in \cite{KS NRIP}, Seo
and the author of the current paper worked on the setting of general
non-reversible inclusion processes.

The metastable behavior of reversible inclusion processes is subjected
to the scheme of \emph{multiple time scales}, which was studied thoroughly
in \cite{BL rev-MP}. For completeness, we briefly recall the result
from \cite[Theorem 2.3]{BDG}, stating that metastable behavior exists
among certain sites $S_{\star}$ (cf. Proposition \ref{conden}) in
the first time scale $1/d_{N}$, where $d_{N}$ denotes the control
factor of randomness of the dynamics which vanishes as the number
of particles $N$ tends to infinity. We must emphasize that the limiting
metastable dynamics on $S_{\star}$ \emph{may not be irreducible},
in contrast to the original underlying random walk. Because the original
process is irreducible, it is expected that all metastable states
are eventually achievable. Hence, \emph{the system is likely to exhibit
completely novel metastable movements at longer time scales.} The
authors of \cite{BDG} conjectured two additional time scales, $N/d_{N}^{2}$
and $N^{2}/d_{N}^{3}$, by proving the existence of such time scales
in a simple one-dimensional setting. Moreover, they showed that $N^{2}/d_{N}^{3}$
represents the terminal level of metastability, in the sense that
there are no time scales larger than $N^{2}/d_{N}^{3}$ in which metastable
movements occur.

\emph{In this study, we extensively generalize the metastable result
of reversible inclusion processes in \cite{BDG}, and }\textbf{\emph{we
fully characterize the metastable behavior in the second time scale,
$\theta_{N,2}=N/d_{N}^{2}$.}} Specifically, we prove that the conjectured
second time scale $N/d_{N}^{2}$ is indeed the correct one for the
most general class of reversible underlying random walks, and that
there are no intermediate time scales between $1/d_{N}$ and $N/d_{N}^{2}$.
This leads us to a \textbf{\emph{complete analysis of the metastability
of reversible inclusion processes up to the second level on the scheme
of multiple time scales.}}

In this research, we encountered \emph{two mathematical obstacles}.
The first obstacle concerns with the \emph{investigation of the sharp
asymptotics of the equilibrium potential}, which is the main ingredient
to apply the Dirichlet principle. To overcome this issue, we first
analyze the simplest case (Condition \ref{splcond}) and \emph{carefully
examine the Dirichlet form of the dynamics to find a proper test function
(Subsection \ref{ss-testfcndefspl}).} This test function naturally
induces a test flow (\eqref{fcnflow} and Proposition \ref{flowDiri}).
The second obstacle concerns with the \emph{control of the major and
minor parts of the divergence of this test flow}. This is essential
to apply Theorem \ref{genThom}, which originates from \cite{Seo NRZRP}.
We deal with the minor part in Subsection \ref{ss-testflowdefspl}
by \emph{replacing the test flow by its asymptotic limit}. Subsequently,
we address the major part in Subsections \ref{ss-flownormspl} and
\ref{ss-remflowspl}, using the fact that \emph{the equilibrium potential
behaves well near metastable valleys (cf. Lemma \ref{eqpotest}).}
After settling these problems for the simple case, we address the
general model by applying a similar method to obtain the main theorem.
A more detailed explanation of the procedure is provided in Remark
\ref{obstrmk} and Section \ref{s-outline}.

Moreover, we strongly agree with the other conjecture in \cite{BDG},
that $N^{2}/d_{N}^{3}$ is indeed the third time scale of this process,
and that the given three time scales completely characterize the metastable
behavior, indicating that there are no additional time scales in metastable
movements. However, it remains the possibility that an intermediate
step of metastable behavior emerges between them. Investigating the
third time scale is out of the scope of the current machineries developed
in this study (Remark \ref{mainrmk}(\ref{tshard})). Hence, this
topic serves as the main objective of future research in this direction.

Notably, the degree of $d_{N}$ increases by $1$ in the consecutive
time scales; $1/d_{N}\to N/d_{N}^{2}\to N^{2}/d_{N}^{3}$. This is
attributed to the fundamental property of transition rates of the
inclusion process. According to \eqref{inclusion}, the process has
to wait a time of order $1/d_{N}$ to send a particle to an empty
site. As long as a site is occupied with at least one particle, it
requires roughly constant-scale time to send the rest of the particles
there. Hence, the degrees of $d_{N}$ in the time scales represent
the graph distance (see footnote \ref{fngd}) between the corresponding
metastable states. Evidently, the scale grows as the distance increases.
This serves as a milestone in constructing the exact test function
representing the equilibrium potential (see Subsections \ref{ss-testfcndefspl}
and \ref{ss-testfcndef} for further detail).

In contrast, the metastability of non-reversible inclusion processes
occurs in an entirely different manner. It has been established in
\cite{KS NRIP} that there are two types of first time scales in the
system, namely $1/d_{N}$ if the limit dynamics is symmetric (and
thus generalizes the reversible case), and $1/(d_{N}N)$ if it is
asymmetric. Succeeding time scales are entirely unidentified and deserve
extensive future research. Further information is provided in \cite[Theorems 3.10 and 3.12]{KS NRIP}.

Metastability of inclusion processes is often compared to that of
the well-known \emph{(supercritical and critical) zero-range processes},
as they both involve bosonic particle systems representing stickiness
in low temperature. Moreover, both metastable behaviors can be proven
by a similar series of techniques, known as the martingale approach
\cite{BL TM,BL TM-nrv}. The main difference between them is that
unlike the inclusion process, the zero-range process exhibits single-step
metastable behavior; hence, there is only one time scale. This is
because particle movements in the zero-range process are affected
only by the number of particles on the starting site, which is the
reason behind the naming ``zero-range.'' Hence, on a suitable time
scale, all metastable states are reachable simultaneously. Full details
on the recent results on condensation and metastability of zero-range
processes are provided in \cite{AGL,BL ZRP,Landim TAZRP,Seo NRZRP}.

We assume throughout this article that the number of sites is fixed,
and the number of particles diverges. Alternatively, a model, named
the inclusion process in the \emph{thermodynamic limit} assumes that
the number of sites tends to infinity along with the number of particles,
so that particle density converges to a certain target density $\rho>0$.
In this case, yet another type of condensation and metastability is
detected. \cite{CCG} provides formulation and computational data,
while \cite{JCG} reports various condensation results depending on
the behavior of $d_{N}$, and \cite[Theorems 3.21, 3.22, and 3.23]{KS NRIP}
present the general result of metastability on the torus.

The main ingredients of the proof of our main result are the \emph{potential
theory} \cite{BdH} and the \emph{martingale approach} \cite{BL TM,BL TM-nrv}.
Compared to the classical pathwise approach to metastability, the
potential-theoretic approach has the big advantage of being highly
useful in calculating the sharp asymptotics of the mean hitting time
between metastable states and the consecutive metastable movements
among the sites in the limit. Based on this technology, Beltrán and
Landim proposed an outstanding method of calculating the mean transition
rates of the trace process by precisely estimating the corresponding
capacities of the system. We explain these methodologies in more detail
in Section \ref{s-outline}.

\section{Notation and Main Results}

We first settle the basic notation in this article. 
\begin{itemize}
\item The set of natural numbers, $\mathbb{N}$, includes $0$, i.e., $\mathbb{N}=\{n\in\mathbb{Z}:n\ge0\}$. 
\item Writing $\alpha,\,\beta\in\mathcal{R}$ or $\{\alpha,\,\beta\}\subseteq\mathcal{R}$
implies that $\alpha$ and $\beta$ are different. 
\item For integers $a$ and $b$ with $a\le b$, $\llbracket a,\,b\rrbracket$
represents $[a,\,b]\cap\mathbb{Z}$, i.e., the set of integers from
$a$ to $b$. 
\item For two sequences $\{\alpha_{N}\}_{N\ge1}$ and $\{\beta_{N}\}_{N\ge1}$
of real numbers, $\alpha_{N}$ and $\beta_{N}$ are \emph{asymptotically
equal}, or $\alpha_{N}\simeq\beta_{N}$ if $\lim_{N\to\infty}\alpha_{N}/\beta_{N}=1$. 
\item In what follows, $C$ denotes a global positive constant which may
vary among equations. 
\item For functions $f$ and $g$ of $N$, we write $f(N)=O(g(N))$ if there
exists a constant $C$ with the property that $|f(N)|\le Cg(N)$ for
all $N\ge1$. Moreover, we write $f(N)=o(g(N))$ if $\lim_{N\to\infty}f(N)/g(N)=0$. 
\end{itemize}

\subsection{Reversible inclusion processes}

We fix a finite state space $S$ which represents our collection of
sites. Suppose that $r:S\times S\to[0,\,\infty)$ is a transition
rate function which defines a continuous-time irreducible random walk
on $S$. For convenience, we let $r(x,\,x)=0$ for all $x\in S$.
We further assume that the random walk is \emph{reversible} with respect
to a probability distribution $m$, namely, 
\[
m(x)r(x,\,y)=m(y)r(y,\,x)\text{ for all }x,\,y\in S.
\]
The sites with maximal measure deserve particular attention, as they
are precisely the sites where particles condensate (Proposition \ref{conden}).
We define 
\[
M_{\star}=\max\{m(x):x\in S\},\ S_{\star}=\{x\in S:m(x)=M_{\star}\},\text{ and }m_{\star}(\cdot)=\frac{m(\cdot)}{M_{\star}}.
\]
Notably, $m_{\star}(x)\le1$ for all $x\in S$, and the equality holds
if and only if $x\in S_{\star}$.

Based on the underlying random walk introduced above, we introduce
the inclusion process on $S$. First, the set of configurations corresponding
to the distribution of $N$ particles on $S$ is denoted by 
\[
\mathcal{H}_{N}=\Big\{\eta\in\mathbb{N}^{S}:\sum_{x\in S}\eta_{x}=N\Big\}.
\]
Hence, $\eta_{x}$ is regarded as the number of particles at $x\in S$
of $\eta.$

Now, we define the \emph{inclusion process} to be a continuous-time
Markov chain $\{\eta_{N}(t)\}_{t\ge0}$ on $\mathcal{H}_{N}$ associated
with generator $\mathcal{L}_{N}$ acting on functions $f:\mathcal{H}_{N}\to\mathbb{R}$
by 
\begin{equation}
(\mathcal{L}_{N}f)(\eta)=\sum_{x,\,y\in S}\eta_{x}(d_{N}+\eta_{y})r(x,\,y)\{f(\sigma^{x,\,y}\eta)-f(\eta)\}\text{ for }\eta\in\mathcal{H}_{N}.\label{inclusion}
\end{equation}
Here, $\sigma^{x,\,y}\eta$ is the configuration obtained from $\eta$
by sending a particle, if possible, from $x$ to $y$. Hence, if $\eta_{x}=0$,
then $\sigma^{x,\,y}\eta=\eta$ and if $\eta_{x}\ge1$, then $(\sigma^{x,\,y}\eta)_{x}=\eta_{x}-1$,
$(\sigma^{x,\,y}\eta)_{y}=\eta_{y}+1$, and $(\sigma^{x,\,y}\eta)_{z}=\eta_{z}$
for $z\ne x,\,y$. Moreover, $\{d_{N}\}_{N\ge1}$ is a sequence of
positive real numbers converging to $0$. We will further assume that
$d_{N}$ decays more quickly than the logarithmic scale; 
\begin{equation}
\lim_{N\to\infty}d_{N}\log N=0.\label{dNcond}
\end{equation}
A typical choice for $d_{N}$ in practice is the polynomial scale,
$d_{N}=1/N^{\alpha}$, $\alpha>0$. One can readily verify that $\eta_{N}(\cdot)$
is irreducible. We denote the transition rate of this process by $\mathbf{q}_{N}:\mathcal{H}_{N}\times\mathcal{H}_{N}\to[0,\,\infty)$,
and the law and expectation of the process starting at $\eta$ by
$\mathbb{P}_{\eta}=\mathbb{P}_{\eta}^{N}$ and $\mathbb{E}_{\eta}=\mathbb{E}_{\eta}^{N}$,
respectively.

We conclude this subsection with a brief explanation of the dynamical
characteristics of the inclusion process. Given a configuration $\eta\in\mathcal{H}_{N}$,
a particle moves from site $x$ to site $y$ at rate 
\[
\mathbf{q}_{N}(\eta,\,\sigma^{x,\,y}\eta)=\eta_{x}(d_{N}+\eta_{y})r(x,\,y)=d_{N}\eta_{x}r(x,\,y)+\eta_{x}\eta_{y}r(x,\,y).
\]
Here, $d_{N}\eta_{x}r(x,\,y)$ denotes the \emph{diffusive} part and
$\eta_{x}\eta_{y}r(x,\,y)$ denotes the \emph{inclusive} part of the
dynamics. More specifically, the diffusive part represents the random
walk of each particle with respect to $r(\cdot,\,\cdot)$, which is
controlled by a parameter $d_{N}$. In contrast, the inclusive part
represents the attractive behavior of particles, because the rate
from $x$ to $y$ increases as $\eta_{y}$ increases, and particles
tend to prefer more occupied sites. As $d_{N}$ decays to $0$, the
inclusive behavior is expected to dominate the dynamics. Consequently,
particles are very likely to assemble at a single site, forming a
\emph{condensate} (Proposition \ref{conden}). However, the small
diffusive interactions trigger a long-term evolution of this condensate
among sites, which is referred to as \emph{tunneling} or \emph{metastable}
behavior (Theorem \ref{fts}). Precise interpretation of these concepts
is provided in the following.

\subsection{Condensation of reversible inclusion processes}

Because the process $\eta_{N}(\cdot)$ is irreducible, it exhibits
a unique invariant distribution. We denote the unique distribution
by $\mu_{N}$. The great advantage we gain by assuming reversibility
of the underlying random walk is that $\eta_{N}(\cdot)$ likewise
becomes reversible with respect to $\mu_{N}$, and that $\mu_{N}$
admits an explicit formula. This is stated in the following proposition,
whose proof is straightforward. Hereafter, $\Gamma(\cdot)$ denotes
the typical Gamma function. 
\begin{prop}
\label{mun}The inclusion process $\{\eta_{N}(t)\}_{t\ge0}$ is reversible
with respect to the invariant measure $\mu_{N}$, which satisfies
\begin{equation}
\mu_{N}(\eta)=\frac{1}{Z_{N}}\prod_{x\in S}w_{N}(\eta_{x})m_{\star}(x)^{\eta_{x}}\text{ for }\eta\in\mathcal{H}_{N},\label{muneq1}
\end{equation}
where 
\[
w_{N}(n)=\frac{\Gamma(d_{N}+n)}{n!\Gamma(d_{N})},\ n\in\mathbb{N},\text{ and }Z_{N}=\sum_{\eta\in\mathcal{H}_{N}}\prod_{x\in S}w_{N}(\eta_{x})m_{\star}(x)^{\eta_{x}}.
\]
\end{prop}

\begin{rem}
The following asymptotics hold for the functions introduced in Proposition
\ref{mun}: 
\begin{equation}
1\le\frac{(d_{N}+k)w_{N}(k)}{d_{N}}=\frac{(k+1)w_{N}(k+1)}{d_{N}}\le e^{d_{N}\log N},\ k\ge0,\text{ and }\lim_{N\to\infty}\frac{NZ_{N}}{d_{N}}=|S_{\star}|.\label{munprop}
\end{equation}
In particular, $(d_{N}+k)w_{N}(k)=(k+1)w_{N}(k+1)\simeq d_{N}$ by
\eqref{dNcond}, which is uniform in $k\ge0$. These convergence results
are frequently applied in the following. The proofs are provided in
\cite[Lemma 3.1 and Proposition 3.2]{BDG}. 
\end{rem}

Next, we define the metastable valleys of the process. Let

\[
\mathcal{E}_{N}^{x}=\{\xi^{x}\}=\{\eta\in\mathcal{H}_{N}:\eta_{x}=N\}\text{ for }x\in S.
\]
Hence, $\xi^{x}$ represents the configuration where all particles
are concentrated on the site $x$. Each $\mathcal{E}_{N}^{x}$ is
referred to as a \emph{valley} of the system. Moreover, we denote
$\mathcal{E}_{N}(A)=\bigcup_{x\in A}\mathcal{E}_{N}^{x}$ for $A\subseteq S$.
Valleys of further special interest are $\mathcal{E}_{N}^{x}$ for
$x\in S_{\star}$, as explained by the following proposition. The
proof of this is provided in \cite[Proposition 2.1]{BDG}. 
\begin{prop}
\label{conden}For each $x\in S_{\star}$, it holds that 
\begin{equation}
\lim_{N\to\infty}\mu_{N}(\mathcal{E}_{N}^{x})=\frac{1}{|S_{\star}|}.\label{condeneq}
\end{equation}
Consequently, we have $\lim_{N\to\infty}\mu_{N}(\mathcal{H}_{N}\setminus\mathcal{E}_{N}(S_{\star}))=0$. 
\end{prop}

\begin{rem}
In particular, $\mathcal{E}_{N}^{x}$, $x\in S_{\star}$, are referred
to as \emph{metastable valleys} of the process. 
\end{rem}

For simplicity, we write $\mathcal{E}_{N}^{\star}=\mathcal{E}_{N}(S_{\star})$.
Proposition \ref{conden} implies that the \emph{(static) condensation}
occurs on $S_{\star}$, i.e., 
\begin{equation}
\lim_{N\to\infty}\mu_{N}(\mathcal{E}_{N}^{\star})=1.\label{condensation}
\end{equation}
This fact depends heavily on the explicit formula \eqref{muneq1}.
If the underlying random walk is non-reversible, then the right-hand
side of \eqref{muneq1} is not necessarily the invariant distribution
of the system. In fact, we do not have a closed formula of the invariant
distribution in this case. Thus, even the basic condensation result
on valleys is not a simple issue for non-reversible inclusion processes.
Nevertheless, condensation on $\mathcal{E}_{N}(S)$ \emph{can} be
demonstrated for non-reversible systems by adding a few minor conditions
on $d_{N}$ and $r(\cdot,\,\cdot)$. For a recent result on this topic,
we refer to \cite[Theorem 3.15]{KS NRIP}.

\subsection{\label{ss-fts}First time scale of the metastable behavior of reversible
dynamics}

The first time scale is fully characterized in \cite{BDG}. We recall
the result in this subsection to motivate our main result of this
study. To this end, we must first introduce the trace process.

A non-empty subset $\mathcal{G}$ of $\mathcal{H}_{N}$ is fixed,
and a non-decreasing random variable $T^{\mathcal{G}}$, the local
time in $\mathcal{G}$ of the process, is defined by 
\[
T^{\mathcal{G}}(t)=\int_{0}^{t}\mathbbm{1}\{\eta_{N}(s)\in\mathcal{G}\}ds,\ t\ge0.
\]
Let $S^{\mathcal{G}}$ be its generalized inverse function:
\[
S^{\mathcal{G}}(t)=\sup\{s\ge0:T^{\mathcal{G}}(s)\le t\},\ t\ge0.
\]
Then, the \emph{trace process} $\{\eta_{N}^{\mathcal{G}}(t)\}_{t\ge0}$
on $\mathcal{G}$ is defined by 
\[
\eta_{N}^{\mathcal{G}}(t)=\eta_{N}(S^{\mathcal{G}}(t))\text{ for }t\ge0.
\]
The random time $T^{\mathcal{G}}(t)$ measures the amount of time
up to $t$ that the process spends in $\mathcal{G}$. Hence, the random
function $S^{\mathcal{G}}$ reconstructs the global time of the process,
starting from the local time in $\mathcal{G}$. In this sense, the
trace process $\eta_{N}^{\mathcal{G}}(\cdot)$ on $\mathcal{G}$ is
obtained from the original process $\eta_{N}(\cdot)$ by turning off
the clock whenever it is not in $\mathcal{G}$. Therefore, $\eta_{N}^{\mathcal{G}}(\cdot)$
becomes a continuous-time, irreducible Markov chain on $\mathcal{G}$.
Rigorous proof of this fact can be found in e.g., \cite{BL TM}.

Here, we trace the original process $\eta_{N}(\cdot)$ on $\mathcal{E}_{N}^{\star}$,
where condensation occurs. For simplicity, it is denoted by 
\[
\eta_{N}^{\star}(\cdot)=\eta_{N}^{\mathcal{E}_{N}^{\star}}(\cdot).
\]
As we are concerned only with the superscripts of the sets $\{\mathcal{E}_{N}^{x}:x\in S_{\star}\}$,
we define a projection function $\Psi_{1,\,N}:\mathcal{E}_{N}^{\star}\to S_{\star}$
as 
\[
\Psi_{1,\,N}(\xi^{x})=x\text{ for }x\in S_{\star}.
\]
The symbol $1$ in the subscript of $\Psi_{1,\,N}$ denotes the \emph{first}
time scale of metastability. Using this function, we define a process
$\{X_{N}(t)\}_{t\ge0}$ on $S_{\star}$ by 
\[
X_{N}(t)=\Psi_{1,\,N}(\eta_{N}^{\star}(t))\text{ for }t\ge0.
\]
In general, $X_{N}(\cdot)$ is non-Markovian, as it is merely a process
of labelling of the metastable valleys. However, in the case of inclusion
processes, $X_{N}(\cdot)$ is indeed a Markov process, as $\Psi_{1,\,N}$
is a bijection between $\mathcal{E}_{N}^{\star}$ and $S_{\star}$.

Here, we can formulate the first metastable behavior in terms of the
projected trace process $X_{N}(\cdot).$ Proof of the following theorem
is provided in \cite[Section 4]{BDG}. 
\begin{thm}[First time scale of reversible inclusion processes]
\label{fts} Fix a site $x_{0}\in S_{\star}$ and let $\theta_{N,\,1}=1/d_{N}$.
\begin{enumerate}
\item The law of the rescaled process $\{X_{N}(\theta_{N,\,1}t)\}_{t\ge0}$
starting at $x_{0}$ converges (with respect to the Skorokhod topology)
on the path space $D([0,\,\infty);\,S_{\star})$ to the law of the
Markov process $\{X_{\textup{first}}(t)\}_{t\ge0}$ on $S_{\star}$
starting at $x_{0}$, which is defined by the generator 
\[
(\mathscr{L}_{1}f)(x)=\sum_{y\in S_{\star}}r(x,\,y)\{f(y)-f(x)\},\ x\in S_{\star}\text{ for }f:S_{\star}\to\mathbb{R}.
\]
\item The process spends negligible time outside the metastable valleys,
i.e., for all $t>0$, 
\[
\lim_{N\to\infty}\sup_{\eta\in\mathcal{E}_{N}^{\star}}\mathbb{E}_{\eta}\Big[\int_{0}^{t}\mathbbm{1}\{\eta_{N}(\theta_{N,\,1}s)\notin\mathcal{E}_{N}^{\star}\}ds\Big]=0.
\]
\end{enumerate}
\end{thm}

\begin{rem}
\label{ftsrmk}In Theorem \ref{fts}, the limiting dynamics $X_{\textup{first}}(\cdot)$
is exactly the underlying random walk restricted to $S_{\star}$.
Here, we must note that even though the underlying system is irreducible,
$X_{\textup{first}}(\cdot)$ can still not be irreducible. For example,
let $S=\{1,\,2,\,3\}$, $r(1,\,2)=r(3,\,2)=1$, and $r(2,\,1)=r(2,\,3)=2$,
as in the left part of Figure \ref{fig1}. Then, we have $S_{\star}=\{1,\,3\}$;
thus, $X_{\textup{first}}(\cdot)$ on $S_{\star}$ represents the
null Markov chain. This phenomenon suggests additional time scales
of the metastable behavior of inclusion processes.

We further remark that \emph{non-reversible} inclusion processes exhibit
a completely different scheme of metastability in the first time scale.
Namely, in non-reversible inclusion processes, the time scale is $1/d_{N}$
if the limiting Markov chain of the process (cf. $X_{\textup{first}}(\cdot)$
in Theorem \ref{fts}) is symmetric, and it is $1/(d_{N}N)$ if the
limiting Markov chain is not symmetric. This is a remarkable difference
in the metastability of reversible and non-reversible inclusion processes,
and the details are provided in \cite[Theorems 3.10 and 3.12]{KS NRIP}. 
\end{rem}

\subsection{\label{ss-stsspl}Second time scale of the metastable behavior of
reversible dynamics: Simple case}

In this subsection, we present a simple case of our general main result.
Namely, we assume that the following condition holds throughout this
subsection.
\begin{condition}
\label{splcond}$S=\{x_{1},\,x_{2},\,y_{1},\,y_{2}\}$ with 
\begin{equation}
r(y_{p},\,x_{i})>r(x_{i},\,y_{p})>0\text{ for }1\le i,\,p\le2,\label{splcondeq1}
\end{equation}
\begin{equation}
r(x_{1},\,x_{2})=r(x_{2},\,x_{1})=0,\text{ and}\label{splcondeq2}
\end{equation}
\begin{equation}
m(x_{1})=m(x_{2}).\label{splcondeq3}
\end{equation}
In this setting, because the process is reversible, we have $m_{\star}(x_{1})=m_{\star}(x_{2})=1$,
$m_{\star}(y_{1})<1$, and $m_{\star}(y_{2})<1$, so that $S_{\star}=\{x_{1},\,x_{2}\}$.
See the right part of Figure \ref{fig1} for a visualization of this
simple model.
\end{condition}

\begin{figure}
\includegraphics[width=0.84\textwidth]{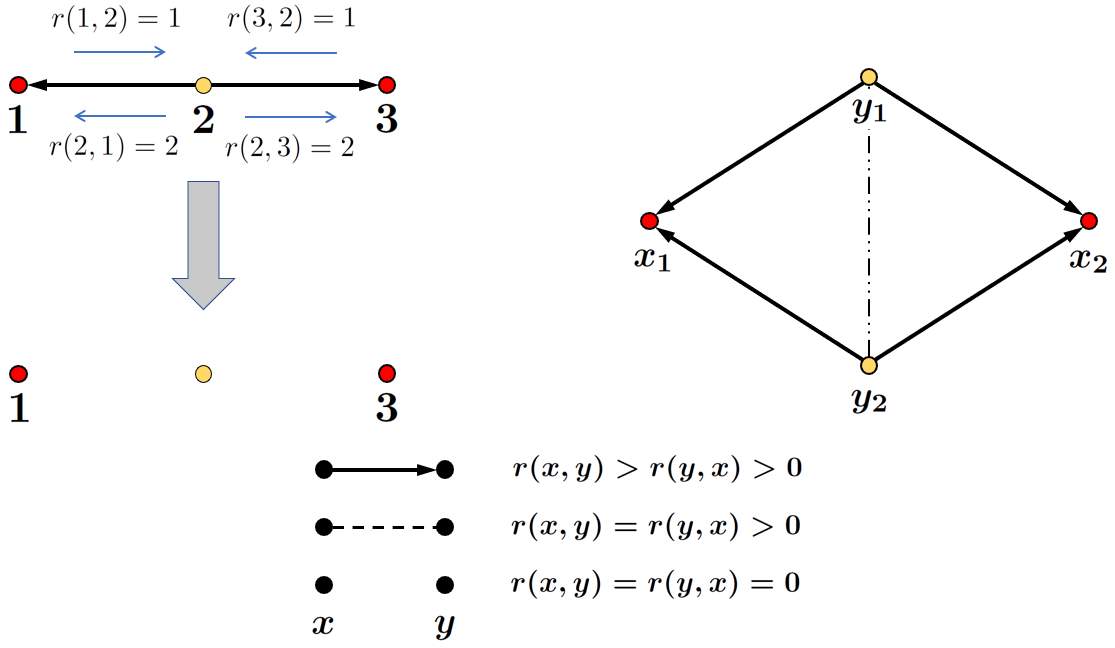} \caption{\label{fig1}(Left) a model where the first time scale of metastability
does not occur (Remark \ref{ftsrmk}). Red points denote the metastable
sites and yellow ones denote the rest. (Right) a simple model for
the second time scale of metastability (Condition \ref{splcond}).
The line between $y_{1}$ and $y_{2}$ implies that there is no restriction
on $r(y_{1},\,y_{2})$ and $r(y_{2},\,y_{1})$.}
\end{figure}

There are two reasons for providing a simple version of the theorem
first, instead of directly addressing the general main result. The
first reason is that this simple model already covers most of the
mathematical essentials of the second level of metastable behavior.
The second reason is that proposing the proof of the general main
result in a straightforward manner would be confusing to the readers,
while inspecting the proof of the simple case first is helpful.

We add the term \emph{spl}, which denotes \emph{simple}, in the superscripts
of some quantities defined in this subsection to avoid possible confusion
with the general main result in the following subsection.

By \eqref{splcondeq2}, we do not observe any movements in the first
time scale by Theorem \ref{fts}. Thus, it is natural to seek the
following time scale, in which metastable behavior is exhibited between
$x_{1}$ and $x_{2}$. Similar to the first scale, we define a projection
function $\Psi_{2,\,N}^{\textup{spl}}:\mathcal{E}_{N}^{\star}\to\{1,\,2\}$
by 
\[
\Psi_{2,\,N}^{\textup{spl}}(\xi^{x_{i}})=i\text{ for }1\le i\le2.
\]
Then, we define a process $Y_{N}^{\textup{spl}}(\cdot)$ on $\{1,\,2\}$
by 
\[
Y_{N}^{\textup{spl}}(t)=\Psi_{2,\,N}^{\textup{spl}}(\eta_{N}^{\star}(t))\text{ for }t\ge0.
\]
Following the notation of \cite{BDG}, we state that $d_{N}$ \emph{decays
subexponentially}, if 
\begin{equation}
\lim_{N\to\infty}d_{N}e^{\epsilon N}=\infty\text{ for any }\epsilon>0.\label{dNcond2}
\end{equation}
Hence, \eqref{dNcond2} indicates that $d_{N}$ decays more slowly
relative to any exponential scales. Moreover, we define a positive
constant $\mathfrak{R}$ by 
\begin{equation}
\mathfrak{R}=\int_{0}^{1}\frac{1}{\sum_{p=1}^{2}\frac{(1-m_{\star}(y_{p}))^{-1}}{\frac{1-t}{r(x_{1},\,y_{p})}+\frac{t}{r(x_{2},\,y_{p})}}}dt.\label{Tdef}
\end{equation}

\begin{thm}[Second time scale of reversible inclusion processes: Simple case]
\label{mainspl} Assume Condition \ref{splcond}. Suppose that $d_{N}$
decays subexponentially. Define the second time scale as $\theta_{N,\,2}=N/d_{N}^{2}$
and fix $i_{0}\in\{1,\,2\}$. Then, the law of the rescaled process
$\{Y_{N}^{\textup{spl}}(\theta_{N,\,2}t)\}_{t\ge0}$ starting at $i_{0}$
converges (with respect to the Skorokhod topology) on the path space
$D([0,\,\infty);\,\{1,\,2\})$ to the law of the Markov process on
$\{1,\,2\}$, starting at $i_{0}$ and jumping back and forth at rate
$1/\mathfrak{R}$. 
\end{thm}

\begin{rem}
\label{obstrmk}Note that Theorem \ref{mainspl} slightly generalizes
\cite[Theorem 2.5]{BDG}, and there is a sole additional non-metastable
site in the system. However, the approach used in \cite[Section 5]{BDG}
fails even in this simplest case, due to two important drawbacks.
First, the test function given in \cite[Subsection 5.2]{BDG} does
not provide a direct clue of the test function we need for this generalized
model, as this step requires a high-level understanding of the whole
landscape of the transition rates. This is provided in Subsection
\ref{ss-testfcndefspl}. Second, it is impossible to apply the Cauchy--Schwarz
inequality consecutively as in \cite[Subsection 5.1]{BDG}. This is
because the inequalities used there do not provide a consistent equality
condition; hence, this merely yields a weaker lower bound for the
capacities. To overcome this, we employ Theorem \ref{genThom} to
obtain the lower bound, which was proposed in \cite{Seo NRZRP}; see
Section \ref{s-LBspl} for further detail. 
\end{rem}

\subsection{\label{ss-sts}Second time scale of the metastable behavior of reversible
dynamics: General case}

Finally, in this subsection, we present the main result of this article
in the most general setting. To this end, we decompose $S_{\star}$
into \emph{irreducible components} with respect to $X_{\textup{first}}(\cdot)$,
which is the limiting dynamics in the first scale (see the left part
of Figure \ref{fig2}):
\begin{equation}
S_{\star}=\bigcup_{i=1}^{\kappa_{\star}}S_{i}^{(2)},\text{ where }S_{i}^{(2)}=\{x_{i,\,1},\,\dots,\,x_{i,\,\mathfrak{n}(i)}\}\text{ for }1\le i\le\kappa_{\star}.\label{Sstdecomp}
\end{equation}

\begin{figure}
\includegraphics[width=0.84\textwidth]{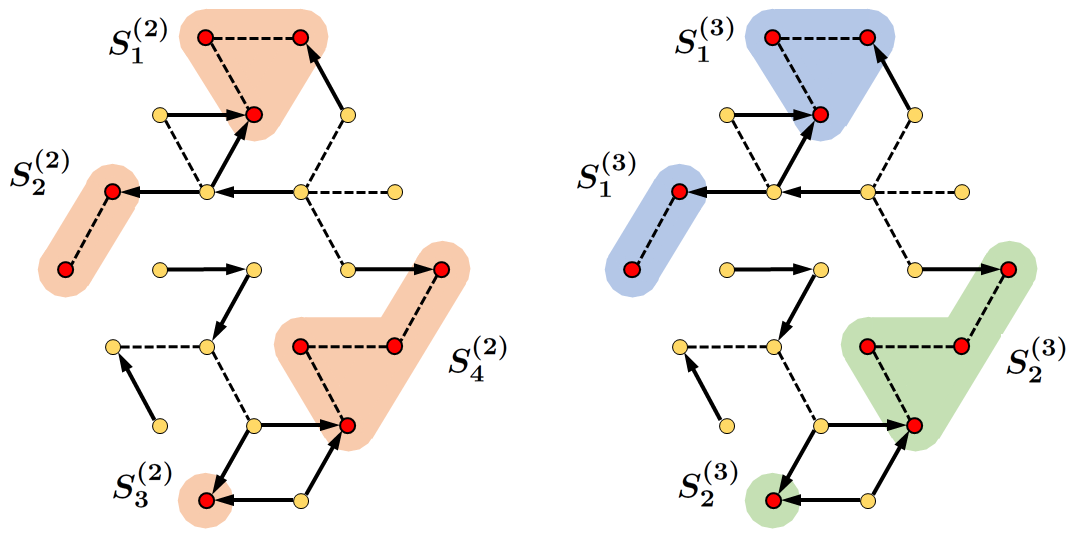}\caption{\label{fig2}(Left) a general model with four irreducible components
$S_{1}^{(2)}$, $S_{2}^{(2)}$, $S_{3}^{(2)}$, and $S_{4}^{(2)}$
according to the first time scale (cf. \eqref{Sstdecomp}). (Right)
the same model, in which $\{S_{1}^{(2)},\,S_{2}^{(2)}\}$ and $\{S_{3}^{(2)},\,S_{4}^{(2)}\}$
form two irreducible components $S_{1}^{(3)}$ and $S_{2}^{(3)}$,
respectively, according to the second time scale (Theorem \ref{main}).}
\end{figure}

Here, we use the second label ($1$ to $\mathfrak{n}(i)$) in the
elements of $S_{i}^{(2)}$ to emphasize that they belong to the same
set $S_{i}^{(2)}$. The common superscript $(2)$ denotes the second
time scale. More specifically, the system with transition rates $r(\cdot,\,\cdot)$
restricted to $S_{i}^{(2)}$ is irreducible for each $i\in\llbracket1,\,\kappa_{\star}\rrbracket$,
and $r(x_{i,\,n},\,x_{j,\,m})=0$ for all $i\ne j$, $1\le n\le\mathfrak{n}(i)$,
and $1\le m\le\mathfrak{n}(j)$. By definition, we have
\[
|S_{\star}|=\sum_{i=1}^{\kappa_{\star}}|S_{i}^{(2)}|=\sum_{i=1}^{\kappa_{\star}}\mathfrak{n}(i).
\]
In this setting, our dynamics in the second scale $\theta_{N,\,2}=N/d_{N}^{2}$
takes place on the set of $\kappa_{\star}$ elements; $\{\mathcal{E}_{N}(S_{i}^{(2)}):1\le i\le\kappa_{\star}\}$.
All elements in $S_{i}^{(2)}$ that are connected in the first scale
$\theta_{N,\,1}$ form a metastable group in the second scale (Theorem
\ref{main}(1)). If $\kappa_{\star}=1$, we observe all possible metastable
movements in the first time scale; thus, the metastable behavior is
fully characterized by Theorem \ref{fts}, and there is no need for
an additional time scale. Hence, hereafter we assume that $\kappa_{\star}\ge2$.
Moreover, we write $S\setminus S_{\star}=\{y_{1},\,\dots,\,y_{\kappa_{0}}\}$,
such that we have 
\[
S=\{x_{i,\,n}:1\le i\le\kappa_{\star},\ 1\le n\le\mathfrak{n}(i)\}\cup\{y_{1},\,\dots,\,y_{\kappa_{0}}\}.
\]
From $\kappa_{\star}\ge2$ and irreducibility of the underlying random
walk, it is straightforward that $\kappa_{0}\ge1$. For $A\subseteq\llbracket1,\,\kappa_{\star}\rrbracket$,
we introduce the notation $\mathcal{E}_{N}^{(2)}(A)=\bigcup_{i\in A}\mathcal{E}_{N}(S_{i}^{(2)})$.
If $A=\{a\}$, we abbreviate $\mathcal{E}_{N}^{(2)}(\{a\})$ as $\mathcal{E}_{N}^{(2)}(a)$.

As in the simple case, we define a projection function. Let $\Psi_{2,\,N}:\mathcal{E}_{N}^{\star}\to\llbracket1,\,\kappa_{\star}\rrbracket$
be defined by 
\[
\Psi_{2,\,N}(\xi^{x_{i,\,n}})=i\text{ for }1\le i\le\kappa_{\star}\text{ and }1\le n\le\mathfrak{n}(i).
\]
Then, we define a process $Y_{N}(\cdot)$ on $\llbracket1,\,\kappa_{\star}\rrbracket$
by 
\[
Y_{N}(t)=\Psi_{2,\,N}(\eta_{N}^{\star}(t))\text{ for }t\ge0.
\]
In contrast to $X_{N}(\cdot)$ (and $Y_{N}^{\textup{spl}}(\cdot)$),
$Y_{N}(\cdot)$ is not necessarily Markovian, since $\Psi_{2,\,N}$
is generally not bijective.

We are ready to state our main theorem. We define constants $\mathfrak{R}_{i,\,j}$
for $i,\,j\in\llbracket1,\,\kappa_{\star}\rrbracket$: 
\begin{equation}
\mathfrak{R}_{i,\,j}=\int_{0}^{1}\frac{1}{\sum_{n=1}^{\mathfrak{n}(i)}\sum_{m=1}^{\mathfrak{n}(j)}\sum_{p=1}^{\kappa_{0}}\frac{(1-m_{\star}(y_{p}))^{-1}}{\frac{1-t}{r(x_{i,\,n},\,y_{p})}+\frac{t}{r(x_{j,\,m},\,y_{p})}}}dt.\label{Tijdef}
\end{equation}
In \eqref{Tijdef}, we regard the fraction in the denominator as $0$
if $r(x_{i,\,n},\,y_{p})r(x_{j,\,m},\,y_{p})=0$. In this sense, we
write $\mathfrak{R}_{i,\,j}=\infty$ if $r(x_{i,\,n},\,y_{p})r(x_{j,\,m},\,y_{p})=0$
for all $1\le n\le\mathfrak{n}(i)$, $1\le m\le\mathfrak{n}(j)$,
and $1\le p\le\kappa_{0}$\footnote{We take $1/\infty$ to be $0$ in the following.}.
It is clear that $\mathfrak{R}_{i,\,j}=\mathfrak{R}_{j,\,i}$.

Further, for $\mathcal{A}\subseteq\mathcal{H}_{N}$, let $\tau_{\mathcal{A}}=\tau_{\mathcal{A}}^{N}$
be the hitting time of the set $\mathcal{A}$. 
\begin{thm}[Second time scale of reversible inclusion processes: General case]
\label{main} Suppose that $d_{N}$ decays subexponentially. Then,
with $\theta_{N,\,2}=N/d_{N}^{2}$, the following statements hold.
\begin{enumerate}
\item For each $1\le i\le\kappa_{\star}$, $\mathcal{E}_{N}(S_{i}^{(2)})$
thermalizes before reaching another metastable set, i.e., 
\begin{equation}
\lim_{N\to\infty}\inf_{\eta,\,\zeta\in\mathcal{E}_{N}(S_{i}^{(2)})}\mathbb{P}_{\eta}[\tau_{\{\zeta\}}<\tau_{\mathcal{E}_{N}(S_{\star}\setminus S_{i}^{(2)})}]=1.\label{maineq}
\end{equation}
\item Fix $i_{0}\in\llbracket1,\,\kappa_{\star}\rrbracket$. Then, the law
of the rescaled process $\{Y_{N}(\theta_{N,\,2}t)\}_{t\ge0}$ starting
at $i_{0}$ converges (with respect to the Skorokhod topology) on
the path space $D([0,\,\infty);\,\llbracket1,\,\kappa_{\star}\rrbracket)$
to the law of the Markov process $X_{\textup{second}}(\cdot)$ on
$\llbracket1,\,\kappa_{\star}\rrbracket$ starting at $i_{0}$ and
defined by the generator, acting on functions $f:\llbracket1,\,\kappa_{\star}\rrbracket\to\mathbb{R}$,
given by 
\begin{equation}
(\mathscr{L}_{2}f)(i)=\sum_{j\in\llbracket1,\,\kappa_{\star}\rrbracket\setminus\{i\}}\frac{1}{|S_{i}^{(2)}|\mathfrak{R}_{i,\,j}}\{f(j)-f(i)\}\text{ for }i\in\llbracket1,\,\kappa_{\star}\rrbracket.\label{maineq2}
\end{equation}
Consequently, $S_{\star}$ is decomposed into irreducible components
with respect to $X_{\textup{second}}(\cdot)$. We denote this partition
by 
\begin{equation}
S_{\star}=S_{1}^{(3)}\cup\cdots\cup S_{\gamma_{\star}}^{(3)}.\label{maineq3}
\end{equation}
\item Fix $1\le i\le\kappa_{\star}$ and $1\le n\le\mathfrak{n}(i)$. From
\eqref{maineq3}, there is a unique $\hat{i}\in\llbracket1,\,\gamma_{\star}\rrbracket$
such that $S_{i}^{(2)}\subseteq S_{\hat{i}}^{(3)}$ (see the right
part of Figure \ref{fig2}). Then, starting at $\xi^{x_{i,\,n}}$,
the process spends negligible time outside $\mathcal{E}_{N}(S_{\hat{i}}^{(3)})$,
which is uniform in all choices of $(i,\,n)$, i.e., for all $t>0$,
\[
\lim_{N\to\infty}\sup_{i\in\llbracket1,\,\kappa_{\star}\rrbracket,\,n\in\llbracket1,\,\mathfrak{n}(i)\rrbracket}\mathbb{E}_{\xi^{x_{i,\,n}}}\Big[\int_{0}^{t}\mathbbm{1}\{\eta_{N}(\theta_{N,\,2}s)\notin\mathcal{E}_{N}(S_{\hat{i}}^{(3)})\}ds\Big]=0.
\]
\end{enumerate}
\end{thm}

\begin{rem}
\label{mainrmk}We remark several issues regarding the main theorem. 
\begin{enumerate}
\item Note that Theorem \ref{mainspl} is indeed a special case of Theorem
\ref{main}, where $\kappa_{\star}=2$, $\mathfrak{n}(1)=\mathfrak{n}(2)=1$,
$x_{1,\,1}=x_{1}$, $x_{2,\,1}=x_{2}$, $\kappa_{0}=2$, and $r(x_{1,\,1},\,y_{p})r(x_{2,\,1},\,y_{p})>0$
for $p=1,\,2$. 
\item Theorem \ref{main} proves the conjecture in \cite{BDG} that $\theta_{N,\,2}=N/d_{N}^{2}$
is indeed the \emph{second} time scale in the metastability of reversible
inclusion processes, in the sense that there are no intermediate time
scales between $\theta_{N,\,1}$ and $\theta_{N,\,2}$.
\item By \eqref{maineq2}, for $i,j\in\llbracket1,\,\kappa_{\star}\rrbracket$,
the limit transition rate from $S_{i}^{(2)}$ to $S_{j}^{(2)}$ is
$1/(|S_{i}^{(2)}|\mathfrak{R}_{i,\,j})$. This vanishes if and only
if $\mathfrak{R}_{i,\,j}=\infty$, which is equivalent to state that
the graph distance\footnote{\label{fngd}For two subsets $A$ and $B$ of $S$, the \emph{graph
distance} between $A$ and $B$ is defined as $\min\{n\ge0:\exists x_{0},\,\dots,\,x_{n}\in S\text{ such that }x_{0}\in A,\ x_{n}\in B,\text{ and }r(x_{i},\,x_{i+1})>0\text{ for }0\le i\le n-1\}$.} between $S_{i}^{(2)}$ and $S_{j}^{(2)}$ is bigger than $2$. In
this sense, we cannot observe a metastable movement between $S_{i}^{(3)}$
and $S_{j}^{(3)}$, $i,\,j\in\llbracket1,\,\gamma_{\star}\rrbracket$,
in the second time scale $\theta_{N,\,2}=N/d_{N}^{2}$. Because the
original underlying random walk is irreducible, it is natural to suggest
the existence of a third time scale, where we can detect metastable
movements among $S_{i}^{(3)}$, $1\le i\le\gamma_{\star}$. In \cite{BDG},
this scale is strongly expected to be $\theta_{N,\,3}=N^{2}/d_{N}^{3}$.
Moreover, even though $\theta_{N,\,3}$ is proven to be the longest
scale possible in \cite{BDG}, there is a possibility that an intermediate
time scale exists between $\theta_{N,\,2}$ and $\theta_{N,\,3}$.
This can be considered a fruitful future research topic. 
\item \label{tshard}According to the previous remark, we attempted to apply
the methodology used in this study to address the third time scale
of metastability of the inclusion process. The first obstacle is encountered
in constructing an exquisite test function which approximates the
equilibrium potential, as in Subsection \ref{ss-testfcndef}. This
becomes far more complicated when compared to what is done here, as
the geometric property of the typical path is highly complex in the
third time scale. The other obstacle is that the asymptotic value
of the equilibrium potential, which is successfully determined in
the second time scale in Subsections \ref{ss-remflowspl} and \ref{ss-remflow},
is unknown in the third time scale. In order to apply a similar methodology
in the third scale, a precise information on the equilibrium potential
of the entire typical path between metastable valleys is needed. At
this point, this is a technically difficult task. 
\item Note that \eqref{maineq} is not included in the previous metastability
results, i.e., Theorems \ref{fts} and \ref{mainspl}. This is because
in the setting of the previous theorems, each metastable valley is
a singleton; hence, thermalization is obvious. 
\item In this study, all convergence results are provided in terms of convergence
of the trace process in the Skorokhod topology. In fact, there are
alternatives to the stated results, represented by convergence of
the original process in the soft topology \cite{Landim Soft} and
convergence of finite-dimensional marginal distributions \cite{LanLouMou}.
We remark that given our result, the other modes of convergence may
be easily proven by verifying some additional technical conditions
presented in the foresaid studies. In Section \ref{s-PMT}, we prove
the convergence of finite-dimensional marginal distributions using
\cite[Proposition 2.1]{LanLouMou}. This result is needed to prove
(3) of Theorem \ref{main}. 
\end{enumerate}
\end{rem}

\section{\label{s-outline}Outline of proof of Theorems \ref{mainspl} and
\ref{main}}

In this section, we review some potential-theoretic notions and explain
how to apply these skills to prove the main theorems, namely Theorems
\ref{mainspl} and \ref{main}. We remark that some claims in this
section hold due to reversibility, and more general results are provided
in e.g., \cite{GL} and \cite{Slowik}.

\subsection{Potential theory and discrete flows}

In this subsection, we introduce some crucial notions from the potential
theory, which are fundamental in stating and proving our results. 
\begin{defn}
\label{potdef}
\begin{enumerate}
\item The \emph{Dirichlet form} $D_{N}(\cdot)$ associated to our inclusion
process is defined by, for $f:\mathcal{H}_{N}\to\mathbb{R}$,
\begin{equation}
\begin{aligned}D_{N}(f) & =\frac{1}{2}\sum_{\eta,\,\zeta\in\mathcal{H}_{N}}\mu_{N}(\eta)\mathbf{q}_{N}(\eta,\,\zeta)\{f(\zeta)-f(\eta)\}^{2}\\
 & =\frac{1}{2}\sum_{\eta\in\mathcal{H}_{N}}\sum_{x,\,y\in S}\mu_{N}(\eta)\eta_{x}(d_{N}+\eta_{y})r(x,\,y)\{f(\sigma^{x,\,y}\eta)-f(\eta)\}^{2}.
\end{aligned}
\label{Diridef}
\end{equation}
\item If $\mathcal{A}$ and $\mathcal{B}$ are disjoint and non-empty subsets
of $\mathcal{H}_{N}$, then the \emph{equilibrium potential} $h_{\mathcal{A},\,\mathcal{B}}$
between $\mathcal{A}$ and $\mathcal{B}$ is defined by 
\[
h_{\mathcal{A},\,\mathcal{B}}(\eta)=\mathbb{P}_{\eta}[\tau_{\mathcal{A}}<\tau_{\mathcal{B}}].
\]
Note that $h_{\mathcal{A},\,\mathcal{B}}=1$ on $\mathcal{A}$ and
$h_{\mathcal{A},\mathcal{B}}=0$ on $\mathcal{B}$.
\item The \emph{capacity} between $\mathcal{A}$ and $\mathcal{B}$ is defined
by 
\[
\mathrm{Cap}_{N}(\mathcal{A},\,\mathcal{B})=D_{N}(h_{\mathcal{A},\,\mathcal{B}}).
\]
\end{enumerate}
\end{defn}

Later in this section, we show that to study the metastable behavior
of interacting particle systems, it suffices to obtain sharp asymptotics
on the corresponding capacities between metastable valleys (cf. \eqref{B-L lem}
and Proposition \ref{B-L conv}). In the following, we define the
discrete flows corresponding to our system. 
\begin{defn}
\label{flowdef}
\begin{enumerate}
\item A function $\phi:\mathcal{H}_{N}\times\mathcal{H}_{N}\to\mathbb{R}$
is considered a \emph{(discrete) flow} on $\mathcal{H}_{N}$, if
\begin{enumerate}
\item $\phi$ is \emph{anti-symmetric}; that is, $\phi(\eta,\,\zeta)=-\phi(\zeta,\,\eta)$
for all $\eta,\,\zeta\in\mathcal{H}_{N}$, and
\item $\phi$ is \emph{compatible} to $\mathbf{q}_{N}(\cdot,\,\cdot)$;
that is, $\phi(\eta,\,\zeta)>0$ only if $\mathbf{q}_{N}(\eta,\,\zeta)>0$.
\end{enumerate}
\item An inner product structure $\langle\cdot,\,\cdot\rangle_{N}$ is defined
for the flows on $\mathcal{H}_{N}$: For flows $\phi$ and $\psi$,
\begin{equation}
\langle\phi,\,\psi\rangle_{N}=\frac{1}{2}\sum_{\eta,\,\zeta\in\mathcal{H}_{N}:\,\mathbf{q}_{N}(\eta,\,\zeta)>0}\frac{\phi(\eta,\,\zeta)\psi(\eta,\,\zeta)}{\mu_{N}(\eta)\mathbf{q}_{N}(\eta,\,\zeta)}.\label{flnormdef}
\end{equation}
Consequently, a norm $\|\cdot\|_{N}$ is established in the system;
$\|\phi\|_{N}=\sqrt{\langle\phi,\,\phi\rangle_{N}}$.
\item Given a flow $\phi$ on $\mathcal{H}_{N}$, the \emph{divergence}
of $\phi$ at $\eta\in\mathcal{H}_{N}$ is defined as 
\[
(\mathrm{div}\,\phi)(\eta)=\sum_{\zeta\in\mathcal{H}_{N}}\phi(\eta,\,\zeta)=\sum_{\zeta\in\mathcal{H}_{N}:\,\mathbf{q}_{N}(\eta,\,\zeta)>0}\phi(\eta,\,\zeta).
\]
Following this notation, the divergence of $\phi$ on $\mathcal{A}\subseteq\mathcal{H}_{N}$
is 
\[
(\mathrm{div}\,\phi)(\mathcal{A})=\sum_{\eta\in\mathcal{A}}(\mathrm{div}\,\phi)(\eta).
\]
\end{enumerate}
\end{defn}

Here, we connect two notions defined above: For a function $f:\mathcal{H}_{N}\to\mathbb{R}$,
we define a flow $\Phi_{f}$ on $\mathcal{H}_{N}$, which is given
by 
\begin{equation}
\Phi_{f}(\eta,\,\zeta)=\mu_{N}(\eta)\mathbf{q}_{N}(\eta,\,\zeta)\{f(\eta)-f(\zeta)\};\ \eta\in\mathcal{H}_{N},\ \zeta\in\mathcal{H}_{N}.\label{fcnflow}
\end{equation}
Then, the following proposition explains the relationship between
potential-theoretic objects and discrete flows. 
\begin{prop}
\label{flowDiri}For each $f:\mathcal{H}_{N}\to\mathbb{R}$, we have
$\|\Phi_{f}\|_{N}^{2}=D_{N}(f)$. Consequently, if $\mathcal{A}$
and $\mathcal{B}$ are disjoint and non-empty subsets of $\mathcal{H}_{N}$,
\begin{equation}
\|\Phi_{h_{\mathcal{A},\,\mathcal{B}}}\|_{N}^{2}=D_{N}(h_{\mathcal{A},\,\mathcal{B}})=\mathrm{Cap}_{N}(\mathcal{A},\,\mathcal{B}).\label{flowDirieq}
\end{equation}
\end{prop}

\begin{proof}
The first statement follows by a simple calculation. Namely,
\[
\begin{aligned}\|\Phi_{f}\|_{N}^{2} & =\frac{1}{2}\sum_{\eta,\,\zeta\in\mathcal{H}_{N}:\,\mathbf{q}_{N}(\eta,\,\zeta)>0}\frac{\{\Phi_{f}(\eta,\,\zeta)\}^{2}}{\mu_{N}(\eta)\mathbf{q}_{N}(\eta,\,\zeta)}\\
 & =\frac{1}{2}\sum_{\eta,\,\zeta\in\mathcal{H}_{N}:\,\mathbf{q}_{N}(\eta,\,\zeta)>0}\mu_{N}(\eta)\mathbf{q}_{N}(\eta,\,\zeta)\{f(\eta)-f(\zeta)\}^{2}=D_{N}(f).
\end{aligned}
\]
Then, \eqref{flowDirieq} is straightforward by Definition \ref{potdef}(3). 
\end{proof}

\subsection{Dual variational principles: The Dirichlet--Thomson principle}

According to the definitions in the last subsection, two important
variational principles hold, namely the Dirichlet principle and the
Thomson principle. These statements play a key role in calculating
the explicit behavior of the capacities in the limit $N\to\infty$. 
\begin{thm}[The Dirichlet--Thomson principle]
\label{Diri-Thom} Suppose that $\mathcal{A}$ and $\mathcal{B}$
are two disjoint and non-empty subsets of $\mathcal{H}_{N}$.
\begin{enumerate}
\item (Dirichlet) It holds that
\[
\mathrm{Cap}_{N}(\mathcal{A},\,\mathcal{B})=\inf_{F}D_{N}(F),
\]
where the infimum runs over functions $F:\mathcal{H}_{N}\to\mathbb{R}$
satisfying
\begin{equation}
F\big|_{\mathcal{A}}=1\text{ and }F\big|_{\mathcal{B}}=0.\label{Diri-Thomeq}
\end{equation}
Moreover, the unique optimizer of the infimum is $h_{\mathcal{A},\,\mathcal{B}}$.
\item (Thomson) It holds that 
\[
\mathrm{Cap}_{N}(\mathcal{A},\,\mathcal{B})=\sup_{\phi}\frac{1}{\|\phi\|_{N}^{2}},
\]
where the supremum runs over flows $\phi$ on $\mathcal{H}_{N}$ satisfying
\begin{equation}
(\mathrm{div}\,\phi)(\mathcal{A})=1\text{ and }(\mathrm{div}\,\phi)(\eta)=0\text{ for all }\eta\in\mathcal{H}_{N}\setminus(\mathcal{A\cup\mathcal{B}}).\label{Diri-Thomeq2}
\end{equation}
Moreover, the unique optimizer of the supremum is $\Phi_{h_{\mathcal{A},\,\mathcal{B}}}/\mathrm{Cap}_{N}(\mathcal{A},\,\mathcal{B})$.
\end{enumerate}
\end{thm}

In fact, the two principles in Theorem \ref{Diri-Thom} have generalizations
to non-reversible dynamics. Namely, \cite[Theorem 2.7]{GL} gives
the non-reversible version of (1), and \cite[Proposition 2.6]{Slowik}
gives the non-reversible version of (2). We refer to the references
\cite{GL,Slowik} for the proof of Theorem \ref{Diri-Thom}.
\begin{rem}
\label{Thomrmk}Like the Dirichlet principle, the Thomson principle
can also be stated in terms of test functions; however, it is somewhat
more complicated to state the result, which is provided in \cite[Proposition 2.1(ii)]{Slowik}. 
\end{rem}

The application of Theorem \ref{Diri-Thom} runs generally as follows.
First, we construct test functions and flows, say $g$ and $\psi$,
which satisfy \eqref{Diri-Thomeq} and \eqref{Diri-Thomeq2}, respectively.
Second, we apply the principles to obtain 
\[
\frac{1}{\|\psi\|_{N}^{2}}\le\mathrm{Cap}_{N}(\mathcal{A},\,\mathcal{B})\le D_{N}(g).
\]
Finally, we send $N$ to infinity to obtain the desired estimate.
From \eqref{flowDirieq}, it is natural to take $g$ and $\psi$,
which in some sense approximate $h_{\mathcal{A},\,\mathcal{B}}$ and
$\Phi_{h_{\mathcal{A},\,\mathcal{B}}}/\mathrm{Cap}_{N}(\mathcal{A},\,\mathcal{B})$,
respectively.

According to the above methodology, the Dirichlet principle is relatively
easy to apply, as the restriction \eqref{Diri-Thomeq} is feeble.
In contrast, the Thomson principle has a strong restriction on the
test flows, \eqref{Diri-Thomeq2}. In particular, it is practically
impossible to find a test flow that has vanishing divergence in each
configuration in $\mathcal{H}_{N}\setminus(\mathcal{A}\cup\mathcal{B})$.

To overcome this drawback, we need the following variant of the useful
result \cite[Theorem 5.3]{Seo NRZRP}, which generalizes the Thomson
principle. 
\begin{thm}[Generalized Thomson principle]
\label{genThom} Suppose that $\mathcal{A}$ and $\mathcal{B}$ are
two disjoint and non-empty subsets of $\mathcal{H}_{N}$. Then, for
any non-trivial flow $\psi$ on $\mathcal{H}_{N}$, it holds that
\begin{equation}
\mathrm{Cap}_{N}(\mathcal{A},\,\mathcal{B})\ge\frac{1}{\|\psi\|_{N}^{2}}\Big[\sum_{\eta\in\mathcal{H}_{N}}h_{\mathcal{A},\,\mathcal{B}}(\eta)(\mathrm{div}\,\psi)(\eta)\Big]^{2}.\label{genThomeq}
\end{equation}
Moreover, the equality holds if and only if $\psi=c\Phi_{h_{\mathcal{A},\,\mathcal{B}}}$
for a non-zero constant $c$. 
\end{thm}

\begin{proof}
By \cite[Proposition 5.1(3)]{Seo NRZRP} and the Cauchy--Schwarz
inequality, 
\begin{align*}
\Big[\sum_{\eta\in\mathcal{H}_{N}}h_{\mathcal{A},\,\mathcal{B}}(\eta)(\mathrm{div}\,\psi)(\eta)\Big]^{2} & =\langle\Phi_{h_{\mathcal{A},\,\mathcal{B}}},\,\psi\rangle_{N}^{2}\\
 & \le\|\Phi_{h_{\mathcal{A},\,\mathcal{B}}}\|_{N}^{2}\times\|\psi\|_{N}^{2}=\mathrm{Cap}_{N}(\mathcal{A},\,\mathcal{B})\times\|\psi\|_{N}^{2}.
\end{align*}
Because $\psi$ is non-trivial, we divide both sides by $\|\psi\|_{N}^{2}$
to obtain the result. The equality condition is straightforward. 
\end{proof}
\begin{rem}
If $\psi$ satisfies the condition \eqref{Diri-Thomeq2}, then Theorem
\ref{genThom} is equivalent to Theorem \ref{Diri-Thom}(2), as 
\[
\sum_{\eta\in\mathcal{H}_{N}}h_{\mathcal{A},\,\mathcal{B}}(\eta)(\mathrm{div}\,\psi)(\eta)=\sum_{\eta\in\mathcal{A}}(\mathrm{div}\,\psi)(\eta)=(\mathrm{div}\,\psi)(\mathcal{A})=1.
\]
Thus, using Theorem \ref{genThom}, we may choose a test flow $\psi$
that does not satisfy the strict constraint \eqref{Diri-Thomeq2}.
This point is demonstrated in Sections \ref{s-LBspl} and \ref{s-LB}. 
\end{rem}

\subsection{Martingale approach and outline of proof}

A sharp estimate of the capacities can be used to calculate the transition
rates of the process traced on $\mathcal{E}_{N}^{\star}$, employing
the following formula from \cite[Lemma 6.8]{BL TM}: We denote by
$\mathbf{q}_{N}^{\star}:\mathcal{E}_{N}^{\star}\times\mathcal{E}_{N}^{\star}\to[0,\,\infty)$
the transition rate of the trace process $\eta_{N}^{\star}(\cdot)$,
and we define the mean transition rate $\mathbf{r}_{N}^{\star}:\llbracket1,\,\kappa_{\star}\rrbracket\times\llbracket1,\,\kappa_{\star}\rrbracket\to[0,\,\infty)$
by $\mathbf{r}_{N}^{\star}(i,\,i)=0$ and 
\[
\mathbf{r}_{N}^{\star}(i,\,j)=\frac{1}{\mu_{N}(\mathcal{E}_{N}(S_{i}^{(2)}))}\sum_{\eta\in\mathcal{E}_{N}(S_{i}^{(2)})}\mu_{N}(\eta)\sum_{\zeta\in\mathcal{E}_{N}(S_{j}^{(2)})}\mathbf{q}_{N}^{\star}(\eta,\,\zeta)\text{ for }i,\,j\in\llbracket1,\,\kappa_{\star}\rrbracket.
\]
Then, for $i,j\in\llbracket1,\,\kappa_{\star}\rrbracket$,
\begin{equation}
\begin{aligned}\mu_{N} & (\mathcal{E}_{N}^{(2)}(i))\mathbf{r}_{N}^{\star}(i,\,j)=\frac{1}{2}\big[\mathrm{Cap}_{N}(\mathcal{E}_{N}^{(2)}(i),\,\mathcal{E}_{N}^{\star}\setminus\mathcal{E}_{N}^{(2)}(i))\\
 & +\mathrm{Cap}_{N}(\mathcal{E}_{N}^{(2)}(j),\,\mathcal{E}_{N}^{\star}\setminus\mathcal{E}_{N}^{(2)}(j))-\mathrm{Cap}_{N}(\mathcal{E}_{N}^{(2)}(\{i,\,j\}),\,\mathcal{E}_{N}^{\star}\setminus\mathcal{E}_{N}^{(2)}(\{i,\,j\}))\big].
\end{aligned}
\label{B-L lem}
\end{equation}
The asymptotics on $\mathbf{r}_{N}^{\star}(\cdot,\,\cdot)$ is the
main ingredient to describe the metastable behavior. This is explained
in the following proposition, which is a consequence of the martingale
approach developed in \cite{BL TM}. We refer to \cite[Theorem 2.7]{BL TM}
for its proof. 
\begin{prop}
\label{B-L conv}Suppose that there exists a sequence $\{\theta_{N}\}_{N\ge1}$
of positive real numbers such that 
\begin{equation}
\lim_{N\to\infty}\theta_{N}\mathbf{r}_{N}^{\star}(i,\,j)=a(i,\,j)\text{ for all }i,\,j\in\llbracket1,\,\kappa_{\star}\rrbracket,\label{H0}
\end{equation}
for some $a:\llbracket1,\,\kappa_{\star}\rrbracket\times\llbracket1,\,\kappa_{\star}\rrbracket\to[0,\,\infty)$.
Moreover, suppose that the following estimate holds for each $1\le i\le\kappa_{\star}$:
\begin{equation}
\lim_{N\to\infty}\frac{\mathrm{Cap}_{N}(\mathcal{E}_{N}^{(2)}(i),\,\mathcal{E}_{N}^{\star}\setminus\mathcal{E}_{N}^{(2)}(i))}{\inf_{\eta,\,\zeta\in\mathcal{E}_{N}^{(2)}(i)}\mathrm{Cap}_{N}(\{\eta\},\,\{\zeta\})}=0.\label{H1}
\end{equation}
Then, for each $1\le i\le\kappa_{\star}$, the following statements
hold.
\begin{enumerate}
\item $\mathcal{E}_{N}^{(2)}(i)$ thermalizes before reaching another metastable
set, i.e., 
\[
\lim_{N\to\infty}\inf_{\eta,\,\zeta\in\mathcal{E}_{N}^{(2)}(i)}\mathbb{P}_{\eta}[\tau_{\{\zeta\}}<\tau_{\mathcal{E}_{N}^{\star}\setminus\mathcal{E}_{N}^{(2)}(i)}]=1.
\]
\item For each $1\le i\le\kappa_{\star}$, the law of the rescaled process
$\{Y_{N}(\theta_{N}t)\}_{t\ge0}$ starting at $i$ converges (with
respect to the Skorokhod topology) on the path space $D([0,\,\infty);\,\llbracket1,\,\kappa_{\star}\rrbracket)$
to the law of the Markov process on $\llbracket1,\,\kappa_{\star}\rrbracket$
starting at $i$ with transition rates $a(\cdot,\,\cdot)$. 
\end{enumerate}
\end{prop}

To prove statement (3) of Theorem \ref{main}, we also must know the
mode of convergence of finite-dimensional distributions of the rescaled
process $\eta_{N}(\theta_{N}\cdot)$. \cite[Proposition 2.1]{LanLouMou}
provides a simple approach of proving this result. 
\begin{prop}
\label{ConvFDD}Suppose that statement (2) of Theorem \ref{main}
holds, and that the process spends negligible time outside the metastable
valleys, i.e., for $t>0$, 
\[
\lim_{N\to\infty}\sup_{\eta\in\mathcal{E}_{N}^{\star}}\mathbb{E}_{\eta}\Big[\int_{0}^{t}\mathbbm{1}\{\eta_{N}(\theta_{N,\,2}s)\notin\mathcal{E}_{N}^{\star}\}ds\Big]=0.
\]
In addition, suppose that the following holds:
\begin{equation}
\lim_{\delta\to0}\limsup_{N\to\infty}\sup_{2\delta\le s\le3\delta}\sup_{\eta\in\mathcal{E}_{N}^{\star}}\mathbb{P}_{\eta}[\eta_{N}(\theta_{N,\,2}s)\notin\mathcal{E}_{N}^{\star}]=0.\label{ConvFDDeq}
\end{equation}
Then, the rescaled original process $\eta_{N}(\theta_{N,\,2}\cdot)$
converges to $X_{\textup{second}}(\cdot)$ in the sense of finite-dimensional
marginal distributions, i.e., for all $0\le t_{1}<\cdots<t_{k}$,
$i\in\llbracket1,\,\kappa_{\star}\rrbracket$, $n\in\llbracket1,\,\mathfrak{n}(i)\rrbracket$,
and $A_{1},\,\dots,\,A_{k}\subseteq\llbracket1,\,\kappa_{\star}\rrbracket$,
it holds that 
\begin{align*}
\lim_{N\to\infty} & \mathbb{P}_{\xi^{x_{i,\,n}}}\big[\eta_{N}(\theta_{N,\,2}t_{1})\in\mathcal{E}_{N}^{(2)}(A_{1}),\ \dots,\ \eta_{N}(\theta_{N,\,2}t_{k})\in\mathcal{E}_{N}^{(2)}(A_{k})\big]\\
 & =\mathbf{P}_{i}[X_{\textup{second}}(t_{1})\in A_{1},\ \dots,\ X_{\textup{second}}(t_{k})\in A_{k}],
\end{align*}
where $\mathbf{P}_{i}$ denotes the law of $X_{\textup{second}}(\cdot)$
starting at $i$.
\end{prop}

The remainder of this study is organized as follows. In Section \ref{s-HT},
we provide some preliminaries regarding hitting times on the tubes.
These are used in Sections \ref{s-LBspl} and \ref{s-LB}. Subsequently,
in Sections \ref{s-UBspl} and \ref{s-LBspl}, we calculate the upper
and lower bounds for the capacities, respectively, in the simple case
of Theorem \ref{mainspl}. This procedure is performed by the variational
principles given in Theorems \ref{Diri-Thom} and \ref{genThom}.
In Sections \ref{s-UB} and \ref{s-LB} we provide the estimate of
the capacities in the general case of Theorem \ref{main}. Then, we
prove the condition \eqref{H1} in the general case in Section \ref{s-H1}.
Finally, in Section \ref{s-PMT}, we use the estimates given in Propositions
\ref{B-L conv}, and \ref{ConvFDD} to prove our main result, stated
in Theorem \ref{main}. This simultaneously proves Theorem \ref{mainspl}
as well.

\section{\label{s-HT}Hitting Times on Tubes}

We recall crucial results from \cite{KS NRIP}, which provide sharp
estimates of hitting times on the tubes. These are used in Sections
\ref{s-LBspl} and \ref{s-LB} to compute the asymptotic equilibrium
potential (Lemmas \ref{lem6spl} and \ref{lem6}).

To state the results, we first define some relevant subsets of $\mathcal{H}_{N}$.
The notation is mainly inherited from \cite{KS NRIP}. 
\begin{defn}
\label{tubedef}
\begin{enumerate}
\item For every subset $R$ of $S$, define the \emph{$R$-tube} $\mathcal{A}_{N}^{R}$
as
\begin{equation}
\mathcal{A}_{N}^{R}=\{\eta\in\mathcal{H}_{N}:\eta_{x}=0\text{ for all }x\in S\setminus R\}.\label{tubedefeq}
\end{equation}
For example, $\mathcal{A}_{N}^{S}=\mathcal{H}_{N}$ and $\mathcal{A}_{N}^{\{x\}}=\mathcal{E}_{N}^{x}$.
We may write the superscripts of $\mathcal{A}_{N}^{R}$ by the explicit
elements of $R$ without commas.
\item Especially, if $R=\{x,\,y\}$, we write
\[
\mathcal{A}_{N}^{xy}=\mathcal{A}_{N}^{R}=\{\eta\in\mathcal{H}_{N}:\eta_{x}+\eta_{y}=N\}.
\]
\item For $x,\,y\in S$ and $0\le i\le N$, define the configuration $\xi_{i}^{xy}$
by 
\[
(\xi_{i}^{xy})_{z}=\begin{cases}
i & \text{if }z=x,\\
N-i & \text{if }z=y,\\
0 & \text{otherwise},
\end{cases}
\]
such that $\mathcal{A}_{N}^{xy}=\{\xi_{0}^{xy},\,\xi_{1}^{xy},\,\dots,\,\xi_{N}^{xy}\}$.
Note that $\xi_{N}^{xy}=\xi^{x}$ and $\xi_{0}^{xy}=\xi^{y}$.
\item Finally, for $x,\,y\in S$, define 
\[
\widehat{\mathcal{A}}_{N}^{R}=\{\eta\in\mathcal{A}_{N}^{R}:\eta_{x}\ge1\text{ for all }x\in R\}.
\]
Clearly, $\widehat{\mathcal{A}}_{N}^{xy}=\{\xi_{1}^{xy},\,\dots,\,\xi_{N-1}^{xy}\}$
and $\mathcal{A}_{N}^{xy}=\widehat{\mathcal{A}}_{N}^{xy}\cup\{\xi^{x},\,\xi^{y}\}$.
\item Generally, if $R=\{a_{1},\,\dots,\,a_{r}\}$, then we denote by $\xi_{n_{1},\,\dots,\,n_{r-1}}^{a_{1}a_{2}\cdots a_{r}}\in\mathcal{A}_{N}^{R}$
the element which satisfies
\begin{equation}
\xi_{n_{1},\,\dots,\,n_{r-1}}^{a_{1}a_{2}\cdots a_{r}}(a)=\begin{cases}
n_{j} & \text{if }a=a_{j}\text{ for }j\in\llbracket1,\,r-1\rrbracket,\\
N-(n_{1}+\cdots+n_{r-1}) & \text{if }a=a_{r},\\
0 & \text{otherwise}.
\end{cases}\label{e_xi-notation}
\end{equation}
\end{enumerate}
\end{defn}

The tube $\mathcal{A}_{N}^{R}$ with $R=\{x,\,y\}$ has the advantage
that it is an one-dimensional bridge of typical paths between two
valleys, $\xi^{x}$ and $\xi^{y}$. More precisely, the following
estimate from \cite[Lemma 4.7]{KS NRIP} holds.
\begin{lem}
\label{tubetrest}Suppose that $E$ is a subset of the path space
which depends only on the hitting times of subsets of $\mathcal{H}_{N}\setminus\widehat{\mathcal{A}}_{N}^{xy}$.
Moreover, suppose that $x,\,y\in S$ satisfy $r(x,\,y)+r(y,\,x)>0$.
Then, there is a fixed constant $C>0$ such that 
\[
\Big|\mathbb{P}_{\xi_{i}^{xy}}[E]-\frac{r(x,\,y)}{r(x,\,y)+r(y,\,x)}\mathbb{P}_{\xi_{i-1}^{xy}}[E]-\frac{r(y,\,x)}{r(x,\,y)+r(y,\,x)}\mathbb{P}_{\xi_{i+1}^{xy}}[E]\Big|\le C\frac{d_{N}N}{i(N-i)}
\]
for all $1\le i\le N-1$. 
\end{lem}

\begin{rem}
In the above lemma, typical examples of subsets $E$ are the following.
\[
\big\{\tau_{\mathcal{E}_{N}^{x}}<\tau_{\mathcal{E}_{N}^{y}}\big\},\ \big\{\tau_{\mathcal{E}_{N}^{y}}=\tau_{\mathcal{E}_{N}(A)}\big\}\text{ for }A\subseteq S,\text{ and }\big\{\tau_{\mathcal{E}_{N}^{x}}=\tau_{\mathcal{H}_{N}\setminus\widehat{\mathcal{A}}_{N}^{xy}}\big\}.
\]
\end{rem}

Lemma \ref{tubetrest} can be iterated to formulate $\mathbb{P}_{\xi_{i}^{xy}}[E]$,
$1\le i\le N-1$ in terms of the boundary values $\mathbb{P}_{\xi^{x}}[E]$
and $\mathbb{P}_{\xi^{y}}[E]$. This imperatively relies on the fact
that the system is approximated to be one-dimensional.

We conclude this section with the following lemma, which estimates
the equilibrium potential on one-dimensional tubes. \emph{This lemma
is the main ingredient to estimate the divergence of the test flow
in Lemma \ref{lem6spl}.}
\begin{lem}
\label{eqpotest}Suppose that $A$ and $B$ are two disjoint subsets
of $S$. Further, assume that $a\in A$, $b\in B$, and $c\in S$
satisfy 
\[
r(c,\,a)>r(a,\,c)>0\text{ and }r(c,\,b)>r(b,\,c)>0.
\]
Then, we have 
\begin{equation}
\sup_{0\le i\le\lfloor N/2\rfloor}\big|h_{\mathcal{E}_{N}(A),\,\mathcal{E}_{N}(B)}(\xi_{N-i}^{ac})-1\big|=o(1)\label{eqpotesteq1}
\end{equation}
and 
\begin{equation}
\sup_{0\le i\le\lfloor N/2\rfloor}h_{\mathcal{E}_{N}(A),\,\mathcal{E}_{N}(B)}(\xi_{N-i}^{bc})=o(1).\label{eqpotesteq2}
\end{equation}
\end{lem}

\begin{proof}
It must be noticed that $\{\tau_{\mathcal{E}_{N}(A)}<\tau_{\mathcal{E}_{N}(B)}\}$
is a subset of the path space satisfying the assumption of Lemma \ref{tubetrest};
thus, we may apply Lemma \ref{tubetrest} to the equilibrium potential
$h_{\mathcal{E}_{N}(A),\,\mathcal{E}_{N}(B)}$.

It suffices to prove \eqref{eqpotesteq1} and \eqref{eqpotesteq2}
for $1\le i\le\lfloor\frac{N}{2}\rfloor$, as they are trivial for
$i=0$. We abbreviate $h_{\mathcal{E}_{N}(A),\,\mathcal{E}_{N}(B)}$
as $h$. Because $a\in A$ and $b\in B$, we have $h(\xi^{a})=1$
and $h(\xi^{b})=0$. Next, write $q=r(a,\,c)/r(c,\,a)<1$ and
\[
\alpha_{i}=h(\xi_{N-i+1}^{ac})-h(\xi_{N-i}^{ac})\text{ for }1\le i\le N.
\]
Then, Lemma \ref{tubetrest} implies 
\begin{equation}
\Big|\alpha_{i+1}-\frac{1}{q}\alpha_{i}\Big|\le C\frac{d_{N}N}{i(N-i)}\text{ for }1\le i\le N-1.\label{eqpotesteq3}
\end{equation}
Now, fix $1\le i\le\lfloor\frac{N}{2}\rfloor$. Because $h(\xi^{a})-h(\xi^{c})=\alpha_{1}+\cdots+\alpha_{N}$,
we may estimate,
\begin{equation}
\begin{aligned} & \Big|h(\xi^{a})-h(\xi^{c})-\frac{1-q^{N}}{q^{N-i}(1-q^{i})}(\alpha_{1}+\cdots+\alpha_{i})\Big|\\
 & =\frac{1-q}{q^{N-i}-q^{N}}\Big|\sum_{j=1}^{i}\sum_{k=i+1}^{N}(q^{N-j}\alpha_{k}-q^{N-k}\alpha_{j})\Big|.
\end{aligned}
\label{eqpotesteq4}
\end{equation}
Applying \eqref{eqpotesteq3}, the last formula is bounded by 
\[
\frac{1-q}{q^{N-i}-q^{N}}\sum_{j=1}^{i}\sum_{k=i+1}^{N}q^{N-j}\sum_{\ell=j}^{k-1}\frac{Cd_{N}N}{q^{\ell-j}\ell(N-\ell)}.
\]
By simple double counting, this is bounded from above by 
\begin{equation}
\frac{Cd_{N}N}{q^{N-i}-q^{N}}\Big(\sum_{\ell=1}^{i-1}\frac{q^{N-\ell}}{N-\ell}+\sum_{\ell=i}^{N-i}\frac{iq^{N-\ell}}{\ell(N-\ell)}+\sum_{\ell=N-i+1}^{N-1}\frac{q^{N-\ell}}{\ell}\Big).\label{eqpotesteq5}
\end{equation}
From $\alpha_{1}+\cdots+\alpha_{i}=h(\xi^{a})-h(\xi_{N-i}^{ac})$,
by \eqref{eqpotesteq4} and \eqref{eqpotesteq5}, we have
\[
\begin{aligned} & \Big|h(\xi_{N-i}^{ac})-\frac{1-q^{N-i}}{1-q^{N}}h(\xi^{a})-\frac{q^{N-i}-q^{N}}{1-q^{N}}h(\xi^{c})\Big|\\
 & \le2Cd_{N}N\Big(\frac{q^{N-i+1}(1-q)^{-1}}{N-i+1}+\frac{2q^{i}(1-q)^{-1}}{N}+\frac{(1-q)^{-1}}{N-i+1}\Big)\le16C(1-q)^{-1}d_{N}.
\end{aligned}
\]
Because $h(\xi^{a})=1$ and $0\le h(\xi^{c})\le1$, \eqref{eqpotesteq1}
follows. Moreover, by a similar computation, we deduce that 
\[
\Big|h(\xi_{N-i}^{bc})-\frac{1-\tilde{q}^{N-i}}{1-\tilde{q}^{N}}h(\xi^{b})-\frac{\tilde{q}^{N-i}-\tilde{q}^{N}}{1-\tilde{q}^{N}}h(\xi^{c})\Big|\le16C(1-\tilde{q})^{-1}d_{N},
\]
where $\tilde{q}=r(b,\,c)/r(c,\,b)<1$. Because $h(\xi^{b})=0$ and
$0\le h(\xi^{c})\le1$, we have \eqref{eqpotesteq2}.
\end{proof}

\section{\label{s-UBspl}Upper Bound for Capacities: Simple Case}

In this section, we assume Condition \ref{splcond} and establish
the upper bound for $\mathrm{Cap}_{N}(\mathcal{E}_{N}^{x_{1}},\,\mathcal{E}_{N}^{x_{2}})$.
As previously mentioned, this and the succeeding subsections have
most of the mathematical essentials for proving the general main result.
Notions from Subsection \ref{ss-stsspl} are frequently employed. 
\begin{prop}[Upper bound for capacities: Simple case]
\label{UBspl} Under the conditions of Theorem \ref{mainspl}, the
following inequality holds. 
\[
\limsup_{N\to\infty}\frac{N}{d_{N}^{2}}\mathrm{Cap}_{N}(\mathcal{E}_{N}^{x_{1}},\,\mathcal{E}_{N}^{x_{2}})\le\frac{1}{2\mathfrak{R}}.
\]
\end{prop}

\subsection{\label{ss-prenotspl}Preliminary notions}

Let $m_{\star}=\max_{p=1}^{2}m_{\star}(y_{p})<1$ and recall the notation
\eqref{tubedefeq}. For all $N$, we define the following discretized
version of the constant $\mathfrak{R}$ given in \eqref{Tdef}:

\[
\mathfrak{R}^{N}=\sum_{t=1}^{N}\frac{1}{\sum_{p=1}^{2}\frac{(1-m_{\star}(y_{p}))^{-1}}{\frac{N-t}{r(x_{1},\,y_{p})}+\frac{t-1}{r(x_{2},\,y_{p})}}}.
\]
Clearly, we have $N^{-2}\mathfrak{R}^{N}\to\mathfrak{R}$ as $N$
tends to infinity.

The constant $\mathfrak{R}^{N}$ has the shape of an inverse effective
conductance of an electrical network consisted of conductors. In this
sense, $\mathfrak{R}^{N}$ can be regarded as the inverse conductance
of $N$ serially connected conductors ($1\le t\le N$). Moreover,
each conductor can be decomposed into two parralelly connected conductors
($1\le p\le2$), and each of them corresponds to the motion of a particle
from site $x_{1}$ to site $x_{2}$, passing through $y_{p}$, for
$p=1,\,2$. In each individual conductance $(1-m_{\star}(y_{p}))^{-1}(\frac{N-t}{r(x_{1},\,y_{p})}+\frac{t-1}{r(x_{2},\,y_{p})})^{-1}$,
the former term corresponds to the sum of geometric series of ratio
$m_{\star}(y_{p})$, and the latter term corresponds to the serial
connection of particle motions $x_{1}\leftrightarrow y_{p}$ and $x_{2}\leftrightarrow y_{p}$
with conductances $\frac{r(x_{1},\,y_{p})}{N-t}$ and $\frac{r(x_{2},\,y_{p})}{t-1}$,
respectively. These heuristic explanations are rigorously formulated
in the proof of Lemma \ref{lem1spl}.

Moreover, we define 
\begin{equation}
\mathcal{U}_{N}=\bigcup_{p=1}^{2}\mathcal{A}_{N}^{x_{1}y_{p}x_{2}}\text{ and }\mathcal{V}_{N}=\mathcal{H}_{N}\setminus\mathcal{U}_{N}.\label{unvndefspl}
\end{equation}

\subsection{\label{ss-testfcndefspl}Construction of test function $f_{\textup{test}}$}

In this subsection, we define a test function $f=f_{\textup{test}}$
on $\mathcal{H}_{N}$, which approximates the equilibrium potential
$h_{\mathcal{E}_{N}^{x_{1}},\,\mathcal{E}_{N}^{x_{2}}}$. This procedure
presents the \emph{first major achievement} of this article. To this
end, $f$ is constructed in four steps. See Figure \ref{fig3} for
a graphical explanation of this process.
\begin{itemize}
\item First, we define $f$ on $\mathcal{E}_{N}(S)$:
\begin{equation}
f(\xi^{x_{1}})=1\text{ and }f(\xi^{x_{2}})=f(\xi^{y_{1}})=f(\xi^{y_{2}})=0,\label{testfcn1spl}
\end{equation}
so that condition \eqref{Diri-Thomeq} is verified.
\item Second, we define $f$ on $\widehat{\mathcal{A}}_{N}^{x_{i}y_{p}}$
for $1\le i\le2$ and $1\le p\le2$ by
\begin{equation}
f(\xi_{k}^{x_{i}y_{p}})=f(\xi^{x_{i}}),\ 1\le k\le N-1.\label{testfcn2spl}
\end{equation}
\item Next, we define $f$ on the remainder of $\mathcal{U}_{N}$, i.e.,
on $\mathcal{A}_{N}^{x_{1}y_{p}x_{2}}\setminus(\mathcal{A}_{N}^{x_{1}y_{p}}\cup\mathcal{A}_{N}^{x_{2}y_{p}})$
for $1\le p\le2$. The main contribution to the Dirichlet form occurs
in this part. For $k\in\llbracket1,\,N-2\rrbracket$ and $\ell\in\llbracket0,\,N-k-1\rrbracket$,
\begin{equation}
f(\xi_{k,\,\ell}^{x_{1}y_{p}x_{2}})=\frac{K_{p}^{k,\,\ell}}{\mathfrak{R}^{N}},\label{testfcn3spl}
\end{equation}
where for $k\ge0$ and $\ell\ge0$, 
\[
K_{p}^{k,\,\ell}=\sum_{t=1}^{k}\frac{\frac{N-t}{r(x_{1},\,y_{p})}\big/(\frac{N-t}{r(x_{1},\,y_{p})}+\frac{t-1}{r(x_{2},\,y_{p})})}{\sum_{q=1}^{2}\frac{(1-m_{\star}(y_{q}))^{-1}}{\frac{N-t}{r(x_{1},\,y_{q})}+\frac{t-1}{r(x_{2},\,y_{q})}}}+\sum_{t=1}^{k+\ell}\frac{\frac{t-1}{r(x_{2},\,y_{p})}\big/(\frac{N-t}{r(x_{1},\,y_{p})}+\frac{t-1}{r(x_{2},\,y_{p})})}{\sum_{q=1}^{2}\frac{(1-m_{\star}(y_{q}))^{-1}}{\frac{N-t}{r(x_{1},\,y_{q})}+\frac{t-1}{r(x_{2},\,y_{q})}}}.
\]
\item Finally, we define $f$ on $\mathcal{V}_{N}$. Assume $\eta\in\mathcal{V}_{N}$.
There are three types, \textbf{(V1)}, \textbf{(V2)}, and \textbf{(V3)},
denoted by $\mathcal{V}_{N}^{1}$, $\mathcal{V}_{N}^{2}$, and $\mathcal{V}_{N}^{3}$,
respectively, such that
\begin{equation}
\mathcal{V}_{N}=\mathcal{V}_{N}^{1}\cup\mathcal{V}_{N}^{2}\cup\mathcal{V}_{N}^{3}.\label{vndivspl}
\end{equation}

\begin{itemize}
\item[\textbf{(V1)}]  If $\eta_{x_{1}}=0$, then define 
\begin{equation}
f(\eta)=0.\label{testfcn4spl}
\end{equation}
\item[\textbf{(V2)}]  If $\eta_{x_{1}}\ge1$ and $\eta_{x_{2}}=0$, then define 
\begin{equation}
f(\eta)=1.\label{testfcn4-2spl}
\end{equation}
\item[\textbf{(V3)}]  If $\eta_{x_{1}}\ge1$ and $\eta_{x_{2}}\ge1$, we define 
\begin{equation}
f(\eta)=f(\xi_{\eta_{x_{1}}}^{x_{1}x_{2}}).\label{testfcn4-3spl}
\end{equation}
\end{itemize}
\item By the above construction, $0\le f(\eta)\le1$ for all $\eta\in\mathcal{H}_{N}$.
\end{itemize}
\begin{figure}
\includegraphics[width=0.84\textwidth]{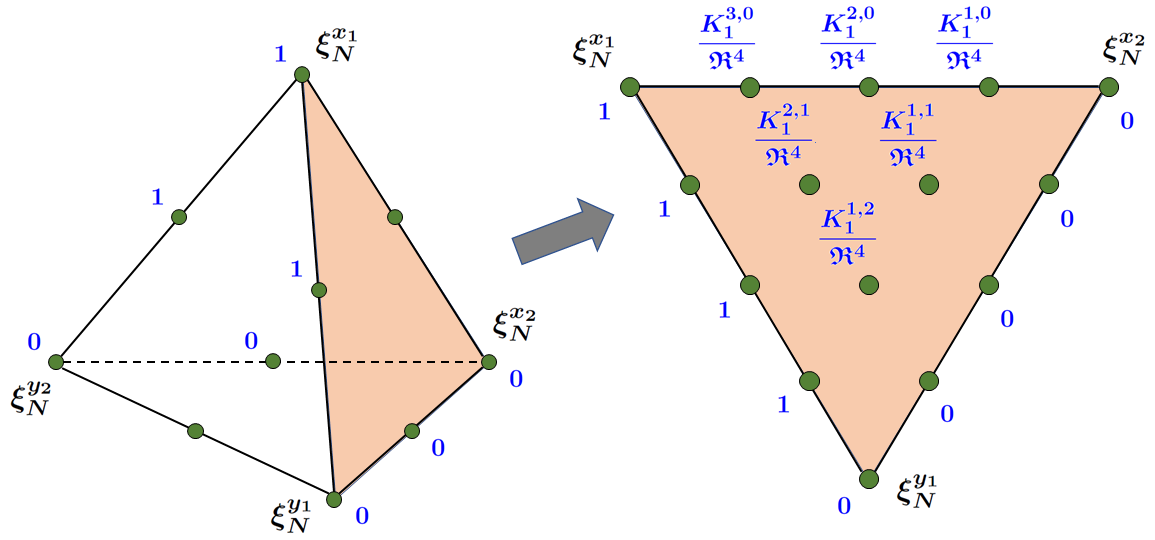}\caption{\label{fig3}(Left) distribution of the test function in a model satisfying
Condition \ref{splcond} with $N=4$. (Right) more detailed landscape
of the test function on the tube $\mathcal{A}_{N}^{x_{1}y_{1}x_{2}}$.}
\end{figure}

Here, we divide the Dirichlet form into four parts:
\[
\begin{aligned}D_{N}(f) & =\sum_{\{\eta,\,\zeta\}\subseteq\mathcal{H}_{N}}\mu_{N}(\eta)\mathbf{q}_{N}(\eta,\,\zeta)\{f(\zeta)-f(\eta)\}^{2}\\
 & =\Sigma_{1}(f)+\Sigma_{2}(f)+\Sigma_{3}(f)+\Sigma_{4}(f).
\end{aligned}
\]
The four summations are defined as follows, according to where the
movement $\eta\leftrightarrow\zeta$ occurs. 
\begin{itemize}
\item The first part $\Sigma_{1}(f)$ consists of movements inside $\mathcal{A}_{N}^{x_{1}y_{p}x_{2}}$
for $1\le p\le2$.
\item The second part $\Sigma_{2}(f)$ consists of movements between the
set differences of $\mathcal{A}_{N}^{x_{1}y_{1}x_{2}}$ and $\mathcal{A}_{N}^{x_{1}y_{2}x_{2}}$.
\item The third part $\Sigma_{3}(f)$ consists of movements between $\mathcal{U}_{N}$
and $\mathcal{V}_{N}$.
\item The last part $\Sigma_{4}(f)$ consists of movements inside $\mathcal{V}_{N}$.
\end{itemize}
From \eqref{unvndefspl}, the above four members are disjoint, and
they characterize $D_{N}(f)$ completely. As shown below, $\Sigma_{1}(f)$
is the main contribution to $D_{N}(f)$, whereas the other summations
vanish (compared to $\Sigma_{1}(f)$) as $N$ tends to infinity.

\subsection{Main contribution of Dirichlet form}

In this subsection, we calculate the main contribution to the Dirichlet
form, which is provided by $\Sigma_{1}(f)$. This is executed in Lemma
\ref{lem1spl}. 
\begin{lem}
\label{lem1spl}Under the conditions of Theorem \ref{mainspl}, it
holds that 
\[
\Sigma_{1}(f)\le\frac{d_{N}^{2}}{2N}\Big[\frac{1}{\mathfrak{R}}+O\Big(\frac{1}{N}\Big)\Big].
\]
\end{lem}

\begin{proof}
To calculate $\Sigma_{1}(f)$, we write down all movements inside
$\mathcal{A}_{N}^{x_{1}y_{p}x_{2}}$ and sum it up for $1\le p\le2$.
More precisely, 
\[
\Sigma_{1}(f)=\sum_{p=1}^{2}\sum_{\eta\in\mathcal{A}_{N}^{x_{1}y_{p}x_{2}}}\sum_{i=1}^{2}\mu_{N}(\eta)\mathbf{q}_{N}(\eta,\,\sigma^{x_{i},\,y_{p}}\eta)\{f(\sigma^{x_{i},\,y_{p}}\eta)-f(\eta)\}^{2}.
\]
There are no overlaps because $r(x_{1},\,x_{2})=r(x_{2},\,x_{1})=0$.
By \eqref{muneq1} and \eqref{munprop}, the right-hand side is asymptotically
equal to
\begin{equation}
\begin{aligned}\frac{d_{N}N}{2}\sum_{p=1}^{2}\sum_{\ell=0}^{N-1}m_{\star}(y_{p})^{\ell}\Big[ & \sum_{k=1}^{N-\ell}w_{N}(N-\ell-k)r(x_{1},\,y_{p})\{f(\xi_{k-1,\,\ell+1}^{x_{1}y_{p}x_{2}})-f(\xi_{k,\,\ell}^{x_{1}y_{p}x_{2}})\}^{2}\\
 & +\sum_{k=0}^{N-\ell-1}w_{N}(k)r(x_{2},\,y_{p})\{f(\xi_{k,\,\ell+1}^{x_{1}y_{p}x_{2}})-f(\xi_{k,\,\ell}^{x_{1}y_{p}x_{2}})\}^{2}\Big].
\end{aligned}
\label{lem1eq1spl}
\end{equation}
Here, we fix $1\le p\le2$ and divide the range $\{0\le\ell\le N-1\}$
into $\{\ell>\lfloor\frac{N}{2}\rfloor\}$ and $\{\ell\le\lfloor\frac{N}{2}\rfloor\}$.
First, using \eqref{munprop} and summing up the geometric series
with respect to $m_{\star}(y_{p})$, summation in the first range
$\{\ell>\lfloor\frac{N}{2}\rfloor\}$ is easily bounded from above
by
\begin{equation}
Cd_{N}N\sum_{\ell>\lfloor N/2\rfloor}m_{\star}(y_{p})^{\ell}=o\Big(\frac{d_{N}^{2}}{N}\Big).\label{lem1eq2spl}
\end{equation}
Second, we calculate summation in the range $\{\ell\le\lfloor\frac{N}{2}\rfloor\}$.
By \eqref{testfcn2spl}, we discard the movements inside $\mathcal{A}_{N}^{x_{1}y_{p}}$
and $\mathcal{A}_{N}^{x_{2}y_{p}}$. Hence, we rewrite this summation
as
\begin{equation}
\begin{aligned}\sum_{\ell=0}^{\lfloor N/2\rfloor}m_{\star}(y_{p})^{\ell}\Big[ & \sum_{k=2}^{N-\ell-1}w_{N}(N-\ell-k)r(x_{1},\,y_{p})\{f(\xi_{k-1,\,\ell+1}^{x_{1}y_{p}x_{2}})-f(\xi_{k,\,\ell}^{x_{1}y_{p}x_{2}})\}^{2}\\
 & +\sum_{k=1}^{N-\ell-2}w_{N}(k)r(x_{2},\,y_{p})\{f(\xi_{k,\,\ell+1}^{x_{1}y_{p}x_{2}})-f(\xi_{k,\,\ell}^{x_{1}y_{p}x_{2}})\}^{2}\Big].\\
 & +w_{N}(N-\ell-1)r(x_{1},\,y_{p})\{f(\xi_{0,\,\ell+1}^{x_{1}y_{p}x_{2}})-f(\xi_{1,\,\ell}^{x_{1}y_{p}x_{2}})\}^{2}\\
 & +w_{N}(N-\ell-1)r(x_{2},\,y_{p})\{f(\xi_{N-\ell-1,\,\ell+1}^{x_{1}y_{p}x_{2}})-f(\xi_{N-\ell-1,\,\ell}^{x_{1}y_{p}x_{2}})\}^{2}\Big].
\end{aligned}
\label{lem1eq3spl}
\end{equation}
By \eqref{munprop} and \eqref{testfcn3spl}, the first line of \eqref{lem1eq3spl}
is asymptotically equivalent to
\[
\frac{d_{N}}{(\mathfrak{R}^{N})^{2}}\sum_{\ell=0}^{\lfloor N/2\rfloor}m_{\star}(y_{p})^{\ell}\sum_{k=2}^{N-\ell-1}\frac{r(x_{1},\,y_{p})}{N-\ell-k}\Big\{\frac{\frac{N-k}{r(x_{1},\,y_{p})}\big/(\frac{N-k}{r(x_{1},\,y_{p})}+\frac{k-1}{r(x_{2},\,y_{p})})}{\sum_{q=1}^{2}\frac{(1-m_{\star}(y_{q}))^{-1}}{\frac{N-k}{r(x_{1},\,y_{q})}+\frac{k-1}{r(x_{2},\,y_{q})}}}\Big\}^{2}.
\]
Dividing $\frac{1}{N-\ell-k}=\frac{1}{N-k}+\frac{\ell}{(N-\ell-k)(N-k)}$
and using $\frac{N-k}{N-\ell-k}\le\ell+1$, the last line is bounded
by
\[
\frac{d_{N}}{(\mathfrak{R}^{N})^{2}}\sum_{\ell=0}^{\lfloor N/2\rfloor}m_{\star}(y_{p})^{\ell}\sum_{k=2}^{N-\ell-1}\Big[\frac{\frac{N-k}{r(x_{1},\,y_{p})}\big/(\frac{N-k}{r(x_{1},\,y_{p})}+\frac{k-1}{r(x_{2},\,y_{p})})^{2}}{\{\sum_{q=1}^{2}\frac{(1-m_{\star}(y_{q}))^{-1}}{\frac{N-k}{r(x_{1},\,y_{q})}+\frac{k-1}{r(x_{2},\,y_{q})}}\}^{2}}+C\ell(\ell+1)\Big].
\]
By the theory of Riemann integration, this is further bounded by
\[
\frac{d_{N}}{(\mathfrak{R}^{N})^{2}}\sum_{\ell=0}^{\lfloor N/2\rfloor}m_{\star}(y_{p})^{\ell}\Big[N^{2}\int_{0}^{1}\frac{\frac{1-t}{r(x_{1},\,y_{p})}\big/(\frac{1-t}{r(x_{1},\,y_{p})}+\frac{t}{r(x_{2},\,y_{p})})^{2}}{\{\sum_{q=1}^{2}\frac{(1-m_{\star}(y_{q}))^{-1}}{\frac{1-t}{r(x_{1},\,y_{q})}+\frac{t}{r(x_{2},\,y_{q})}}\}^{2}}dt+O(N)+CN\ell(\ell+1)\Big].
\]
Calculating the geometric series in $0\le\ell\le\lfloor\frac{N}{2}\rfloor$,
this asymptotically equals
\begin{equation}
\frac{d_{N}}{\mathfrak{R}^{2}N^{2}}\frac{1}{1-m_{\star}(y_{p})}\Big[\int_{0}^{1}\frac{\frac{1-t}{r(x_{1},\,y_{p})}\big/(\frac{1-t}{r(x_{1},\,y_{p})}+\frac{t}{r(x_{2},\,y_{p})})^{2}}{\{\sum_{q=1}^{2}\frac{(1-m_{\star}(y_{q}))^{-1}}{\frac{1-t}{r(x_{1},\,y_{q})}+\frac{t}{r(x_{2},\,y_{q})}}\}^{2}}dt+O\Big(\frac{1}{N}\Big)\Big].\label{lem1eq4spl}
\end{equation}
Similarly, the second line of \eqref{lem1eq3spl} is asymptotically
bounded from above by 
\begin{equation}
\frac{d_{N}}{\mathfrak{R}^{2}N^{2}}\frac{1}{1-m_{\star}(y_{p})}\Big[\int_{0}^{1}\frac{\frac{t}{r(x_{2},\,y_{p})}\big/(\frac{1-t}{r(x_{1},\,y_{p})}+\frac{t}{r(x_{2},\,y_{p})})^{2}}{\{\sum_{q=1}^{2}\frac{(1-m_{\star}(y_{q}))^{-1}}{\frac{1-t}{r(x_{1},\,y_{q})}+\frac{t}{r(x_{2},\,y_{q})}}\}^{2}}dt+O\Big(\frac{1}{N}\Big)\Big].\label{lem1eq5spl}
\end{equation}
The remaining parts of \eqref{lem1eq3spl} are asymptotically equal
to
\begin{equation}
\begin{aligned}d_{N}\sum_{\ell=0}^{\lfloor N/2\rfloor}m_{\star}(y_{p})^{\ell}\Big[ & \frac{r(x_{1},\,y_{p})}{N-\ell-1}\{f(\xi_{0,\,\ell+1}^{x_{1}y_{p}x_{2}})-f(\xi_{1,\,\ell}^{x_{1}y_{p}x_{2}})\}^{2}\\
 & +\frac{r(x_{2},\,y_{p})}{N-\ell-1}\{f(\xi_{N-\ell-1,\,\ell+1}^{x_{1}y_{p}x_{2}})-f(\xi_{N-\ell-1,\,\ell}^{x_{1}y_{p}x_{2}})\}^{2}\Big].
\end{aligned}
\label{lem1eq6spl}
\end{equation}
By \eqref{testfcn2spl} and \eqref{testfcn3spl},
\[
\big|f(\xi_{0,\,\ell+1}^{x_{1}y_{p}x_{2}})-f(\xi_{1,\,\ell}^{x_{1}y_{p}x_{2}})\big|=\frac{K_{p}^{1,\,\ell}}{\mathfrak{R}^{N}}\le\frac{1}{\mathfrak{R}^{N}}\sum_{t=1}^{\ell+1}\frac{1}{\sum_{q=1}^{2}\frac{(1-m_{\star}(y_{q}))^{-1}}{\frac{N-t}{r(x_{1},\,y_{q})}+\frac{t-1}{r(x_{2},\,y_{q})}}},
\]
which is of order $(\ell+1)\times O(1/N)$, and
\[
\big|f(\xi_{N-\ell-1,\,\ell+1}^{x_{1}y_{p}x_{2}})-f(\xi_{N-\ell-1,\,\ell}^{x_{1}y_{p}x_{2}})\big|=\frac{\mathfrak{R}^{N}-K_{p}^{N-\ell-1,\,\ell}}{\mathfrak{R}^{N}}\le\frac{1}{\mathfrak{R}^{N}}\sum_{t=N-\ell}^{N}\frac{1}{\sum_{q=1}^{2}\frac{(1-m_{\star}(y_{q}))^{-1}}{\frac{N-t}{r(x_{1},\,y_{q})}+\frac{t-1}{r(x_{2},\,y_{q})}}},
\]
which is again of order $(\ell+1)\times O(1/N)$. Hence, \eqref{lem1eq6spl}
is bounded from above by 
\begin{equation}
\frac{Cd_{N}}{N}\sum_{\ell=0}^{\lfloor N/2\rfloor}m_{\star}^{\ell}\frac{(\ell+1)^{2}}{N^{2}}=O\Big(\frac{d_{N}}{N^{3}}\Big).\label{lem1eq7spl}
\end{equation}
Therefore, by \eqref{lem1eq4spl}, \eqref{lem1eq5spl}, and \eqref{lem1eq7spl},
we have the following asymptotic upper bound for \eqref{lem1eq3spl}:
\begin{equation}
\frac{d_{N}}{\mathfrak{R}^{2}N^{2}}\Big[\int_{0}^{1}\frac{(1-m_{\star}(y_{p}))^{-1}(\frac{1-t}{r(x_{1},\,y_{p})}+\frac{t}{r(x_{2},\,y_{p})})^{-1}}{\{\sum_{q=1}^{2}\frac{(1-m_{\star}(y_{q}))^{-1}}{\frac{1-t}{r(x_{1},\,y_{q})}+\frac{t}{r(x_{2},\,y_{q})}}\}^{2}}dt+O\Big(\frac{1}{N}\Big)\Big].\label{lem1eq8spl}
\end{equation}
Collecting \eqref{lem1eq1spl}, \eqref{lem1eq2spl}, and \eqref{lem1eq8spl},
and the fact that $d_{N}$ decays subexponentially, $\Sigma_{1}(f)$
has the following asymptotic upper bound:
\[
\begin{aligned} & \frac{d_{N}^{2}}{2\mathfrak{R}^{2}N}\Big[\int_{0}^{1}\frac{\sum_{p=1}^{2}(1-m_{\star}(y_{p}))^{-1}(\frac{1-t}{r(x_{1},\,y_{p})}+\frac{t}{r(x_{2},\,y_{p})})^{-1}}{\{\sum_{q=1}^{2}\frac{(1-m_{\star}(y_{q}))^{-1}}{\frac{1-t}{r(x_{1},\,y_{q})}+\frac{t}{r(x_{2},\,y_{q})}}\}^{2}}dt+O\Big(\frac{1}{N}\Big)\Big]\\
 & =\frac{d_{N}^{2}}{2\mathfrak{R}^{2}N}\Big[\int_{0}^{1}\frac{1}{\sum_{q=1}^{2}\frac{(1-m_{\star}(y_{q}))^{-1}}{\frac{1-t}{r(x_{1},\,y_{q})}+\frac{t}{r(x_{2},\,y_{q})}}}dt+O\Big(\frac{1}{N}\Big)\Big].
\end{aligned}
\]
The integral in the last line is exactly $\mathfrak{R}$. Hence, we
have 
\[
\Sigma_{1}(f)\le\frac{d_{N}^{2}}{2N}\Big[\frac{1}{\mathfrak{R}}+O\Big(\frac{1}{N}\Big)\Big].
\]
The last formula yields our exact expectations.
\end{proof}

\subsection{\label{ss-remDirispl}Remainder of Dirichlet form}

Next, we deal with the remaining terms in the Dirichlet form, $\Sigma_{2}(f)$,
$\Sigma_{3}(f)$, and $\Sigma_{4}(f)$. Lemma \ref{lem2spl} deals
with $\Sigma_{2}(f)$. 
\begin{lem}
\label{lem2spl}Under the conditions of Theorem \ref{mainspl}, it
holds that 
\[
\Sigma_{2}(f)=O\Big(\frac{d_{N}^{3}\log N}{N^{2}}\Big)=o\Big(\frac{d_{N}^{2}}{N}\Big).
\]
\end{lem}

\begin{proof}
Recall that $\Sigma_{2}(f)$ consists of dynamics between the set
differences of $\mathcal{A}_{N}^{x_{1}y_{1}x_{2}}$ and $\mathcal{A}_{N}^{x_{1}y_{2}x_{2}}$.
This happens when a sole particle moves between $y_{1}$ and $y_{2}$.
Precisely,
\begin{equation}
\Sigma_{2}(f)=\sum_{k=0}^{N-1}\mu_{N}(\xi_{k,\,1}^{x_{1}y_{1}x_{2}})d_{N}r(y_{1},\,y_{2})\{f(\xi_{k,\,1}^{x_{1}y_{2}x_{2}})-f(\xi_{k,\,1}^{x_{1}y_{1}x_{2}})\}^{2}.\label{lem2eq1spl}
\end{equation}
If $k=0$ or $N-1$, then $f(\xi_{k,\,1}^{x_{1}y_{1}x_{2}})=f(\xi_{k,\,1}^{x_{1}y_{2}x_{2}})$
by \eqref{testfcn2spl}. If $1\le k\le N-2$, then by \eqref{testfcn3spl},
\[
f(\xi_{k,\,1}^{x_{1}y_{2}x_{2}})-f(\xi_{k,\,1}^{x_{1}y_{1}x_{2}})=\frac{\frac{k}{r(x_{2},\,y_{2})}\big/(\frac{N-k-1}{r(x_{1},\,y_{2})}+\frac{k}{r(x_{2},\,y_{2})})-\frac{k}{r(x_{2},\,y_{1})}\big/(\frac{N-k-1}{r(x_{1},\,y_{1})}+\frac{k}{r(x_{2},\,y_{1})})}{\mathfrak{R}^{N}\sum_{q=1}^{2}\frac{(1-m_{\star}(y_{q}))^{-1}}{\frac{N-k-1}{r(x_{1},\,y_{q})}+\frac{k}{r(x_{2},\,y_{q})}}},
\]
which is $O(1/N)$. Thus, \eqref{lem2eq1spl} is bounded by 
\[
C\sum_{k=1}^{N-2}\frac{Nd_{N}^{3}}{k(N-k-1)}\times\frac{1}{N^{2}}=O\Big(\frac{d_{N}^{3}\log N}{N^{2}}\Big).
\]
This concludes the proof. 
\end{proof}
Next, we consider $\Sigma_{3}(f)$. 
\begin{lem}
\label{lem3spl}Under the conditions of Theorem \ref{mainspl}, it
holds that 
\[
\Sigma_{3}(f)=O\Big(d_{N}^{2}m_{\star}^{N}+\frac{d_{N}^{3}}{N}\log N\Big)=o\Big(\frac{d_{N}^{2}}{N}\Big).
\]
\end{lem}

\begin{proof}
We can formulate
\[
\Sigma_{3}(f)=\sum_{\eta\in\mathcal{U}_{N}}\sum_{\zeta\in\mathcal{V}_{N}}\mu_{N}(\eta)\mathbf{q}_{N}(\eta,\,\zeta)\{f(\zeta)-f(\eta)\}^{2}.
\]
We divide the summation into three cases, depending on which subset
$\eta$ belongs to.\medskip{}

\noindent \textbf{(Case 1)} $\eta\in\mathcal{A}_{N}^{x_{1}x_{2}}$:
In this case, there are no particle movements with $\zeta\in\mathcal{V}_{N}$.\medskip{}

\noindent \textbf{(Case 2)} $\eta\in\mathcal{A}_{N}^{x_{i}y_{p}}\setminus\mathcal{E}_{N}^{x_{i}}$
for some $1\le i\le2$ and $1\le p\le2$: We divide again according
to types of the particle movement.\medskip{}
\begin{itemize}
\item \textbf{(Case 2-1)} $\zeta=\sigma^{y_{p},\,y_{q}}\eta$, where $q\in\{1,\,2\}\setminus\{p\}$:
The corresponding summation becomes
\[
\sum_{i=1}^{2}\sum_{p=1}^{2}\sum_{\ell=1}^{N}\frac{w_{N}(N-\ell)w_{N}(\ell)}{Z_{N}}m_{\star}(y_{p})^{\ell}\ell d_{N}r(y_{p},\,y_{q})\{f(\xi_{N-\ell,\,\ell-1}^{x_{i}y_{p}y_{q}})-f(\xi_{N-\ell}^{x_{i}y_{p}})\}^{2}.
\]
By \eqref{testfcn1spl}, \eqref{testfcn2spl}, \eqref{testfcn4spl}
and \eqref{testfcn4-2spl}, $f(\xi_{N-\ell,\,\ell-1}^{x_{i}y_{p}y_{q}})=f(\xi_{N-\ell}^{x_{i}y_{p}})$
for all $i$, $p$, and $\ell$. Therefore, the summation is $0$
in this case.\medskip{}
\item \textbf{(Case 2-2)} $\zeta=\sigma^{x_{i},\,y_{q}}\eta$, where $q\in\{1,\,2\}\setminus\{p\}$:
The corresponding summation becomes
\[
\sum_{i=1}^{2}\sum_{p=1}^{2}\sum_{\ell=1}^{N-1}\frac{w_{N}(N-\ell)w_{N}(\ell)}{Z_{N}}m_{\star}(y_{p})^{\ell}(N-\ell)d_{N}r(x_{i},\,y_{q})\{f(\xi_{N-\ell-1,\,\ell}^{x_{i}y_{p}y_{q}})-f(\xi_{N-\ell}^{x_{i}y_{p}})\}^{2}.
\]
This vanishes unless $\ell=N-1$, in which case it becomes $O(d_{N}^{2}m_{\star}^{N})$.
Concluding, \textbf{(Case 2)} yields a contribution $O(d_{N}^{2}m_{\star}^{N})$.\medskip{}
\end{itemize}
\textbf{(Case 3)} $\eta\in\widehat{\mathcal{A}}_{N}^{x_{1}y_{p}x_{2}}$
for some $1\le p\le2$:

In this case, we can write the summation as
\begin{equation}
\begin{aligned} & \sum_{p=1}^{2}\sum_{\ell=2}^{N-2}\sum_{k=1}^{N-\ell-1}\mu_{N}(\xi_{k,\,\ell}^{x_{1}y_{p}x_{2}})\cdot\ell d_{N}r(y_{p},\,y_{q})\{f(\xi_{k,\,\ell-1,\,1}^{x_{1}y_{p}y_{q}x_{2}})-f(\xi_{k,\,\ell}^{x_{1}y_{p}x_{2}})\}^{2}\\
 & +\sum_{p=1}^{2}\sum_{\ell=1}^{N-2}\sum_{k=1}^{N-\ell-1}\mu_{N}(\xi_{k,\,\ell}^{x_{1}y_{p}x_{2}})\cdot kd_{N}r(x_{1},\,y_{q})\{f(\xi_{k-1,\,\ell,\,1}^{x_{1}y_{p}y_{q}x_{2}})-f(\xi_{k,\,\ell}^{x_{1}y_{p}x_{2}})\}^{2}\\
 & +\sum_{p=1}^{2}\sum_{\ell=1}^{N-2}\sum_{k=1}^{N-\ell-1}\mu_{N}(\xi_{k,\,\ell}^{x_{1}y_{p}x_{2}})\cdot(N-\ell-k)d_{N}r(x_{2},\,y_{q})\{f(\xi_{k,\,\ell,\,1}^{x_{1}y_{p}y_{q}x_{2}})-f(\xi_{k,\,\ell}^{x_{1}y_{p}x_{2}})\}^{2},
\end{aligned}
\label{lem3eq1spl}
\end{equation}
where in the summation, $q\in\{1,\,2\}\setminus\{p\}$. For the first
line in \eqref{lem3eq1spl}, it is bounded by
\[
C\sum_{p=1}^{2}\sum_{\ell=2}^{N-2}\sum_{k=1}^{N-\ell-1}\frac{Nd_{N}^{3}m_{\star}^{\ell}}{k(N-\ell-k)}\cdot\{f(\xi_{k,\,\ell-1,\,1}^{x_{1}y_{p}y_{q}x_{2}})-f(\xi_{k,\,\ell}^{x_{1}y_{p}x_{2}})\}^{2}.
\]
By \eqref{testfcn4-3spl}, $f(\xi_{k,\,\ell-1,\,1}^{x_{1}y_{p}y_{q}x_{2}})=f(\xi_{k}^{x_{1}x_{2}})$,
and thus by \eqref{testfcn3spl}, this is bounded by
\[
C\sum_{\ell=2}^{N-2}\sum_{k=1}^{N-\ell-1}\frac{Nd_{N}^{3}m_{\star}^{\ell}}{k(N-\ell-k)}\cdot\frac{\ell^{2}}{N^{2}}=O\Big(\frac{d_{N}^{3}}{N^{2}}\log N\Big).
\]
For the second line in \eqref{lem3eq1spl}, it is bounded by
\[
C\sum_{p=1}^{2}\sum_{\ell=1}^{N-2}\sum_{k=1}^{N-\ell-1}\frac{Nd_{N}^{3}m_{\star}^{\ell}}{\ell(N-\ell-k)}\cdot\{f(\xi_{k-1,\,\ell,\,1}^{x_{1}y_{p}y_{q}x_{2}})-f(\xi_{k,\,\ell}^{x_{1}y_{p}x_{2}})\}^{2}.
\]
Similarly, this is bounded by
\[
C\sum_{p=1}^{2}\sum_{\ell=2}^{N-2}\sum_{k=1}^{N-\ell-1}\frac{Nd_{N}^{3}m_{\star}^{\ell}}{\ell(N-\ell-n)}\cdot\frac{(\ell+1)^{2}}{N^{2}}=O\Big(\frac{d_{N}^{3}}{N}\log N\Big).
\]
Similarly, the third line in \eqref{lem3eq1spl} is also bounded by
$O(\frac{d_{N}^{3}}{N}\log N)$. Collecting all cases, we conclude
that
\[
\Sigma_{3}(f)=O\Big(d_{N}^{2}m_{\star}^{N}+\frac{d_{N}^{3}}{N}\log N\Big).
\]
\end{proof}
Finally, we deal with $\Sigma_{4}(f)$.
\begin{lem}
\label{lem4spl}Under the conditions of Theorem \ref{mainspl}, it
holds that 
\[
\Sigma_{4}(f)=O(d_{N}^{2}m_{\star}^{N})+O\Big(\frac{d_{N}^{3}}{N}\log N\Big)=o\Big(\frac{d_{N}^{2}}{N}\Big).
\]
\end{lem}

\begin{proof}
By definition, we have
\[
\Sigma_{4}(f)=\frac{1}{2}\sum_{\eta,\,\zeta\in\mathcal{V}_{N}}\mu_{N}(\eta)\mathbf{q}_{N}(\eta,\,\zeta)\{f(\zeta)-f(\eta)\}^{2}.
\]
By \eqref{vndivspl}, we divide the summation in $\eta,\,\zeta\in\mathcal{V}_{N}$
by where $\eta\leftrightarrow\zeta$ happens.\medskip{}

\noindent \textbf{(Case 1)} $\mathcal{V}_{N}^{1}\leftrightarrow\mathcal{V}_{N}^{1}$
or $\mathcal{V}_{N}^{2}\leftrightarrow\mathcal{V}_{N}^{2}$: $f$
remains unchanged by \eqref{testfcn4spl} and \eqref{testfcn4-2spl}.\medskip{}

\noindent \textbf{(Case 2)} $\mathcal{V}_{N}^{1}\leftrightarrow\mathcal{V}_{N}^{2}$:
There are at least $N-1$ particles in $\{y_{1},\,y_{2}\}$, so the
summation behaves as $O(d_{N}^{2}m_{\star}^{N})$.\medskip{}

\noindent \textbf{(Case 3)} $\mathcal{V}_{N}^{1}\leftrightarrow\mathcal{V}_{N}^{3}$:
This is impossible.\medskip{}

\noindent \textbf{(Case 4)} $\mathcal{V}_{N}^{2}\leftrightarrow\mathcal{V}_{N}^{3}$:
This case can be bounded by
\[
\begin{aligned} & C\sum_{\ell=2}^{N-2}\sum_{\ell'=1}^{N-\ell-1}\frac{Nd_{N}^{3}}{(N-\ell-\ell')\ell'}m_{\star}^{\ell+\ell'}\{1-f(\xi_{N-\ell-\ell'}^{x_{1}x_{2}})\}^{2}\\
 & \le C\sum_{\ell=2}^{N-2}\sum_{\ell'=1}^{N-\ell-1}\frac{Nd_{N}^{3}}{(N-\ell-\ell')\ell'}m_{\star}^{\ell+\ell'}\cdot\frac{(\ell+\ell')^{2}}{N^{2}}=O\Big(\frac{d_{N}^{3}}{N^{2}}\Big).
\end{aligned}
\]
\medskip{}

\noindent \textbf{(Case 5)} $\mathcal{V}_{N}^{3}\leftrightarrow\mathcal{V}_{N}^{3}$:
By \eqref{testfcn4-3spl}, the summation becomes
\[
\sum_{i=1}^{2}\sum_{p=1}^{2}\sum_{k=2}^{N-3}\sum_{\ell=1}^{N-k-2}\sum_{\ell'=1}^{N-k-\ell-1}\mu_{N}(\xi_{k,\,\ell,\,\ell'}^{x_{i}y_{p}y_{q}x_{j}})\cdot k(d_{N}+\ell)r(x_{i},\,y_{p})\{f(\xi_{k-1,\,\ell+1,\,\ell'}^{x_{i}y_{p}y_{q}x_{j}})-f(\xi_{k,\,\ell,\,\ell'}^{x_{i}y_{p}y_{q}x_{j}})\}^{2},
\]
where $\{i,\,j\}=\{p,\,q\}=\{1,\,2\}$. This is bounded by
\[
\begin{aligned} & C\sum_{k=2}^{N-3}\sum_{\ell=1}^{N-k-2}\sum_{\ell'=1}^{N-k-\ell-1}\frac{Nd_{N}^{3}}{\ell'(N-k-\ell-\ell')}m_{\star}^{\ell+\ell'}\cdot\{f(\xi_{k-1}^{x_{1}x_{2}})-f(\xi_{k}^{x_{1}x_{2}})\}^{2}\\
 & \le C\sum_{k=2}^{N-3}\sum_{\ell=1}^{N-k-2}\sum_{\ell'=1}^{N-k-\ell-1}\frac{Nd_{N}^{3}}{\ell'(N-k-\ell-\ell')}m_{\star}^{\ell+\ell'}\cdot\frac{1}{N^{2}}=O\Big(\frac{d_{N}^{3}}{N}\log N\Big).
\end{aligned}
\]
This concludes the proof of Lemma \ref{lem4spl}.
\end{proof}

\subsection{Proof of Proposition \ref{UBspl}}

Thus, we are in the position to prove Proposition \ref{UBspl}. 
\begin{proof}[Proof of Proposition \ref{UBspl}]
 By Lemmas \ref{lem1spl}, \ref{lem2spl}, \ref{lem3spl}, and \ref{lem4spl},
\[
D_{N}(f_{\textup{test}})\le\frac{d_{N}^{2}}{2N\mathfrak{R}}+O\Big(\frac{d_{N}^{2}}{N^{2}}\Big)+O(d_{N}^{2}m_{\star}^{N})+O\Big(\frac{d_{N}^{3}}{N}\log N\Big).
\]
Sending $N\to\infty$, as $\lim_{N\to\infty}d_{N}\log N=0$ and $d_{N}$
decays subexponentially, we have 
\[
\limsup_{N\to\infty}\frac{N}{d_{N}^{2}}D_{N}(f_{\textup{test}})\le\frac{1}{2\mathfrak{R}}.
\]
Therefore, by Theorem \ref{Diri-Thom}, we obtain the desired result.
\end{proof}

\section{\label{s-LBspl}Lower Bound for Capacities: Simple Case}

In this section, we assume Condition \ref{splcond} and establish
the lower bound for $\mathrm{Cap}_{N}(\mathcal{E}_{N}^{x_{1}},\,\mathcal{E}_{N}^{x_{2}})$.
Once more, we recall the notions from Subsection \ref{ss-stsspl}.
The following proposition is our main objective. 
\begin{prop}[Lower bound for capacities: Simple case]
\label{LBspl} Under the conditions of Theorem \ref{mainspl}, the
following inequality holds. 
\begin{equation}
\liminf_{N\to\infty}\frac{N}{d_{N}^{2}}\mathrm{Cap}_{N}(\mathcal{E}_{N}^{x_{1}},\,\mathcal{E}_{N}^{x_{2}})\ge\frac{1}{2\mathfrak{R}}.\label{LBeqspl}
\end{equation}
\end{prop}

As mentioned after Remark \ref{Thomrmk}, the procedure involves the
use of a test flow, which is in some sense close to $c\Phi_{f_{\textup{test}}}$,
as $\psi$ in Theorem \ref{genThom}. The main difficulty is to find
a suitable flow, such that
\[
\sum_{\eta\in\mathcal{H}_{N}}h_{\mathcal{E}_{N}^{x_{1}},\,\mathcal{E}_{N}^{x_{2}}}(\eta)(\mathrm{div}\,\psi)(\eta)
\]
can be easily calculated. Here, the major obstacle is that the exact
values of $h_{\mathcal{E}_{N}^{x_{1}},\,\mathcal{E}_{N}^{x_{2}}}$
are unknown except on the one-dimensional tubes, $\mathcal{A}_{N}^{ab}$
for $a,\,b\in S$, as shown in Section \ref{s-HT}. Thus, the objective
is to find a proper approximating flow $\psi_{\textup{test}}$ whose
divergence can be neglected outside those tubes.

\subsection{\label{ss-testflowdefspl}Construction of test flow $\psi_{\textup{test}}$}

In this subsection, we build the test flow $\psi=\psi_{\textup{test}}$
on $\mathcal{H}_{N}$. As mentioned above, the key here is as follows:
We must construct $\psi$ such that 
\begin{enumerate}
\item the flow norm of $\psi$ is asymptotically equal to $c\Phi_{f_{\textup{test}}}$,
$c\ne0$, 
\item the divergence of $\psi$ can be summed up in the sense of the right-hand
side of \eqref{genThomeq}. 
\end{enumerate}
To overcome both issues, \emph{we modify $\Phi_{f_{\textup{test}}}$
properly, so that the divergence vanishes on $\mathcal{A}_{N}^{x_{1}y_{p}x_{2}}\setminus(\mathcal{A}_{N}^{x_{1}y_{p}}\cup\mathcal{A}_{N}^{x_{2}y_{p}})$:}

First, we define, for $1\le p\le2$, $\ell\in\llbracket0,\,\frac{N}{2}-1\rrbracket$,
and $k\in\llbracket1,\,N-\ell-1\rrbracket$,
\begin{equation}
\psi_{0}(\xi_{k,\,\ell}^{x_{1}y_{p}x_{2}},\,\xi_{k-1,\,\ell+1}^{x_{1}y_{p}x_{2}})=\frac{m_{\star}(y_{p})^{\ell}\big/(\frac{N-\ell-k-1}{r(x_{1},\,y_{p})}+\frac{k+\ell}{r(x_{2},\,y_{p})})}{\mathfrak{R}\sum_{q=1}^{2}\frac{(1-m_{\star}(y_{q}))^{-1}}{\frac{N-\ell-k-1}{r(x_{1},\,y_{q})}+\frac{k+\ell}{r(x_{2},\,y_{q})}}},\label{testflowspl}
\end{equation}
\begin{equation}
\psi_{0}(\xi_{k,\,\ell}^{x_{1}y_{p}x_{2}},\,\xi_{k,\,\ell+1}^{x_{1}y_{p}x_{2}})=\frac{-m_{\star}(y_{p})^{\ell}\big/(\frac{N-\ell-k-1}{r(x_{1},\,y_{p})}+\frac{k+\ell}{r(x_{2},\,y_{p})})}{\mathfrak{R}\sum_{q=1}^{2}\frac{(1-m_{\star}(y_{q}))^{-1}}{\frac{N-\ell-k-1}{r(x_{1},\,y_{q})}+\frac{k+\ell}{r(x_{2},\,y_{q})}}},\label{testflow2spl}
\end{equation}
and $0$ otherwise.

Observe that by the above construction, $(\mathrm{div}\,\psi_{0})(\xi^{x_{i}})=0$
for $1\le i\le2$ and $(\mathrm{div}\,\psi_{0})(\eta)=0$ for all
$\eta\in\mathcal{V}_{N}$.

However, it holds that $(\mathrm{div}\,\psi_{0})(\xi_{k,\,\ell}^{x_{1}y_{p}x_{2}})\ne0$
for $\ell\in\llbracket1,\,\frac{N}{2}\rrbracket$ and $k\in\llbracket0,\,N-\ell\rrbracket$.
\emph{We overcome this issue by adding correction flows to $\psi_{0}$
and make the divergence to be zero.}

Before the exact definition, we calculate the non-zero term $(\mathrm{div}\,\psi_{0})(\xi_{k,\,\ell}^{x_{1}y_{p}x_{2}})$.
We define, for $k\in\llbracket1,\,N-1\rrbracket$,
\begin{equation}
\mathfrak{A}(N,\,k):=\frac{1}{\mathfrak{R}}\frac{N-1}{(\frac{\frac{N-k}{r(x_{1},\,y_{2})}+\frac{k-1}{r(x_{2},\,y_{2})}}{1-m_{\star}(y_{1})}+\frac{\frac{N-k}{r(x_{1},\,y_{1})}+\frac{k-1}{r(x_{2},\,y_{1})}}{1-m_{\star}(y_{2})})(\frac{\frac{N-k-1}{r(x_{1},\,y_{2})}+\frac{k}{r(x_{2},\,y_{2})}}{1-m_{\star}(y_{1})}+\frac{\frac{N-k-1}{r(x_{1},\,y_{1})}+\frac{k}{r(x_{2},\,y_{1})}}{1-m_{\star}(y_{2})})}.\label{ANkdef}
\end{equation}
Then, by the estimate
\[
\frac{N-k}{r(x_{1},\,y_{2})}+\frac{k-1}{r(x_{2},\,y_{2})}\ge\frac{N-k}{C}+\frac{k-1}{C}=\frac{N-1}{C}
\]
and three additional similar bounds, it is straightforward that
\begin{equation}
\mathfrak{A}(N,\,k)\le\frac{C}{N}.\label{ANkbound}
\end{equation}
The next lemma represents $(\mathrm{div}\,\psi_{0})(\xi_{k,\,\ell}^{x_{1}y_{p}x_{2}})$
in terms of $\mathfrak{A}(N,\,k+\ell)$.
\begin{lem}
For $1\le p\le2$, $\ell\in\llbracket1,\,\frac{N}{2}\rrbracket$,
and $k\in\llbracket1,\,N-\ell-1\rrbracket$, we have
\begin{equation}
(\mathrm{div}\,\psi_{0})(\xi_{k,\,\ell}^{x_{1}y_{p}x_{2}})=\frac{m_{\star}(y_{p})^{\ell-1}}{1-m_{\star}(y_{s})}\Big[\frac{1}{r(x_{1},\,y_{s})r(x_{2},\,y_{p})}-\frac{1}{r(x_{2},\,y_{s})r(x_{1},\,y_{p})}\Big]\mathfrak{A}(N,\,k+\ell),\label{divpsi0}
\end{equation}
where $\{p,\,s\}=\{1,\,2\}$.
\end{lem}

\begin{proof}
By \eqref{testflowspl} and \eqref{testflow2spl}, $(\mathrm{div}\,\psi_{0})(\xi_{k,\,\ell}^{x_{1}y_{p}x_{2}})$
equals
\[
\begin{aligned} & \psi_{0}(\xi_{k,\,\ell}^{x_{1}y_{p}x_{2}},\,\xi_{k+1,\,\ell-1}^{x_{1}y_{p}x_{2}})+\psi_{0}(\xi_{k,\,\ell}^{x_{1}y_{p}x_{2}},\,\xi_{k,\,\ell-1}^{x_{1}y_{p}x_{2}})\\
 & =-\frac{m_{\star}(y_{p})^{\ell-1}\big/(\frac{N-\ell-k-1}{r(x_{1},\,y_{p})}+\frac{k+\ell}{r(x_{2},\,y_{p})})}{\mathfrak{R}\sum_{q=1}^{2}\frac{(1-m_{\star}(y_{q}))^{-1}}{\frac{N-\ell-k-1}{r(x_{1},\,y_{q})}+\frac{k+\ell}{r(x_{2},\,y_{q})}}}+\frac{m_{\star}(y_{p})^{\ell-1}\big/(\frac{N-\ell-k}{r(x_{1},\,y_{p})}+\frac{k+\ell-1}{r(x_{2},\,y_{p})})}{\mathfrak{R}\sum_{q=1}^{2}\frac{(1-m_{\star}(y_{q}))^{-1}}{\frac{N-\ell-k}{r(x_{1},\,y_{q})}+\frac{k+\ell-1}{r(x_{2},\,y_{q})}}},
\end{aligned}
\]
where the first line holds since the other two flow values cancel
out with each other. Noting that $\{p,\,s\}=\{1,\,2\}$, we rearrange
the right-hand side as $m_{\star}(y_{p})^{\ell-1}/\mathfrak{R}$ times
\[
-\frac{\frac{N-\ell-k-1}{r(x_{1},\,y_{s})}+\frac{k+\ell}{r(x_{2},\,y_{s})}}{\frac{\frac{N-\ell-k-1}{r(x_{1},\,y_{s})}+\frac{k+\ell}{r(x_{2},\,y_{s})}}{1-m_{\star}(y_{p})}+\frac{\frac{N-\ell-k-1}{r(x_{1},\,y_{p})}+\frac{k+\ell}{r(x_{2},\,y_{p})}}{1-m_{\star}(y_{s})}}+\frac{\frac{N-\ell-k}{r(x_{1},\,y_{s})}+\frac{k+\ell-1}{r(x_{2},\,y_{s})}}{\frac{\frac{N-\ell-k}{r(x_{1},\,y_{s})}+\frac{k+\ell-1}{r(x_{2},\,y_{s})}}{1-m_{\star}(y_{p})}+\frac{\frac{N-\ell-k}{r(x_{1},\,y_{p})}+\frac{k+\ell-1}{r(x_{2},\,y_{p})}}{1-m_{\star}(y_{s})}}.
\]
Reducing to a common denominator, the last display equals
\begin{equation}
\begin{aligned} & -\Big(\frac{N-\ell-k-1}{r(x_{1},\,y_{s})}+\frac{k+\ell}{r(x_{2},\,y_{s})}\Big)\Big(\frac{\frac{N-\ell-k}{r(x_{1},\,y_{s})}+\frac{k+\ell-1}{r(x_{2},\,y_{s})}}{1-m_{\star}(y_{p})}+\frac{\frac{N-\ell-k}{r(x_{1},\,y_{p})}+\frac{k+\ell-1}{r(x_{2},\,y_{p})}}{1-m_{\star}(y_{s})}\Big)\\
 & +\Big(\frac{N-\ell-k}{r(x_{1},\,y_{s})}+\frac{k+\ell-1}{r(x_{2},\,y_{s})}\Big)\Big(\frac{\frac{N-\ell-k-1}{r(x_{1},\,y_{s})}+\frac{k+\ell}{r(x_{2},\,y_{s})}}{1-m_{\star}(y_{p})}+\frac{\frac{N-\ell-k-1}{r(x_{1},\,y_{p})}+\frac{k+\ell}{r(x_{2},\,y_{p})}}{1-m_{\star}(y_{s})}\Big)
\end{aligned}
\label{divcalculation1}
\end{equation}
divided by
\[
\Big(\frac{\frac{N-\ell-k-1}{r(x_{1},\,y_{s})}+\frac{k+\ell}{r(x_{2},\,y_{s})}}{1-m_{\star}(y_{p})}+\frac{\frac{N-\ell-k-1}{r(x_{1},\,y_{p})}+\frac{k+\ell}{r(x_{2},\,y_{p})}}{1-m_{\star}(y_{s})}\Big)\Big(\frac{\frac{N-\ell-k}{r(x_{1},\,y_{s})}+\frac{k+\ell-1}{r(x_{2},\,y_{s})}}{1-m_{\star}(y_{p})}+\frac{\frac{N-\ell-k}{r(x_{1},\,y_{p})}+\frac{k+\ell-1}{r(x_{2},\,y_{p})}}{1-m_{\star}(y_{s})}\Big).
\]
Thus, according to \eqref{ANkdef}, it remains to prove that \eqref{divcalculation1}
equals
\[
\frac{N-1}{1-m_{\star}(y_{s})}\times\Big[\frac{1}{r(x_{1},\,y_{s})r(x_{2},\,y_{p})}-\frac{1}{r(x_{2},\,y_{s})r(x_{1},\,y_{p})}\Big].
\]
In \eqref{divcalculation1}, the two terms involving $1-m_{\star}(y_{p})$
cancel out with each other. Thus, \eqref{divcalculation1} becomes
$(1-m_{\star}(y_{s}))^{-1}$ times
\[
\begin{aligned} & -\Big(\frac{N-\ell-k-1}{r(x_{1},\,y_{s})}+\frac{k+\ell}{r(x_{2},\,y_{s})}\Big)\Big(\frac{N-\ell-k}{r(x_{1},\,y_{p})}+\frac{k+\ell-1}{r(x_{2},\,y_{p})}\Big)\\
 & +\Big(\frac{N-\ell-k}{r(x_{1},\,y_{s})}+\frac{k+\ell-1}{r(x_{2},\,y_{s})}\Big)\Big(\frac{N-\ell-k-1}{r(x_{1},\,y_{p})}+\frac{k+\ell}{r(x_{2},\,y_{p})}\Big).
\end{aligned}
\]
Again, the terms cancel out with each other so that \eqref{divcalculation1}
equals
\[
(1-m_{\star}(y_{s}))^{-1}\Big[\frac{N-1}{r(x_{1},\,y_{s})r(x_{2},\,y_{p})}-\frac{N-1}{r(x_{2},\,y_{s})r(x_{1},\,y_{p})}\Big],
\]
as wanted.
\end{proof}
Now, for all $1\le p\le2$ and $k\in\llbracket1,\,N-1\rrbracket$,
we define a correction flow $\phi_{p,\,k}$ as follows.
\begin{itemize}
\item Suppose that $\frac{N}{2}<k\le N-1$. Then, for $\ell\in\llbracket1,\,\frac{N}{2}\rrbracket$,
\begin{equation}
\phi_{p,\,k}(\xi_{k-\ell,\,\ell}^{x_{1}y_{p}x_{2}},\,\xi_{k-\ell+1,\,\ell-1}^{x_{1}y_{p}x_{2}}):=-\sum_{t=\ell}^{\lfloor N/2\rfloor}(\mathrm{div}\,\psi_{0})(\xi_{k-t,\,t}^{x_{1}y_{p}x_{2}}),\label{e_phipk1}
\end{equation}
and $\phi_{p,\,k}=0$ on all other edges.
\item Suppose that $1\le k\le\frac{N}{2}$. Then, for $\ell\in\llbracket1,\,k-1\rrbracket$,
\begin{equation}
\phi_{p,\,k}(\xi_{k-\ell,\,\ell}^{x_{1}y_{p}x_{2}},\,\xi_{k-\ell+1,\,\ell-1}^{x_{1}y_{p}x_{2}}):=-\sum_{t=\ell}^{k-1}(\mathrm{div}\,\psi_{0})(\xi_{k-t,\,t}^{x_{1}y_{p}x_{2}}),\label{e_phipk2}
\end{equation}
and $\phi_{p,\,k}=0$ on all other edges.
\end{itemize}
Finally, we define a flow

\[
\psi=\psi_{\textup{test}}:=\psi_{0}+\sum_{p=1}^{2}\sum_{k=1}^{N-1}\phi_{p,\,k}.
\]
Then, the flows $\phi_{p,\,k}$ for $1\le p\le2$ and $k\in\llbracket1,\,N-1\rrbracket$
cancel the divergence of $\psi_{0}$ at each $\xi_{k-\ell,\,\ell}^{x_{1}y_{p}x_{2}}\in\mathcal{A}_{N}^{x_{1}y_{p}x_{2}}$.
Thus, we obtain that $(\mathrm{div}\,\psi)(\eta)=0$ for all $\eta$
in
\[
\mathcal{A}_{N}^{x_{1}y_{p}x_{2}}\setminus(\mathcal{A}_{N}^{x_{1}y_{p}}\cup\mathcal{A}_{N}^{x_{2}y_{p}}\cup\mathcal{A}_{N}^{x_{1}x_{2}})\text{ for }1\le p\le2.
\]

\subsection{\label{ss-flownormspl}Flow norm of $\psi_{\textup{test}}$}

In this subsection, we calculate the flow norm of the test flow $\psi$. 
\begin{lem}
\label{lem5spl}Under the conditions of Theorem \ref{mainspl}, it
holds that 
\[
\|\psi\|_{N}^{2}\le(1+o(1))\frac{2N}{d_{N}^{2}\mathfrak{R}}.
\]
\end{lem}

\begin{proof}
By \eqref{testflowspl}, \eqref{testflow2spl}, and Definition \ref{flowdef},
we have
\begin{equation}
\begin{aligned}\|\psi\|_{N}^{2}= & \sum_{p=1}^{2}\sum_{\ell=0}^{\lfloor N/2\rfloor-1}\sum_{k=1}^{N-\ell-1}\\
 & \Big[\frac{(\psi(\xi_{k,\,\ell}^{x_{1}y_{p}x_{2}},\,\xi_{k-1,\,\ell+1}^{x_{1}y_{p}x_{2}}))^{2}}{\mu_{N}(\xi_{k,\,\ell}^{x_{1}y_{p}x_{2}})k(d_{N}+\ell)r(x_{1},\,y_{p})}+\frac{(\psi(\xi_{k,\,\ell}^{x_{1}y_{p}x_{2}},\,\xi_{k,\,\ell+1}^{x_{1}y_{p}x_{2}}))^{2}}{\mu_{N}(\xi_{k,\,\ell}^{x_{1}y_{p}x_{2}})(N-\ell-k)(d_{N}+\ell)r(x_{2},\,y_{p})}.
\end{aligned}
\label{lem5eq1spl}
\end{equation}
By \eqref{munprop} and \eqref{testflowspl}, the part of \eqref{lem5eq1spl}
including the first fraction inside bracket is asymptotically equivalent
to 
\begin{align*}
\frac{2}{d_{N}^{2}N}\sum_{p=1}^{2}\sum_{\ell=0}^{\lfloor N/2\rfloor-1}\frac{m_{\star}(y_{p})^{\ell}}{\mathfrak{R}^{2}}\sum_{k=1}^{N-\ell-1} & \frac{\frac{N-\ell-k}{r(x_{1},\,y_{p})}\big/(\frac{N-k-\ell-1}{r(x_{1},\,y_{p})}+\frac{k+\ell}{r(x_{2},\,y_{p})})^{2}}{\{\sum_{q=1}^{2}\frac{(1-m_{\star}(y_{q}))^{-1}}{\frac{N-k-\ell-1}{r(x_{1},\,y_{q})}+\frac{k+\ell}{r(x_{2},\,y_{q})}}\}^{2}}.
\end{align*}
Divide $N-\ell-k=(N-k-\ell-1)+1$. Then, as in obtaining \eqref{lem1eq4spl},
the last formula can be bounded from above by 
\begin{equation}
\frac{2}{d_{N}^{2}N}\sum_{p=1}^{2}\frac{N^{2}\mathfrak{R}^{-2}}{1-m_{\star}(y_{p})}\Big[\int_{0}^{1}\frac{\frac{1-t}{r(x_{1},\,y_{p})}\big/(\frac{1-t}{r(x_{1},\,y_{p})}+\frac{t}{r(x_{2},\,y_{p})})^{2}}{\{\sum_{q=1}^{2}\frac{(1-m_{\star}(y_{q}))^{-1}}{\frac{1-t}{r(x_{1},\,y_{q})}+\frac{t}{r(x_{2},\,y_{q})}}\}^{2}}dt+o(1)\Big].\label{lem5eq2spl}
\end{equation}
Similarly, the part of \eqref{lem5eq1spl} including the second fraction
inside bracket is asymptotically bounded from above by 
\begin{equation}
\frac{2}{d_{N}^{2}N}\sum_{p=1}^{2}\frac{N^{2}\mathfrak{R}^{-2}}{1-m_{\star}(y_{p})}\Big[\int_{0}^{1}\frac{\frac{t}{r(x_{2},\,y_{p})}\big/(\frac{1-t}{r(x_{1},\,y_{p})}+\frac{t}{r(x_{2},\,y_{p})})^{2}}{\{\sum_{q=1}^{2}\frac{(1-m_{\star}(y_{q}))^{-1}}{\frac{1-t}{r(x_{1},\,y_{q})}+\frac{t}{r(x_{2},\,y_{q})}}\}^{2}}dt+o(1)\Big].\label{lem5eq3spl}
\end{equation}
Hence, by \eqref{lem5eq1spl}, \eqref{lem5eq2spl}, and \eqref{lem5eq3spl},
we have the following asymptotic upper bound for $\|\psi\|_{N}^{2}$:
\[
\frac{2N}{d_{N}^{2}}\sum_{p=1}^{2}\mathfrak{R}^{-2}\int_{0}^{1}\frac{(1-m_{\star}(y_{p}))^{-1}(\frac{1-t}{r(x_{1},\,y_{p})}+\frac{t}{r(x_{2},\,y_{p})})^{-1}}{\{\sum_{q=1}^{2}\frac{(1-m_{\star}(y_{q}))^{-1}}{\frac{1-t}{r(x_{1},\,y_{q})}+\frac{t}{r(x_{2},\,y_{q})}}\}^{2}}dt=\frac{2N}{d_{N}^{2}}\times\frac{1}{\mathfrak{R}}.
\]
This concludes the proof.
\end{proof}

\subsection{\label{ss-remflowspl}Remaining terms}

Here, we address the remaining terms on the right-hand side of \eqref{genThomeq}
with respect to $\psi$. \emph{To this end, Lemma \ref{eqpotest}
is used to calculate the equilibrium potential near the metastable
valleys.}
\begin{lem}
\label{lem6spl}Under the conditions of Theorem \ref{mainspl}, it
holds that 
\begin{equation}
\sum_{\eta\in\mathcal{H}_{N}\setminus\mathcal{E}_{N}(S_{\star})}h_{\mathcal{E}_{N}^{x_{1}},\,\mathcal{E}_{N}^{x_{2}}}(\eta)(\mathrm{div}\,\psi)(\eta)=\frac{1+o(1)}{\mathfrak{R}}.\label{lem6eq1spl}
\end{equation}
\end{lem}

\begin{proof}
We will abbreviate $h_{\mathcal{E}_{N}^{x_{1}},\,\mathcal{E}_{N}^{x_{2}}}$
as $h$. It follows from the last observation in Section \ref{ss-testflowdefspl}
that we only need to sum up the configurations in $\mathcal{A}_{N}^{x_{1}y_{p}}\setminus\mathcal{E}_{N}^{x_{1}}$,
$\mathcal{A}_{N}^{x_{2}y_{p}}\setminus\mathcal{E}_{N}^{x_{2}}$, and
$\mathcal{A}_{N}^{x_{1}x_{2}}\setminus(\mathcal{E}_{N}^{x_{1}}\cup\mathcal{E}_{N}^{x_{2}})$.
First, we claim that
\begin{equation}
\sum_{p=1}^{2}\Big[\sum_{\eta\in\mathcal{A}_{N}^{x_{1}y_{p}}\setminus\mathcal{E}_{N}^{x_{1}}}+\sum_{\eta\in\mathcal{A}_{N}^{x_{2}y_{p}}\setminus\mathcal{E}_{N}^{x_{2}}}\Big]h(\eta)(\mathrm{div}\,\psi_{0})(\eta)=\frac{1+o(1)}{\mathfrak{R}}.\label{lem6eq2spl}
\end{equation}
The left-hand side of \eqref{lem6eq2spl} is
\begin{equation}
\sum_{p=1}^{2}\sum_{\ell=1}^{\lfloor N/2\rfloor}h(\xi_{\ell}^{x_{1}y_{p}})(\mathrm{div}\,\psi_{0})(\xi_{\ell}^{x_{1}y_{p}})+\sum_{p=1}^{2}\sum_{\ell=1}^{\lfloor N/2\rfloor}h(\xi_{\ell}^{x_{2}y_{p}})(\mathrm{div}\,\psi_{0})(\xi_{\ell}^{x_{2}y_{p}}).\label{lem6eq3spl}
\end{equation}
By Lemma \ref{eqpotest}, we have
\begin{equation}
\sup_{1\le\ell\le\lfloor N/2\rfloor}\big|h(\xi_{\ell}^{x_{1}y_{p}})-1\big|=o(1),\label{lem6eq4spl}
\end{equation}
and 
\begin{equation}
\sup_{1\le\ell\le\lfloor N/2\rfloor}h(\xi_{\ell}^{x_{2}y_{p}})=o(1).\label{lem6eq5spl}
\end{equation}
Hence, \eqref{lem6eq3spl} is equal to 
\begin{equation}
(1+o(1))\sum_{p=1}^{2}\sum_{\ell=1}^{\lfloor N/2\rfloor}(\mathrm{div}\,\psi_{0})(\xi_{\ell}^{x_{1}y_{p}})+o(1)\sum_{p=1}^{2}\sum_{\ell=1}^{\lfloor N/2\rfloor}(\mathrm{div}\,\psi_{0})(\xi_{\ell}^{x_{2}y_{p}}).\label{lem6eq6spl}
\end{equation}
By \eqref{testflowspl}, the first term of \eqref{lem6eq6spl} is
asymptotically equivalent to
\[
\frac{1}{\mathfrak{R}}\sum_{p=1}^{2}\sum_{\ell=1}^{\lfloor N/2\rfloor}\frac{m_{\star}(y_{p})^{\ell-1}\big/\frac{N-1}{r(x_{2},\,y_{p})}}{\sum_{q=1}^{2}\frac{(1-m_{\star}(y_{q}))^{-1}}{\frac{N-1}{r(x_{2},\,y_{q})}}}=\frac{1}{\mathfrak{R}}\sum_{p=1}^{2}\sum_{\ell=1}^{\lfloor N/2\rfloor}\frac{m_{\star}(y_{p})^{\ell-1}r(x_{2},\,y_{p})}{\sum_{q=1}^{2}\frac{r(x_{2},\,y_{q})}{1-m_{\star}(y_{q})}}.
\]
Summing for $1\le\ell\le\frac{N}{2}$, the last formula equals $1/\mathfrak{R}$.
Similarly, the second part of \eqref{lem6eq6spl} equals $o(1)/\mathfrak{R}=o(1)$.
This concludes the proof of \eqref{lem6eq2spl}.

Next, from the definition, note that $(\mathrm{div}\,\phi_{p,\,k})(\eta)$
vanishes unless 
\[
\eta\in\mathcal{A}_{N}^{x_{1}y_{p}x_{2}}\setminus(\mathcal{A}_{N}^{x_{1}y_{p}}\cup\mathcal{A}_{N}^{x_{2}y_{p}}).
\]
Moreover, it is verified right after the definition of $\psi$ that
$\mathrm{div}\,\psi$ vanishes in
\[
\widehat{\mathcal{A}}_{N}^{x_{1}y_{p}x_{2}}\text{ for }1\le p\le2.
\]
Finally, it is straightforward that $(\mathrm{div}\,\psi_{0})(\eta)=0$
for $\eta\in\mathcal{A}_{N}^{x_{1}x_{2}}$. Combining these observations,
it remains to prove that
\[
\sum_{k=1}^{N-1}\sum_{p=1}^{2}\sum_{\eta\in\mathcal{A}_{N}^{x_{1}x_{2}}\setminus(\mathcal{E}_{N}^{x_{1}}\cup\mathcal{E}_{N}^{x_{2}})}h(\eta)(\mathrm{div}\,\phi_{p,\,k})(\eta)=o(1).
\]
This can be rewritten as
\[
\sum_{k=1}^{N-1}\sum_{p=1}^{2}h(\xi_{k}^{x_{1}x_{2}})(\mathrm{div}\,\phi_{p,\,k})(\xi_{k}^{x_{1}x_{2}})=o(1).
\]
Since $0\le h\le1$, it suffices to prove that
\[
\sum_{k=1}^{N-1}\Big|\sum_{p=1}^{2}(\mathrm{div}\,\phi_{p,\,k})(\xi_{k}^{x_{1}x_{2}})\Big|=o(1).
\]

For $\frac{N}{2}<k\le N-1$, by \eqref{e_phipk1} it holds that
\[
\sum_{p=1}^{2}(\mathrm{div}\,\phi_{p,\,k})(\xi_{k}^{x_{1}x_{2}})=\sum_{p=1}^{2}\sum_{t=1}^{\lfloor N/2\rfloor}(\mathrm{div}\,\psi_{0})(\xi_{k-t,\,t}^{x_{1}y_{p}x_{2}}).
\]
By \eqref{divpsi0}, for fixed $1\le p\le2$, the summation in $1\le t\le\lfloor N/2\rfloor$
is calculated as
\[
\begin{aligned} & \sum_{t=1}^{\lfloor N/2\rfloor}\frac{m_{\star}(y_{p})^{t-1}}{1-m_{\star}(y_{s})}\times\Big[\frac{1}{r(x_{1},\,y_{s})r(x_{2},\,y_{p})}-\frac{1}{r(x_{2},\,y_{s})r(x_{1},\,y_{p})}\Big]\times\mathfrak{A}(N,\,k)\\
 & =\frac{\mathfrak{A}(N,\,k)}{1-m_{\star}(y_{s})}\times\Big[\frac{1}{r(x_{1},\,y_{s})r(x_{2},\,y_{p})}-\frac{1}{r(x_{2},\,y_{s})r(x_{1},\,y_{p})}\Big]\times\Big[\frac{1}{1-m_{\star}(y_{p})}+O(m_{\star}(y_{p})^{\frac{N}{2}})\Big]\\
 & =\frac{\mathfrak{A}(N,\,k)}{(1-m_{\star}(y_{p}))(1-m_{\star}(y_{s}))}\Big[\frac{1}{r(x_{1},\,y_{s})r(x_{2},\,y_{p})}-\frac{1}{r(x_{2},\,y_{s})r(x_{1},\,y_{p})}\Big]+O\Big(\frac{m_{\star}(y_{p})^{\frac{N}{2}}}{N}\Big),
\end{aligned}
\]
where $\{p,\,s\}=\{1,\,2\}$, where in the first equality we used
\[
\sum_{t=1}^{\lfloor N/2\rfloor}m_{\star}(y_{p})^{t-1}=\frac{1}{1-m_{\star}(y_{p})}+\sum_{t>\lfloor N/2\rfloor}m_{\star}(y_{p})^{t-1}=\frac{1}{1-m_{\star}(y_{p})}+O(m_{\star}(y_{p})^{\frac{N}{2}}),
\]
and where in the second equality we used \eqref{ANkbound}. Summing
up for $p\in\{1,\,2\}$, the two terms involving the square bracket
cancel out with each other. Hence, we conclude that
\[
\sum_{k=\lfloor N/2\rfloor+1}^{N-1}\Big|\sum_{p=1}^{2}(\mathrm{div}\,\phi_{p,\,k})(\xi_{k}^{x_{1}x_{2}})\Big|\le\frac{N}{2}\times\sum_{p=1}^{2}O\Big(\frac{m_{\star}(y_{p})^{\frac{N}{2}}}{N}\Big)=O(m_{\star}(y_{p})^{\frac{N}{2}})=o(1).
\]
Therefore, it remains to prove that
\[
\sum_{k=1}^{\lfloor N/2\rfloor}\Big|\sum_{p=1}^{2}(\mathrm{div}\,\phi_{p,\,k})(\xi_{k}^{x_{1}x_{2}})\Big|=o(1).
\]
By a similar calculation, we obtain that the left-hand side is bounded
by
\[
\sum_{k=1}^{\lfloor N/2\rfloor}\sum_{p=1}^{2}O\Big(\frac{m_{\star}(y_{p})^{k-1}}{N}\Big)=O\Big(\frac{1}{N}\Big)=o(1).
\]
Thus, we conclude the proof.
\end{proof}

\subsection{Proof of Proposition \ref{LBspl}}

We are now ready to prove Proposition \ref{LBspl}. 
\begin{proof}[Proof of Proposition \ref{LBspl}]
 By Lemmas \ref{lem5spl} and \ref{lem6spl}, we have 
\[
\frac{1}{\|\psi_{\textup{test}}\|_{N}^{2}}\Big[\sum_{\eta\in\mathcal{H}_{N}\setminus\mathcal{E}_{N}^{\star}}h_{\mathcal{E}_{N}^{x_{1}},\,\mathcal{E}_{N}^{x_{2}}}(\eta)(\mathrm{div}\,\psi_{\textup{test}})(\eta)\Big]^{2}\ge(1+o(1))\frac{d_{N}^{2}}{2N\mathfrak{R}}.
\]
Therefore, we deduce from Theorem \ref{genThom} that
\[
\mathrm{Cap}_{N}(\mathcal{E}_{N}^{x_{1}},\,\mathcal{E}_{N}^{x_{2}})\ge(1+o(1))\frac{d_{N}^{2}}{2N\mathfrak{R}}.
\]
This concludes the proof of Proposition \ref{LBspl}.
\end{proof}

\section{\label{s-UB}Upper Bound for Capacities: General Case}

In this section, we omit Condition \ref{splcond} and extend the results
of Section \ref{s-UBspl} to the most general setting of Theorem \ref{main},
described in Subsection \ref{ss-sts}. Because proofs of the assertions
here are fundamentally similar to those in Section \ref{s-UBspl},
they will be written in a brief manner. 
\begin{prop}[Upper bound for capacities: General case]
\label{UB} Assume the conditions of Theorem \ref{main}. Then, for
any non-trivial partition $\{A,\,B\}$ of $\llbracket1,\,\kappa_{\star}\rrbracket$,
the following inequality holds. 
\[
\limsup_{N\to\infty}\frac{N}{d_{N}^{2}}\mathrm{Cap}_{N}(\mathcal{E}_{N}^{(2)}(A),\,\mathcal{E}_{N}^{(2)}(B))\le\frac{1}{|S_{\star}|}\sum_{i\in A}\sum_{j\in B}\frac{1}{\mathfrak{R}_{i,\,j}}.
\]
\end{prop}

\begin{rem}
In Proposition \ref{UB}, it is crucial to have $A\cup B=\llbracket1,\,\kappa_{\star}\rrbracket$;
if $A\cup B\subsetneq\llbracket1,\,\kappa_{\star}\rrbracket$, then
the equilibrium potential is significantly more complicated. Moreover,
we remark that if $\mathfrak{R}_{i,\,j}=\infty$ for all $i\in A$
and $j\in B$, then Proposition \ref{UB} asserts that
\[
\mathrm{Cap}_{N}(\mathcal{E}_{N}^{(2)}(A),\,\mathcal{E}_{N}^{(2)}(B))=o\Big(\frac{d_{N}^{2}}{N}\Big).
\]
This equality indicates that we do not observe metastable transitions
in the time scale $N/d_{N}^{2}$. Therefore, in this case, we expect
that yet another time scale is required to observe the metastable
transitions. This is conjectured to be $N^{2}/d_{N}^{3}$ in \cite{BDG}. 
\end{rem}

\subsection{Preliminary notions}

Once more, we define $m_{\star}=\max_{p=1}^{\kappa_{0}}m_{\star}(y_{p})<1$.
Because there are too many subscripts in the general case, we introduce
a convenient notation that helps us calculate the objects.

We define the following discretized version of the constant $\mathfrak{R}_{i,\,j}$
for $i,\,j\in\llbracket1,\,\kappa_{\star}\rrbracket$ given in \eqref{Tijdef}:

\[
\mathfrak{R}_{i,\,j}^{N}=\sum_{t=1}^{N}\frac{1}{\sum_{n=1}^{\mathfrak{n}(i)}\sum_{m=1}^{\mathfrak{n}(j)}\sum_{p=1}^{\kappa_{0}}\frac{(1-m_{\star}(y_{p}))^{-1}}{\frac{N-t}{r(x_{i,\,n},\,y_{p})}+\frac{t-1}{r(x_{j,\,m},\,y_{p})}}}.
\]
As in the special case in Subsection \ref{ss-prenotspl}, we write
$\mathfrak{R}_{i,\,j}^{N}=\infty$ if $r(x_{i,\,n},\,y_{p})r(x_{j,\,m},\,y_{p})=0$
for all $1\le n\le\mathfrak{n}(i)$, $1\le m\le\mathfrak{n}(j)$,
and $1\le p\le\kappa_{0}$. Clearly, we have $N^{-2}\mathfrak{R}_{i,\,j}^{N}\to\mathfrak{R}_{i,\,j}$
as $N$ tends to infinity. Moreover, define 
\[
I=\big\{(i,\,j)\in A\times B:\mathfrak{R}_{i,\,j}<\infty\big\},
\]
and for $(i,\,j)\in A\times B$,
\[
P_{i,\,n,\,j,\,m}=\{p:r(x_{i,\,n},\,y_{p})r(x_{j,\,m},\,y_{p})>0\}
\]
and 
\[
Q_{i,\,n,\,j,\,m}=\{p:r(x_{i,\,n},\,y_{p})+r(x_{j,\,m},\,y_{p})>0\}.
\]
For example, $(i,\,j)\in I$ if and only if $r(x_{i,\,n},\,y_{p})r(x_{j,\,m},\,y_{p})>0$
for some $n$, $m$, and $p$, which is also equivalent to 
\[
\bigcup_{n=1}^{\mathfrak{n}(i)}\bigcup_{m=1}^{\mathfrak{n}(j)}P_{i,\,n,\,j,\,m}\ne\emptyset.
\]
Moreover, we have $P_{i,\,n,\,j,\,m}\subseteq Q_{i,\,n,\,j,\,m}$.
Finally, we define
\begin{equation}
\mathcal{U}_{N}=\bigcup_{i\in A}\bigcup_{j\in B}\bigcup_{n=1}^{\mathfrak{n}(i)}\bigcup_{m=1}^{\mathfrak{n}(j)}\bigcup_{p=1}^{\kappa_{0}}\mathcal{A}_{N}^{x_{i,\,n}y_{p}x_{j,\,m}}\text{ and }\mathcal{V}_{N}=\mathcal{H}_{N}\setminus\mathcal{U}_{N}.\label{unvndef}
\end{equation}

\subsection{\label{ss-testfcndef}Construction of test function $f_{\textup{test}}^{A}$}

In this subsection, we define a test function $f=f_{\textup{test}}^{A}$
on $\mathcal{H}_{N}$, which approximates the equilibrium potential
$h_{\mathcal{E}_{N}^{(2)}(A),\,\mathcal{E}_{N}^{(2)}(B)}$. \emph{This
procedure is a natural extension of the definition in Subsection \ref{ss-testfcndefspl}.}
\begin{itemize}
\item First, we define $f$ on $\mathcal{E}_{N}(S)$:
\begin{equation}
f(\xi^{x_{i,\,n}})=1,\ i\in A\text{ and }f(\xi^{z})=0,\ z\in S\setminus\{x_{i,\,n}:i\in A,\ 1\le n\le\mathfrak{n}(i)\},\label{testfcn1}
\end{equation}
such that we have $f|_{\mathcal{E}_{N}^{(2)}(A)}=1$ and $f|_{\mathcal{E}_{N}^{(2)}(B)}=0$.
\item Second, we define $f$ on $\widehat{\mathcal{A}}_{N}^{x_{i,\,n}y_{p}}$
for $i\in\llbracket1,\,\kappa_{\star}\rrbracket$, $n\in\llbracket1,\,\mathfrak{n}(i)\rrbracket$,
and $p\in\llbracket1,\,\kappa_{0}\rrbracket$ by 
\begin{equation}
f(\xi_{k}^{x_{i,\,n}y_{p}})=f(\xi^{x_{i,\,n}}),\ 1\le k\le N-1.\label{testfcn2}
\end{equation}
\item Next, we define $f$ on the remainder of $\mathcal{U}_{N}$, i.e.,
on $\mathcal{A}_{N}^{x_{i,\,n}y_{p}x_{j,\,m}}\setminus(\mathcal{A}_{N}^{x_{i,\,n}y_{p}}\cup\mathcal{A}_{N}^{x_{j,\,m}y_{p}})$
for $(i,\,j)\in A\times B$, $n,\,m\ge1$, and $p\in\llbracket1,\,\kappa_{0}\rrbracket$.
This part is the main technical obstacle in the definition of $f_{\textup{test}}^{A}$.
There are four types, \textbf{(U1)} through \textbf{(U4)}.
\begin{itemize}
\item[\textbf{(U1)}]  If $(i,\,j)\in I$ and $n,\,m\ge1$ with $p\in Q_{i,\,n,\,j,\,m}$,
then for $\ell\in\llbracket1,\,N-2\rrbracket$ and $k\in\llbracket1,\,N-\ell-1\rrbracket$,
\begin{equation}
f(\xi_{k,\,\ell}^{x_{i,\,n}y_{p}x_{j,\,m}})=\frac{K_{i,\,n,\,p,\,j,\,m}^{k,\,\ell}}{\mathfrak{R}_{i,\,j}^{N}},\label{testfcn3-1}
\end{equation}
where
\[
\begin{aligned}K_{i,\,n,\,p,\,j,\,m}^{k,\,\ell}= & \sum_{t=1}^{k}\frac{\frac{N-t}{r(x_{i,\,n},\,y_{p})}\big/(\frac{N-t}{r(x_{i,\,n},\,y_{p})}+\frac{t-1}{r(x_{j,\,m},\,y_{p})})}{\sum_{\widetilde{n}=1}^{\mathfrak{n}(i)}\sum_{\widetilde{m}=1}^{\mathfrak{n}(j)}\sum_{q=1}^{\kappa_{0}}\frac{(1-m_{\star}(y_{q}))^{-1}}{\frac{N-t}{r(x_{i,\,\widetilde{n}},\,y_{q})}+\frac{t-1}{r(x_{j,\,\widetilde{m}},\,y_{q})}}}\\
 & +\sum_{t=1}^{k+\ell}\frac{\frac{t-1}{r(x_{j,\,m},\,y_{p})}\big/(\frac{N-t}{r(x_{i,\,n},\,y_{p})}+\frac{t-1}{r(x_{j,\,m},\,y_{p})})}{\sum_{\widetilde{n}=1}^{\mathfrak{n}(i)}\sum_{\widetilde{m}=1}^{\mathfrak{n}(j)}\sum_{q=1}^{\kappa_{0}}\frac{(1-m_{\star}(y_{q}))^{-1}}{\frac{N-t}{r(x_{i,\,\widetilde{n}},\,y_{q})}+\frac{t-1}{r(x_{j,\,\widetilde{m}},\,y_{q})}}}.
\end{aligned}
\]
By substituting $\ell=0$, one can verify that \eqref{testfcn3-1}
is well defined on $\widehat{\mathcal{A}}_{N}^{x_{i,\,n}x_{j,\,m}}$.
The fractions inside summations are well defined, as $(i,\,j)\in I$
implies that the common denominator is strictly positive. The numerators
of the fractions must be understood naturally if $r(x_{i,\,n},\,y_{p})r(x_{j,\,m},\,y_{p})=0$.
Indeed, if e.g., $r(x_{i,\,n},\,y_{p})>0$ and $r(x_{j,\,m},\,y_{p})=0$,
then the first one (``$0/\infty$'') is $0$, and the second one
(``$\infty/\infty$'') is $1$. 
\item[\textbf{(U2)}]  If $(i,\,j)\in I$ and $n,\,m\ge1$ with $p\notin Q_{i,\,n,\,j,\,m}$,
then for $\ell\in\llbracket1,\,N-2\rrbracket$ and $k\in\llbracket1,\,N-\ell-1\rrbracket$,
\begin{equation}
f(\xi_{k,\,\ell}^{x_{i,\,n}y_{p}x_{j,\,m}})=\frac{1}{\mathfrak{R}_{i,\,j}^{N}}\sum_{t=1}^{k}\frac{1}{\sum_{\widetilde{n}=1}^{\mathfrak{n}(i)}\sum_{\widetilde{m}=1}^{\mathfrak{n}(j)}\sum_{q=1}^{\kappa_{0}}\frac{(1-m_{\star}(y_{q}))^{-1}}{\frac{N-t}{r(x_{i,\,\widetilde{n}},\,y_{q})}+\frac{t-1}{r(x_{j,\,\widetilde{m}},\,y_{q})}}}.\label{testfcn3-2}
\end{equation}
Note that \eqref{testfcn3-2} is consistent with \eqref{testfcn3-1}
on $\widehat{\mathcal{A}}_{N}^{x_{i,\,n}x_{j,\,m}}$.
\item[\textbf{(U3)}]  If $(i,\,j)\notin I$ and $n,\,m\ge1$ with $p\in Q_{i,\,n,\,j,\,m}\setminus P_{i,\,n,\,j,\,m}$,
such that $r(x_{i,\,n},\,y_{p})>0$ and $r(x_{j,\,m},\,y_{p})=0$,
then for $\ell\in\llbracket0,\,N-2\rrbracket$ and $k\in\llbracket1,\,N-\ell-1\rrbracket$,
\begin{equation}
f(\xi_{k,\,\ell}^{x_{i,\,n}y_{p}x_{j,\,m}})=\frac{k+\ell}{N}.\label{testfcn3-3}
\end{equation}
As done previously, one can substitute $\ell=0$ to verify that \eqref{testfcn3-3}
is well defined on $\widehat{\mathcal{A}}_{N}^{x_{i,\,n}x_{j,\,m}}$.
Similarly, if $p\in Q_{i,\,n,\,j,\,m}\setminus P_{i,\,n,\,j,\,m}$
with $r(x_{i,\,n},\,y_{p})=0$ and $r(x_{j,\,m},\,y_{p})>0$, then
define
\begin{equation}
f(\xi_{k,\,\ell}^{x_{i,\,n}y_{p}x_{j,\,m}})=\frac{k}{N}.\label{testfcn3-4}
\end{equation}
\item[\textbf{(U4)}]  If $(i,\,j)\notin I$ and $n,\,m\ge1$ with $p\notin Q_{i,\,n,\,j,\,m}$,
then for $\ell\in\llbracket1,\,N-2\rrbracket$ and $k\in\llbracket1,\,N-\ell-1\rrbracket$,
\begin{equation}
f(\xi_{k,\,\ell}^{x_{i,\,n}y_{p}x_{j,\,m}})=\frac{k}{N}.\label{testfcn3-5}
\end{equation}
\eqref{testfcn3-5} is well defined on $\widehat{\mathcal{A}}_{N}^{x_{i,\,n}x_{j,\,m}}$
and consistent with \eqref{testfcn3-3}; substitute $\ell=0$.
\end{itemize}
\item Finally, we define $f$ on $\mathcal{V}_{N}$. Assume $\eta\in\mathcal{V}_{N}$.
There are three types, \textbf{(V1)}, \textbf{(V2)}, and \textbf{(V3)},
denoted by $\mathcal{V}_{N}^{1}$, $\mathcal{V}_{N}^{2}$, and $\mathcal{V}_{N}^{3}$,
respectively, such that
\begin{equation}
\mathcal{V}_{N}=\mathcal{V}_{N}^{1}\cup\mathcal{V}_{N}^{2}\cup\mathcal{V}_{N}^{3}.\label{vndiv}
\end{equation}
Define $\eta_{A}:=\sum_{i\in A}\sum_{n=1}^{\mathfrak{n}(i)}\eta_{x_{i,\,n}}$
and $\eta_{B}:=\sum_{j\in B}\sum_{m=1}^{\mathfrak{n}(j)}\eta_{x_{j,\,m}}$.
\begin{itemize}
\item[\textbf{(V1)}]  If $\eta_{A}=0$, then define 
\begin{equation}
f(\eta)=0.\label{testfcn4}
\end{equation}
\item[\textbf{(V2)}]  If $\eta_{A}\ge1$ and $\eta_{B}=0$, then define 
\begin{equation}
f(\eta)=1.\label{testfcn4-2}
\end{equation}
\item[\textbf{(V3)}]  If $\eta_{A},\,\eta_{B}\ge1$, we define
\begin{equation}
f(\eta)=\sum_{i\in A}\sum_{j\in B}\sum_{n=1}^{\mathfrak{n}(i)}\sum_{m=1}^{\mathfrak{n}(j)}\frac{\eta_{x_{i,\,n}}\eta_{x_{j,\,m}}}{\eta_{A}\eta_{B}}f(\xi_{\eta_{A}}^{x_{i,\,n}x_{j,\,m}}).\label{testfcn4-3}
\end{equation}
\end{itemize}
\item By construction, we have $0\le f(\eta)\le1$ for all $\eta\in\mathcal{H}_{N}$. 
\end{itemize}
We divide the Dirichlet form into four parts: 
\[
D_{N}(f)=\Sigma_{1}(f)+\Sigma_{2}(f)+\Sigma_{3}(f)+\Sigma_{4}(f).
\]

\begin{itemize}
\item The first part $\Sigma_{1}(f)$ consists of movements inside $\mathcal{A}_{N}^{x_{i,\,n}y_{p}x_{j,\,m}}$
for all $i,\,j\in\llbracket1,\,\kappa_{\star}\rrbracket$, $n\in\llbracket1,\,\mathfrak{n}(i)\rrbracket$,
$m\in\llbracket1,\,\mathfrak{n}(j)\rrbracket$, and $p\in\llbracket1,\,\kappa_{0}\rrbracket$. 
\item The second part $\Sigma_{2}(f)$ consists of movements between the
set differences of two distinct $\mathcal{A}_{N}^{x_{i,\,n}y_{p}x_{j,\,m}}$-type
sets.
\item The third part $\Sigma_{3}(f)$ consists of movements between $\mathcal{U}_{N}$
and $\mathcal{V}_{N}$. 
\item The last part $\Sigma_{4}(f)$ consists of movements inside $\mathcal{V}_{N}$.
\end{itemize}

\subsection{Main contribution of Dirichlet form}

In this subsection, we calculate $\Sigma_{1}(f)$, which is the main
ingredient of $D_{N}(f)$. 
\begin{lem}
\label{lem1}Under the conditions of Theorem \ref{main}, it holds
that 
\[
\Sigma_{1}(f)\le\frac{d_{N}^{2}}{|S_{\star}|N}\Big[\sum_{i\in A}\sum_{j\in B}\frac{1}{\mathfrak{R}_{i,\,j}}+o(1)\Big].
\]
\end{lem}

\begin{proof}
We write down all movements inside $\mathcal{A}_{N}^{x_{i,\,n}y_{p}x_{j,\,m}}$
and sum it up for all $i,\,j,\,n,\,m,\,p$. Namely,
\[
\begin{aligned} & \Sigma_{1}(f)\le\sum_{i\in A}\sum_{j\in B}\sum_{n,\,m,\,p}\sum_{\eta\in\mathcal{A}_{N}^{x_{i,\,n}y_{p}x_{j,\,m}}}\mu_{N}(\eta)\times\\
 & \big[\mathbf{q}_{N}(\eta,\,\sigma^{x_{i,\,n},\,y_{p}}\eta)\{f(\sigma^{x_{i,\,n},\,y_{p}}\eta)-f(\eta)\}^{2}+\mathbf{q}_{N}(\eta,\,\sigma^{x_{j,\,m},\,y_{p}}\eta)\{f(\sigma^{x_{j,\,m},\,y_{p}}\eta)-f(\eta)\}^{2}\big]
\end{aligned}
\]
is the desired equation. The only overlapping terms on the right-hand
side above are movements along $\mathcal{A}_{N}^{x_{i,\,n}y_{p}}$
for $i\in\llbracket1,\,\kappa_{\star}\rrbracket$, $n\in\llbracket1,\,\mathfrak{n}(i)\rrbracket$,
and $p\in\llbracket1,\,\kappa_{0}\rrbracket$. In fact, by \eqref{testfcn2},
these terms have an exponentially small effect on the entire summation.
Thus, the inequality used above is actually sharp. By \eqref{muneq1}
and \eqref{munprop}, the right-hand side is asymptotically equal
to
\[
\begin{aligned}\frac{d_{N}N}{|S_{\star}|} & \sum_{i\in A}\sum_{j\in B}\sum_{n,\,m}\sum_{p\in Q_{i,\,n,\,j,\,m}}\sum_{\ell=0}^{N-1}m_{\star}(y_{p})^{\ell}\times\\
 & \Big[\sum_{k=1}^{N-\ell}w_{N}(N-\ell-k)r(x_{i,\,n},\,y_{p})\{f(\xi_{k-1,\,\ell+1}^{x_{i,\,n}y_{p}x_{j,\,m}})-f(\xi_{k,\,\ell}^{x_{i,\,n}y_{p}x_{j,\,m}})\}^{2}\\
 & \ +\sum_{k=0}^{N-\ell-1}w_{N}(k)r(x_{j,\,m},\,y_{p})\{f(\xi_{k,\,\ell+1}^{x_{i,\,n}y_{p}x_{j,\,m}})-f(\xi_{k,\,\ell}^{x_{i,\,n}y_{p}x_{j,\,m}})\}^{2}\Big].
\end{aligned}
\]
We only need to consider $p\in Q_{i,\,n,\,j,\,m}$, as otherwise $r(x_{i,\,n},\,y_{p})=r(x_{j,\,m},\,y_{p})=0$.
Next, if $p\in Q_{i,\,n,\,j,\,m}\setminus P_{i,\,n,\,j,\,m}$, then
the terms inside the bracket vanish due to \eqref{testfcn3-1}, \eqref{testfcn3-2},
and \eqref{testfcn3-3}. Gathering the preceding observations, the
last formula is asymptotically equal to
\[
\begin{aligned}\frac{d_{N}N}{|S_{\star}|} & \sum_{(i,\,j)\in I\cap(A\times B)}\sum_{n,\,m}\sum_{p\in P_{i,\,n,\,j,\,m}}\sum_{\ell=0}^{N-1}m_{\star}(y_{p})^{\ell}\times\\
 & \Big[\sum_{k=1}^{N-\ell}w_{N}(N-\ell-k)r(x_{i,\,n},\,y_{p})\{f(\xi_{k-1,\,\ell+1}^{x_{i,\,n}y_{p}x_{j,\,m}})-f(\xi_{k,\,\ell}^{x_{i,\,n}y_{p}x_{j,\,m}})\}^{2}\\
 & \;+\sum_{k=0}^{N-\ell-1}w_{N}(k)r(x_{j,\,m},\,y_{p})\{f(\xi_{k,\,\ell+1}^{x_{i,\,n}y_{p}x_{j,\,m}})-f(\xi_{k,\,\ell}^{x_{i,\,n}y_{p}x_{j,\,m}})\}^{2}\Big].
\end{aligned}
\]
The rest of the proof is almost identical to that of Lemma \ref{lem1spl};
we obtain 
\[
\Sigma_{1}(f)\le\frac{d_{N}^{2}}{|S_{\star}|N}\Big[\sum_{(i,\,j)\in I\cap(A\times B)}\frac{1}{\mathfrak{R}_{i,\,j}}+o(1)\Big].
\]
Because $\mathfrak{R}_{i,\,j}=\infty$ if $(i,\,j)\notin I$, the
last formula is exactly what we expect.
\end{proof}

\subsection{\label{ss-remDiri}Remainder of Dirichlet form}

Here, we deal with the remaining terms in the Dirichlet form, $\Sigma_{2}(f)$,
$\Sigma_{3}(f)$, and $\Sigma_{4}(f)$. Lemma \ref{lem2} deals with
$\Sigma_{2}(f)$. 
\begin{lem}
\label{lem2}Under the conditions of Theorem \ref{main}, it holds
that 
\[
\Sigma_{2}(f)=O\Big(\frac{d_{N}^{3}\log N}{N^{2}}\Big)=o\Big(\frac{d_{N}^{2}}{N}\Big).
\]
\end{lem}

\begin{proof}
Recalling that $\Sigma_{2}(f)$ consists of dynamics between the set
differences of two distinct $\mathcal{A}_{N}^{x_{i,\,n}y_{p}x_{j,\,m}}$-type
sets, there are two such types of movements.\medskip{}

\noindent \textbf{(Case 1)} The first case is represented when a sole
particle moves between $x_{i,\,n}$ and $x_{i,\,\widetilde{n}}$.
More specifically, this is written as
\begin{equation}
\begin{aligned}\sum_{i\in A}\sum_{j\in B}\sum_{n,\,m,\,p} & \sum_{\ell=0}^{N-1}\mu_{N}(\xi_{1,\,\ell}^{x_{i,\,n}y_{p}x_{j,\,m}})\times\\
 & \Big[\sum_{\widetilde{n}}d_{N}r(x_{i,\,n},\,x_{i,\,\widetilde{n}})\{f(\xi_{1,\,\ell}^{x_{i,\,\widetilde{n}}y_{p}x_{j,\,m}})-f(\xi_{1,\,\ell}^{x_{i,\,n}y_{p}x_{j,\,m}})\}^{2}\\
 & \ +\sum_{\widetilde{m}}d_{N}r(x_{j,\,m},\,x_{j,\,\widetilde{m}})\{f(\xi_{1,\,\ell}^{x_{i,\,n}y_{p}x_{j,\,\widetilde{m}}})-f(\xi_{1,\,\ell}^{x_{i,\,n}y_{p}x_{j,\,m}})\}^{2}\Big].
\end{aligned}
\label{lem2eq1}
\end{equation}
If $\ell=0$, then this vanishes by \eqref{testfcn3-1} and \eqref{testfcn3-3}.
If $\ell=N-1$, then this vanishes by \eqref{testfcn1} and \eqref{testfcn2}.
If $\ell\in\llbracket1,\,N-2\rrbracket$, then 
\[
f(\xi_{1,\,\ell}^{x_{i,\,\widetilde{n}}y_{p}x_{j,\,m}})-f(\xi_{1,\,\ell}^{x_{i,\,n}y_{p}x_{j,\,m}})=O\Big(\frac{\ell}{N}\Big),
\]
by \eqref{testfcn3-1}, \eqref{testfcn3-2}, \eqref{testfcn3-3},
\eqref{testfcn3-4}, and \eqref{testfcn3-5}. Therefore, \eqref{lem2eq1}
is bounded from above by 
\begin{equation}
C\sum_{\ell=1}^{N-2}\frac{Nd_{N}^{3}m_{\star}^{\ell}}{\ell(N-\ell-1)}\frac{\ell^{2}}{N^{2}}=O\Big(\frac{d_{N}^{3}}{N^{2}}\Big).\label{lem2eq2}
\end{equation}
\medskip{}

\noindent \textbf{(Case 2)} The second case is represented when a
sole particle moves between $y_{p}$ and $y_{q}$. This case is identical
to Lemma \ref{lem2spl}, which is bounded by 
\begin{equation}
C\sum_{k=1}^{N-2}\frac{Nd_{N}^{3}m_{\star}}{k(N-k-1)}\frac{1}{N^{2}}=O\Big(\frac{d_{N}^{3}\log N}{N^{2}}\Big).\label{lem2eq3}
\end{equation}
Gathering the cases, we have by \eqref{lem2eq2} and \eqref{lem2eq3}
that $\Sigma_{2}(f)=O(d_{N}^{3}N^{-2}\log N)$. This concludes the
proof. 
\end{proof}
Next, we consider $\Sigma_{3}(f)$. 
\begin{lem}
\label{lem3}Under the conditions of Theorem \ref{main}, it holds
that
\[
\Sigma_{3}(f)=O\Big(d_{N}^{2}m_{\star}^{N}+\frac{d_{N}^{3}}{N}\log N\Big)=o\Big(\frac{d_{N}^{2}}{N}\Big).
\]
\end{lem}

\begin{proof}
We formulate
\[
\Sigma_{3}(f)=\sum_{\eta\in\mathcal{U}_{N}}\sum_{\zeta\in\mathcal{V}_{N}}\mu_{N}(\eta)\mathbf{q}_{N}(\eta,\,\zeta)\{f(\zeta)-f(\eta)\}^{2}.
\]
We divide this into several cases depending on which subset $\eta$
belongs to.\medskip{}

\noindent \textbf{(Case 1)} $\eta\in\mathcal{A}_{N}^{x_{i,\,n}x_{j,\,m}}$
for some $(i,\,j)\in A\times B$ and $n,\,m\ge1$: In this case, the
movement must occur between sites in $S_{i}^{(2)}$ or between sites
in $S_{j}^{(2)}$. Otherwise, $\zeta\notin\mathcal{V}_{N}$. Hence,
$f$ remains unchanged by \eqref{testfcn3-1}, \eqref{testfcn3-3},
and \eqref{testfcn4-3}.\medskip{}

\noindent \textbf{(Case 2)} $\eta\in\mathcal{A}_{N}^{x_{i,\,n}y_{p}}\setminus\mathcal{E}_{N}^{x_{i,\,n}}$
for some $i\in\llbracket1,\,\kappa_{\star}\rrbracket$ and $n,\,p\ge1$:
We divide again by types of the particle movement.\medskip{}
\begin{itemize}
\item \textbf{(Case 2-1)} Movement from $y_{p}$ into $S_{\star}\setminus\{x_{i,\,n}\}$:
We have $\eta_{y_{p}}\le N-1$, since otherwise $\zeta\notin\mathcal{V}_{N}$.
Hence, $f$ remains unchanged by \eqref{testfcn2}, \eqref{testfcn4},
and \eqref{testfcn4-2}.\medskip{}
\item \textbf{(Case 2-2)} Movement from $y_{p}$ into $S\setminus(S_{\star}\cup\{y_{p}\})$:
This is identical to \textbf{(Case 2-1)} of Lemma \ref{lem3spl},
so that the summation is $0$.\medskip{}
\item \textbf{(Case 2-3)} Movement from $x_{i,\,n}$ into $S_{\star}\setminus\{x_{i,\,n}\}$:
$f$ remains unchanged by \eqref{testfcn2}, \eqref{testfcn4}, and
\eqref{testfcn4-2}.\medskip{}
\item \textbf{(Case 2-4)} Movement from $x_{i,\,n}$ into $S\setminus(S_{\star}\cup\{y_{p}\})$:
This is same with \textbf{(Case 2-2)} of Lemma \ref{lem3spl}. We
obtain the upper bound $O(d_{N}^{2}m_{\star}^{N})$. In conclusion,
\textbf{(Case 2)} yields $O(d_{N}^{2}m_{\star}^{N})$.\medskip{}
\end{itemize}
\textbf{(Case 3)} $\eta\in\widehat{\mathcal{A}}_{N}^{x_{i,\,n}y_{p}x_{j,\,m}}$
for some $(i,\,j)\in A\times B$ and $n,\,m,\,p\ge1$: This case is
almost identical to \textbf{(Case 3)} of Lemma \ref{lem3spl}. The
only different thing is that the particle movement might occur as
$x_{i,\,n}\leftrightarrow x_{i,\,\widetilde{n}}$ or $x_{j,\,m}\leftrightarrow x_{j,\,\widetilde{m}}$.
In this other case, we may bound the summation as
\[
C\sum_{k=1}^{N-2}\sum_{\ell=1}^{N-k-1}\frac{Nd_{N}^{3}m_{\star}^{\ell}}{\ell(N-k-\ell)}\cdot\frac{(\ell+1)^{2}}{N^{2}}=O\Big(\frac{d_{N}^{3}}{N}\log N\Big).
\]
Thus, we can bound \textbf{(Case 3)} by $O(\frac{d_{N}^{3}}{N}\log N)$.
Summarizing all cases, we conclude that $\Sigma_{3}(f)=O(d_{N}^{2}m_{\star}^{N}+\frac{d_{N}^{3}}{N}\log N)$.
\end{proof}
Our final aim of this subsection is $\Sigma_{4}(f)$.
\begin{lem}
\label{lem4}Under the conditions of Theorem \ref{main}, it holds
that
\[
\Sigma_{4}(f)=O(d_{N}^{2}m_{\star}^{N})+O\Big(\frac{d_{N}^{3}}{N}\log N\Big)=o\Big(\frac{d_{N}^{2}}{N}\Big).
\]
\end{lem}

\begin{proof}
By definition, we have
\[
\Sigma_{4}(f)=\frac{1}{2}\sum_{\eta,\,\zeta\in\mathcal{V}_{N}}\mu_{N}(\eta)\mathbf{q}_{N}(\eta,\,\zeta)\{f(\zeta)-f(\eta)\}^{2}.
\]
Recalling \eqref{vndiv}, we divide the summation in $\eta,\zeta\in\mathcal{V}_{N}$
by where $\eta$ and $\zeta$ belong.\medskip{}

\noindent \textbf{(Case 1)} $\mathcal{V}_{N}^{1}\leftrightarrow\mathcal{V}_{N}^{1}$
or $\mathcal{V}_{N}^{2}\leftrightarrow\mathcal{V}_{N}^{2}$: $f$
remains unchanged by \eqref{testfcn4} and \eqref{testfcn4-2}.\medskip{}

\noindent \textbf{(Case 2)} $\mathcal{V}_{N}^{1}\leftrightarrow\mathcal{V}_{N}^{2}$:
Similarly to \textbf{(Case 2)} of Lemma \ref{lem4spl}, the summation
is exponentially small scaling as $O(d_{N}^{2}m_{\star}^{N})$.\medskip{}

\noindent \textbf{(Case 3)} $\mathcal{V}_{N}^{1}\leftrightarrow\mathcal{V}_{N}^{3}$:
This is impossible.\medskip{}

\noindent \textbf{(Case 4)} $\mathcal{V}_{N}^{2}\leftrightarrow\mathcal{V}_{N}^{3}$:
We can bound this case by $O(\frac{d_{N}^{3}}{N^{2}})$ in a similar
way as in \textbf{(Case 4)} of Lemma \ref{lem4spl}.\medskip{}

\noindent \textbf{(Case 5)} $\mathcal{V}_{N}^{3}\leftrightarrow\mathcal{V}_{N}^{3}$:
In the same reasoning with (Case 5) of Lemma \ref{lem4spl}, with
some additional care, we may bound this case by
\[
C\frac{d_{N}^{2}}{N}\cdot\sum_{\ell=1}^{\infty}(d_{N}\log N)^{\ell}=C\frac{d_{N}^{3}}{N}\log N\cdot\frac{1}{1-d_{N}\log N}=O\Big(\frac{d_{N}^{3}}{N}\log N\Big).
\]
\end{proof}

\subsection{Proof of Proposition \ref{UB}}

Now, we are in position to prove Proposition \ref{UB}. 
\begin{proof}[Proof of Proposition \ref{UB}]
 By Lemmas \ref{lem1}, \ref{lem2}, \ref{lem3}, and \ref{lem4},
\[
D_{N}(f_{\textup{test}}^{A})\le\frac{d_{N}^{2}}{|S_{\star}|N}\sum_{i\in A}\sum_{j\in B}\frac{1}{\mathfrak{R}_{i,\,j}}+O\Big(\frac{d_{N}^{2}}{N^{2}}\Big)+O(d_{N}^{2}m_{\star}^{N})+O\Big(\frac{d_{N}^{3}}{N}\log N\Big).
\]
Sending $N\to\infty$, as $\lim_{N\to\infty}d_{N}\log N=0$ and $d_{N}$
decays subexponentially, we have 
\[
\limsup_{N\to\infty}\frac{N}{d_{N}^{2}}D_{N}(f_{\textup{test}}^{A})\le\frac{1}{|S_{\star}|}\sum_{i\in A}\sum_{j\in B}\frac{1}{\mathfrak{R}_{i,\,j}}.
\]
Therefore, by Theorem \ref{Diri-Thom}, we obtain the desired result. 
\end{proof}

\section{\label{s-LB}Lower Bound for Capacities: General Case}

In this section, we establish the lower bound for the capacities in
the most general setting given in Subsection \ref{ss-sts}. The following
proposition explains the result. The proofs in this section will be
stated concisely. 
\begin{prop}[Lower bound for capacities: General case]
\label{LB} Assume the conditions of Theorem \ref{main}. Suppose
that $\{A,\,B\}$ is a non-trivial partition $\llbracket1,\,\kappa_{\star}\rrbracket$.
Then, the following inequality holds. 
\begin{equation}
\liminf_{N\to\infty}\frac{N}{d_{N}^{2}}\mathrm{Cap}_{N}(\mathcal{E}_{N}^{(2)}(A),\,\mathcal{E}_{N}^{(2)}(B))\ge\frac{1}{|S_{\star}|}\sum_{i\in A}\sum_{j\in B}\frac{1}{\mathfrak{R}_{i,\,j}}.\label{LBeq}
\end{equation}
\end{prop}

We construct a test flow, whose divergence can be handled outside
the one-dimensional tubes.

\subsection{Construction of test flow $\psi_{\textup{test}}^{A}$}

In this subsection, we build the test flow $\psi=\psi_{\textup{test}}^{A}$
on $\mathcal{H}_{N}$.
\begin{itemize}
\item We define, for $(i,\,j)\in I\cap(A\times B)$, $n,\,m\ge1$, $p\in P_{i,\,n,\,j,\,m}$,
$k\ge1$, $N-\ell-k\ge1$, and $\ell\in\llbracket0,\,\frac{N}{2}-1\rrbracket$,
\begin{equation}
\psi_{0}^{A}(\xi_{k,\,\ell}^{x_{i,\,n}y_{p}x_{j,\,m}},\,\xi_{k-1,\,\ell+1}^{x_{i,\,n}y_{p}x_{j,\,m}})=\frac{m_{\star}(y_{p})^{\ell}\big/(\frac{N-k-\ell-1}{r(x_{i,\,n},\,y_{p})}+\frac{k+\ell}{r(x_{j,\,m},\,y_{p})})}{\mathfrak{R}_{i,\,j}\sum_{n=1}^{\mathfrak{n}(i)}\sum_{m=1}^{\mathfrak{n}(j)}\sum_{q=1}^{\kappa_{0}}\frac{(1-m_{\star}(y_{q}))^{-1}}{\frac{N-k-\ell-1}{r(x_{i,\,n},\,y_{q})}+\frac{k+\ell}{r(x_{j,\,m},\,y_{q})}}},\label{testflow}
\end{equation}
\begin{equation}
\psi_{0}^{A}(\xi_{k,\,\ell}^{x_{i,\,n}y_{p}x_{j,\,m}},\,\xi_{k,\,\ell+1}^{x_{i,\,n}y_{p}x_{j,\,m}})=\frac{-m_{\star}(y_{p})^{\ell}\big/(\frac{N-k-\ell-1}{r(x_{i,\,n},\,y_{p})}+\frac{k+\ell}{r(x_{j,\,m},\,y_{p})})}{\mathfrak{R}_{i,j}\sum_{n=1}^{\mathfrak{n}(i)}\sum_{m=1}^{\mathfrak{n}(j)}\sum_{q=1}^{\kappa_{0}}\frac{(1-m_{\star}(y_{q}))^{-1}}{\frac{N-k-\ell-1}{r(x_{i,\,n},\,y_{q})}+\frac{k+\ell}{r(x_{j,\,m},\,y_{q})}}},\label{testflow2}
\end{equation}
and $0$ otherwise.
\item Then, for all $(i,\,j)\in I\cap(A\times B)$, $n,\,m\ge1$, $p\in P_{i,\,n,\,j,\,m}$,
and $k\in\llbracket1,\,N-1\rrbracket$, we define a correction flow
$\phi_{i,\,j,\,p,\,k}^{A}$ as follows. 
\begin{itemize}
\item Suppose that $\frac{N}{2}<k\le N-1$. Then, for $\ell\in\llbracket1,\,\frac{N}{2}\rrbracket$,
\[
\phi_{i,\,j,\,p,\,k}^{A}(\xi_{k-\ell,\,\ell}^{x_{i,\,n}y_{p}x_{j,\,m}},\,\xi_{k-\ell+1,\,\ell-1}^{x_{i,\,n}y_{p}x_{j,\,m}}):=-\sum_{t=\ell}^{\lfloor N/2\rfloor}(\mathrm{div}\,\psi_{0})(\zeta_{k-t,\,t}^{x_{i,\,n}y_{p}x_{j,\,m}}),
\]
and $\phi_{i,\,j,\,p,\,k}^{A}=0$ on all other edges.
\item Suppose that $1\le k\le\frac{N}{2}$. Then, for $\ell\in\llbracket1,\,k-1\rrbracket$,
\[
\phi_{i,\,j,\,p,\,k}^{A}(\xi_{k-\ell,\,\ell}^{x_{i,\,n}y_{p}x_{j,\,m}},\,\xi_{k-\ell+1,\,\ell-1}^{x_{i,\,n}y_{p}x_{j,\,m}}):=-\sum_{t=\ell}^{k-1}(\mathrm{div}\,\psi_{0})(\xi_{k-t,\,t}^{x_{i,\,n}y_{p}x_{j,\,m}}),
\]
and $\phi_{i,\,j,\,p,\,k}^{A}=0$ on all other edges.
\end{itemize}
\item Finally, we define a flow
\end{itemize}
\[
\psi=\psi_{\textup{test}}^{A}:=\psi_{0}^{A}+\sum_{(i,\,j)\in I\cap(A\times B)}\sum_{n=1}^{\mathfrak{n}(i)}\sum_{m=1}^{\mathfrak{m}(j)}\sum_{p\in P_{i,\,n,\,j,\,m}}\sum_{k=1}^{N-1}\phi_{i,\,j,\,p,\,k}^{A}.
\]
Then, observe that $(\mathrm{div}\,\psi)(\xi^{x_{i,\,n}})=0$ for
all $i\in\llbracket1,\,\kappa_{\star}\rrbracket$ and $n\in\llbracket1,\,\mathfrak{n}(i)\rrbracket$.
Moreover, it holds that $(\mathrm{div}\,\psi)(\eta)=0$ for all $\eta$
in 
\[
\mathcal{A}_{N}^{x_{i,\,n}y_{p}x_{j,\,m}}\setminus(\mathcal{A}_{N}^{x_{i,\,n}y_{p}}\cup\mathcal{A}_{N}^{x_{j,\,m}y_{p}})\text{ for }i\ne j\text{ and }n,\,m,\,p\ge1.
\]

\subsection{Flow norm of $\psi_{\textup{test}}^{A}$}

In this subsection, we calculate the flow norm of the test flow $\psi$. 
\begin{lem}
\label{lem5}Suppose that $I\cap(A\times B)\ne\emptyset$. Then, under
the conditions of Theorem \ref{main}, 
\[
\|\psi\|_{N}^{2}\le(1+o(1))\frac{|S_{\star}|N}{d_{N}^{2}}\Big(\sum_{i\in A}\sum_{j\in B}\frac{1}{\mathfrak{R}_{i,\,j}}\Big).
\]
\end{lem}

\begin{proof}
The proof is almost identical to that of Lemma \ref{lem5spl}; therefore,
we omit it. 
\end{proof}

\subsection{\label{ss-remflow}Remaining terms}

We estimate the remaining terms on the right-hand side of \eqref{genThomeq}
with respect to $\psi$. Lemma \ref{eqpotest} is employed once more. 
\begin{lem}
\label{lem6}Suppose that $I\cap(A\times B)\ne\emptyset$. Then, under
the conditions of Theorem \ref{main}, 
\[
\sum_{\eta\in\mathcal{H}_{N}\setminus\mathcal{E}_{N}^{\star}}h_{\mathcal{E}_{N}^{(2)}(A),\,\mathcal{E}_{N}^{(2)}(B)}(\eta)(\mathrm{div}\,\psi)(\eta)=(1+o(1))\Big(\sum_{i\in A}\sum_{j\in B}\frac{1}{\mathfrak{R}_{i,\,j}}\Big).
\]
\end{lem}

\begin{proof}
We omit the proof due to its similarity to Lemma \ref{lem6spl}. 
\end{proof}

\subsection{Proof of Proposition \ref{LB}}

We are now ready to prove Proposition \ref{LB}. 
\begin{proof}[Proof of Proposition \ref{LB}]
 There remains nothing to prove if $I\cap(A\times B)=\emptyset$,
as then the right-hand side of \eqref{LBeq} equals $0$. Thus, we
may assume $I\cap(A\times B)\ne\emptyset$. Then, by Lemmas \ref{lem5}
and \ref{lem6}, we have 
\[
\frac{1}{\|\psi_{\textup{test}}^{A}\|_{N}^{2}}\Big[\sum_{\eta\notin\mathcal{E}_{N}^{\star}}h_{\mathcal{E}_{N}^{(2)}(A),\,\mathcal{E}_{N}^{(2)}(B)}(\eta)(\mathrm{div}\,\psi_{\textup{test}}^{A})(\eta)\Big]^{2}\ge(1+o(1))\frac{d_{N}^{2}}{|S_{\star}|N}\sum_{i\in A}\sum_{j\in B}\frac{1}{\mathfrak{R}_{i,\,j}}.
\]
Therefore, we deduce from Theorem \ref{genThom} that 
\[
\mathrm{Cap}_{N}(\mathcal{E}_{N}^{(2)}(A),\,\mathcal{E}_{N}^{(2)}(B))\ge(1+o(1))\frac{d_{N}^{2}}{|S_{\star}|N}\sum_{i\in A}\sum_{j\in B}\frac{1}{\mathfrak{R}_{i,\,j}}.
\]
This concludes the proof of Proposition \ref{LB}. 
\end{proof}

\section{\label{s-H1}Proof of Condition \eqref{H1}}

In this section, we prove the condition \eqref{H1} formulated in
Proposition \ref{B-L conv}. 
\begin{prop}
\label{H1prop}The condition \eqref{H1} holds for every $i\in\llbracket1,\,\kappa_{\star}\rrbracket$. 
\end{prop}

\begin{proof}
The numerator in \eqref{H1} can be dealt with using Proposition \ref{UBspl};
\begin{equation}
\mathrm{Cap}_{N}(\mathcal{E}_{N}^{(2)}(i),\,\mathcal{E}_{N}^{\star}\setminus\mathcal{E}_{N}^{(2)}(i))=\frac{d_{N}^{2}}{N}\cdot O(1).\label{H1propeq}
\end{equation}
For the denominator in \eqref{H1}, fix $\eta,\,\zeta\in\mathcal{E}_{N}^{(2)}(i)$
and write $\eta=\xi^{x}$ and $\zeta=\xi^{y}$ with $x,\,y\in S_{i}^{(2)}$.
By the definition of $S_{i}^{(2)}$, there exist $x=x_{0},\,x_{1},\,\dots,\,x_{t}=y$
in $S_{i}^{(2)}$ so that $t\le|S_{i}^{(2)}|$ and
\[
r(x_{n},\,x_{n+1})=r(x_{n+1},\,x_{n})>0
\]
for all $0\le n\le t-1$. Take any $F:\mathcal{H}_{N}\to\mathbb{R}$
with $F(\eta)=1\text{ and }F(\zeta)=0$. Recalling Definition \ref{tubedef},
and by reversibility, we calculate $D_{N}(F)$ by 
\[
\begin{aligned} & \frac{1}{2}\sum_{\eta\in\mathcal{H}_{N}}\sum_{a,\,b\in S}\mu_{N}(\eta)\eta_{a}(d_{N}+\eta_{b})r(a,\,b)\{F(\sigma^{a,\,b}\eta)-F(\eta)\}^{2}\\
 & \ge\sum_{n=0}^{t-1}\sum_{j=1}^{N}\mu_{N}(\xi_{j}^{x_{n}x_{n+1}})j(d_{N}+N-j)r(x_{n},\,x_{n+1})\{F(\xi_{j}^{x_{n},x_{n+1}})-F(\xi_{j-1}^{x_{n},x_{n+1}})\}^{2}.
\end{aligned}
\]
By Proposition \ref{mun} and \eqref{munprop}, the last line equals
\[
(1+o(1))\sum_{n=0}^{t-1}\frac{Nd_{N}}{|S_{\star}|}r(x_{n},\,x_{n+1})\sum_{j=1}^{N}\{F(\xi_{j}^{x_{n}x_{n+1}})-F(\xi_{j-1}^{x_{n}x_{n+1}})\}^{2}.
\]
By the Cauchy--Schwarz inequality on $1\le j\le N$, the above is
bounded from below by 
\[
(1+o(1))\sum_{n=0}^{t-1}\frac{d_{N}}{|S_{\star}|}r(x_{n},\,x_{n+1})\{F(\xi^{x_{n}})-F(\xi^{x_{n+1}})\}^{2}.
\]
Using the Cauchy--Schwarz inequality once more on $0\le n\le t-1$,
we obtain the following lower bound for $D_{N}(F)$: 
\[
(1+o(1))\frac{d_{N}}{|S_{\star}|}\frac{\{F(\eta)-F(\zeta)\}^{2}}{\sum_{n=0}^{t-1}\frac{1}{r(x_{n},\,x_{n+1})}}\ge(1+o(1))\frac{d_{N}}{|S_{\star}|}\frac{\min\{r(u,\,v)>0:u,\,v\in S_{i}^{(2)}\}}{|S_{i}^{(2)}|}.
\]
As $F$ was arbitrary, by the Dirichlet principle given in Theorem
\ref{Diri-Thom}-(1), we have 
\begin{equation}
\mathrm{Cap}_{N}(\{\eta\},\,\{\zeta\})\ge(1+o(1))\frac{d_{N}}{|S_{\star}|}\frac{\min\{r(u,\,v)>0:u,\,v\in S_{i}^{(2)}\}}{|S_{i}^{(2)}|}.\label{H1propeq2}
\end{equation}
Therefore, by \eqref{H1propeq} and \eqref{H1propeq2}, we obtain
\[
\limsup_{N\to\infty}\frac{\mathrm{Cap}_{N}(\mathcal{E}_{N}^{(2)}(i),\,\mathcal{E}_{N}^{\star}\setminus\mathcal{E}_{N}^{(2)}(i))}{\inf_{\eta,\,\zeta\in\mathcal{E}_{N}^{(2)}(i)}\mathrm{Cap}_{N}(\{\eta\},\,\{\zeta\})}\le C\limsup_{N\to\infty}\frac{d_{N}^{2}/N}{d_{N}}=C\limsup_{N\to\infty}\frac{d_{N}}{N}=0.
\]
The last formula concludes the proof of Proposition \ref{H1prop}.
\end{proof}

\section{\label{s-PMT}Proof of the Main Theorem}

Now, we are in position to prove the main theorem given in Theorem
\ref{main}. First, we provide sharp asymptotics for the transition
rate of the trace process $\eta_{N}^{\star}(\cdot)$. 
\begin{prop}[Transition rates of the trace process]
\label{jr} Suppose that $d_{N}$ decays subexponentially. Then,
for $i,\,j\in\llbracket1,\,\kappa_{\star}\rrbracket$, 
\[
\lim_{N\to\infty}\frac{N}{d_{N}^{2}}\mathbf{r}_{N}^{\star}(i,\,j)=\frac{1}{|S_{i}^{(2)}|\mathfrak{R}_{i,\,j}}.
\]
\end{prop}

\begin{proof}
By Proposition \ref{conden}, $\lim_{N\to\infty}\mu_{N}(\mathcal{E}_{N}^{(2)}(i))=|S_{i}^{(2)}|/|S_{\star}|$
for each $i\in\llbracket1,\,\kappa_{\star}\rrbracket$. Hence, by
Propositions \ref{UBspl}, \ref{LBspl}, and \eqref{B-L lem}, we
have
\[
\begin{aligned}\frac{|S_{i}^{(2)}|}{|S_{\star}|}\mathbf{r}_{N}^{\star}(i,\,j) & =\frac{d_{N}^{2}}{2|S_{\star}|N}\Big[\sum_{k:\,k\ne i}\frac{1}{\mathfrak{R}_{i,\,k}}+\sum_{k:\,k\ne j}\frac{1}{\mathfrak{R}_{j,\,k}}-\sum_{k:\,k\ne i,\,j}\Big(\frac{1}{\mathfrak{R}_{i,\,k}}+\frac{1}{\mathfrak{R}_{j,\,k}}\Big)+o(1)\Big]\\
 & =\frac{d_{N}^{2}}{2|S_{\star}|N}\Big[\frac{2}{\mathfrak{R}_{i,\,j}}+o(1)\Big]=\frac{d_{N}^{2}}{|S_{\star}|N}\Big[\frac{1}{\mathfrak{R}_{i,\,j}}+o(1)\Big].
\end{aligned}
\]
Multiplying $(|S_{\star}|N)/(|S_{i}^{(2)}|d_{N}^{2})$ on both sides,
we obtain the desired result. 
\end{proof}
Finally, we provide the proof of Theorem \ref{main}.
\begin{proof}[Proof of Theorem \ref{main}]
 By Propositions \ref{H1prop} and \ref{jr}, the conditions \eqref{H0}
and \eqref{H1} are verified for 
\[
a(i,\,j)=\frac{1}{|S_{i}^{(2)}|\mathfrak{R}_{i,\,j}}\text{ for }i,\,j\in\llbracket1,\,\kappa_{\star}\rrbracket\text{ and }\theta_{N}=\theta_{N,\,2}=\frac{d_{N}^{2}}{N}.
\]
Therefore, Proposition \ref{B-L conv} establishes the thermalization
result stated in (1) and the convergence result stated in (2).

For the last statement in (3), we first show that 
\begin{equation}
\lim_{N\to\infty}\sup_{i\in\llbracket1,\,\kappa_{\star}\rrbracket,\,n\in\llbracket1,\,\mathfrak{n}(i)\rrbracket}\mathbb{E}_{\xi^{x_{i,\,n}}}\Big[\int_{0}^{t}\mathbbm{1}\{\eta_{N}(\theta_{N,\,2}s)\notin\mathcal{E}_{N}^{\star}\}ds\Big]=0.\label{mainpfeq}
\end{equation}
To this end, fix $i$ and $n$. Note that
\begin{equation}
\mathbb{P}_{\xi^{x_{i,\,n}}}\big[\eta_{N}(\theta_{N,\,2}s)\notin\mathcal{E}_{N}^{\star}\big]\le\frac{1}{\mu_{N}(\mathcal{E}_{N}^{x_{i,\,n}})}\mathbb{P}_{\mu_{N}}\big[\eta_{N}(\theta_{N,\,2}s)\notin\mathcal{E}_{N}^{\star}\big]=\frac{\mu_{N}(\mathcal{H}_{N}\setminus\mathcal{E}_{N}^{\star})}{\mu_{N}(\mathcal{E}_{N}^{x_{i,\,n}})}.\label{mainpfeq2}
\end{equation}
Here, $\mathbb{P}_{\mu_{N}}$ is the law of the process whose initial
distribution is $\mu_{N}$. The identity holds, as $\mu_{N}$ is the
invariant distribution. Therefore,
\[
\begin{aligned}\mathbb{E}_{\xi^{x_{i,\,n}}}\Big[\int_{0}^{t}\mathbbm{1}\{\eta_{N}(\theta_{N,\,2}s)\notin\mathcal{E}_{N}^{\star}\}ds\Big] & =\int_{0}^{t}\mathbb{P}_{\xi^{x_{i,\,n}}}\big[\eta_{N}(\theta_{N,\,2}s)\notin\mathcal{E}_{N}^{\star}\big]ds\\
 & \le\int_{0}^{t}\frac{\mu_{N}(\mathcal{H}_{N}\setminus\mathcal{E}_{N}^{\star})}{\mu_{N}(\mathcal{E}_{N}^{x_{i,\,n}})}ds=t\cdot\frac{\mu_{N}(\mathcal{H}_{N}\setminus\mathcal{E}_{N}^{\star})}{\mu_{N}(\mathcal{E}_{N}^{x_{i,\,n}})},
\end{aligned}
\]
which vanishes uniformly in the limit $N\to\infty$ by Proposition
\ref{conden}. This proves \eqref{mainpfeq}.

It remains to show that 
\begin{equation}
\lim_{N\to\infty}\sup_{i\in\llbracket1,\,\kappa_{\star}\rrbracket,\,n\in\llbracket1,\,\mathfrak{n}(i)\rrbracket}\mathbb{E}_{\xi^{x_{i,\,n}}}\Big[\int_{0}^{t}\mathbbm{1}\{\eta_{N}(\theta_{N,\,2}s)\in\mathcal{E}_{N}^{\star}\setminus\mathcal{E}_{N}(S_{\hat{i}}^{(3)})\}ds\Big]=0.\label{mainpfeq3}
\end{equation}
We apply Proposition \ref{ConvFDD}. Because the first two conditions
are already proven, it suffices to prove \eqref{ConvFDDeq}. This
is clear from \eqref{mainpfeq2}. Hence, we have the convergence of
finite-dimensional marginal distributions. Therefore, for each pair
$(i,\,n)$ and $s\in[0,\,t]$, 
\[
\lim_{N\to\infty}\mathbb{P}_{\xi^{x_{i,\,n}}}\big[\eta_{N}(\theta_{N,\,2}s)\in\mathcal{E}_{N}^{\star}\setminus\mathcal{E}_{N}(S_{\hat{i}}^{(3)})\big]=\mathbf{P}_{i}\big[X_{\textup{second}}(s)\in S_{\star}\setminus S_{\hat{i}}^{(3)}\big]=0.
\]
The last equality holds, as starting at $i$, $X_{\textup{second}}(\cdot)$
never visits $S_{\star}\setminus S_{\hat{i}}^{(3)}$ by \eqref{maineq3}.
Because $S_{\star}$ is finite, we have \eqref{mainpfeq3}. Finally,
\eqref{mainpfeq} and \eqref{mainpfeq3} conclude the proof of Theorem
\ref{main}. 
\end{proof}
\begin{acknowledgement*}
S. Kim received support from the National Research Foundation of Korea
(NRF) grant funded by the Korea government (MSIT) (No. 2018R1C1B6006896)
and NRF-2019-Fostering Core Leaders of the Future Basic Science Program/Global
Ph.D. Fellowship Program.
\end{acknowledgement*}

\end{document}